\documentclass{amsart}

\usepackage{gensymb,epic,eepic,float,fullpage}
\usepackage{longtable}

\usepackage{amsmath,amsfonts,amssymb,amsthm}
\usepackage{mathtools}
\usepackage[colorinlistoftodos]{todonotes}
\usepackage{cancel}
\usepackage{stmaryrd}
\usepackage{url}
\usepackage{mathrsfs}
\usepackage{bbm}
\usepackage{bm}
\usepackage{tikz}
\usepackage[outline]{contour}
\usetikzlibrary{graphs,patterns,decorations.markings,arrows.meta,matrix}
\usetikzlibrary{calc,decorations.pathmorphing,decorations.pathreplacing,shapes}
\usepackage{graphicx}

\usepackage{hyperref}

\DeclareFontFamily{U} {MnSymbolA}{}
\DeclareFontShape{U}{MnSymbolA}{m}{n}{
  <-6> MnSymbolA5
  <6-7> MnSymbolA6
  <7-8> MnSymbolA7
  <8-9> MnSymbolA8
  <9-10> MnSymbolA9
  <10-12> MnSymbolA10
  <12-> MnSymbolA12}{}
\DeclareFontShape{U}{MnSymbolA}{b}{n}{
  <-6> MnSymbolA-Bold5
  <6-7> MnSymbolA-Bold6
  <7-8> MnSymbolA-Bold7
  <8-9> MnSymbolA-Bold8
  <9-10> MnSymbolA-Bold9
  <10-12> MnSymbolA-Bold10
  <12-> MnSymbolA-Bold12}{}

\DeclareSymbolFont{MnSyA} {U} {MnSymbolA}{m}{n}

\DeclareMathSymbol{\mnrightarrow}{\mathrel}{MnSyA}{0}
\DeclareMathSymbol{\mndownarrow}{\mathrel}{MnSyA}{3}
\DeclareMathSymbol{\mnleftarrow}{\mathrel}{MnSyA}{2}
\DeclareMathSymbol{\mnuparrow}{\mathrel}{MnSyA}{1}
\DeclareMathSymbol{\mnnearrow}{\mathrel}{MnSyA}{4}
\DeclareMathSymbol{\mnnwarrow}{\mathrel}{MnSyA}{5}
\DeclareMathSymbol{\mnswarrow}{\mathrel}{MnSyA}{6}
\DeclareMathSymbol{\mnsearrow}{\mathrel}{MnSyA}{7}
\numberwithin{equation}{section}

\newtheorem{prop}{Proposition}[section]
\newtheorem{theo}[prop]{Theorem}
\newtheorem{lem}{Lemma}[prop]
\newtheorem{cor}[prop]{Corollary}
\theoremstyle{definition}
\newtheorem{rem}[prop]{Remark}
\newtheorem{defin}[prop]{Definition}

\newcommand\restr[2]{{% we make the whole thing an ordinary symbol
  \left.\kern-\nulldelimiterspace % automatically resize the bar with \right
  #1 % the function
  \vphantom{\big|} % pretend it's a little taller at normal size
  \right|_{#2} % this is the delimitert
 }}
 
 \colorlet{lgray}{white!85!black}
\colorlet{lred}{white!85!red}
\colorlet{lgreen}{white!60!green}
\colorlet{dgreen}{black!30!green}
\colorlet{lpurple}{white!60!purple}
\colorlet{lblue}{white!60!blue}
\definecolor{green}{rgb}{0.1,0.8,0.1}
\definecolor{yellow}{rgb}{1.0,0.85,0.25}
\definecolor{purple}{rgb}{1.0, 0, 1.0}
\definecolor{blue}{rgb}{0, 0, 1.0}

\tikzstyle{unfused}=[lgray, line width=1.5pt, ->]
\tikzstyle{fused}=[lgray, line width=4pt, ->]
\tikzstyle{dual}=[black, line width=1pt, dashed]
\tikzstyle{lightdual}=[black, line width=0.5pt, dashed]
\tikzstyle{cut}=[black, line width=1.0pt]
 
 \contourlength{1pt}
 
 \renewcommand{\tikz}[2]{
\begin{tikzpicture}[scale=#1,baseline=(current bounding box.center),>=stealth]
#2
\end{tikzpicture}}

\newcommand{\tikzbase}[3]{
\begin{tikzpicture}[scale=#1,baseline={([yshift=#2]current bounding box.center)},>=stealth]
#3
\end{tikzpicture}}
 
\renewcommand{\vert}[4]{
\tikz{0.5}{
	\draw[lgray,line width=1pt,->] (-1,0) -- (1,0);
	\draw[lgray,line width=1pt,->] (0,-1) -- (0,1);
	\node[left] at (-1,0) {\tiny $#2$};\node[right] at (1,0) {\tiny $#4$};
	\node[below] at (0,-1) {\tiny $#1$};\node[above] at (0,1) {\tiny $#3$};
}}

\def \be{\begin{equation*}}
\def \ee{\end{equation*}}
\def\({\left(}
\def\){\right)}
\def\[{\left[}
\def\]{\right]}
\DeclareMathOperator{\res}{res}
\DeclareMathOperator{\Comp}{Cmp}
\DeclareMathOperator{\inv}{inv}
\DeclareMathOperator{\tinv}{\widetilde{inv}}

\DeclareMathOperator{\Row}{Row}
\DeclareMathOperator{\Col}{Col}

\def \bA{\bm A}
\def \bB{\bm B}
\def \bC{\bm C}
\def \bD{\bm D}
\def \bI{\bm I}
\def \bJ{\bm J}
\def \bK{\bm K}

\def \bP{\bm P}

\def \NN{{\sf N}}
\def \MM{{\sf M}}
\def \LL{{\sf L}}

\def \Z{\mathbb{Z}}
\def \dZ{\mathbb{Z}+\frac12}
\def \ddZ{\(\dZ\)^2}

\def \A{\mathcal{A}}
\def \B{\mathcal{B}}
\def \C{\mathcal{C}}
\def \M{\mathcal{M}}
\def \R{\mathcal R}

\def \z{\zeta}
\def \p{\mathfrak p}
\def \q{\mathfrak q}
\def \m{\mathfrak m}

\def \bu{\mathbf{u}}
\def \bz{\mathbf{z}}
\def \bx{\mathbf{x}}
\def \by{\mathbf{y}}
\def \bw{\mathbf{w}}
\def \bi{\bm i}
\def \bj{\bm j}

\def \bp{\mathbf p}
\def \bc{\mathbf c}

\def \1{\mathbbm 1}
\def \x{\mathbf x}
\def \y{\mathbf y}
\def \e{\mathbf e}
\def \c{\mathbf{c}}
\def \i{\mathbf{i}}

\def \H{\mathcal{H}}

\def \E{\mathbb E}

\def \C{\mathcal C}

\def \O{\mathcal O}

\title[Observables and local relation]
{Observables of stochastic colored vertex models and local relation}

\author{Alexey Bufetov}

\address[Alexey Bufetov]{Hausdorff Center for Mathematics \& Institute for Applied Mathematics, University of Bonn, Germany. E-mail: alexey.bufetov@gmail.com}

\author{Sergei Korotkikh}

\address[Sergei Korotkikh]{Massachusetts Institute of Technology, Cambridge, U.S.A, E-mail:shortkih@gmail.com}

\begin{document}

\begin{abstract}
We study the stochastic colored six vertex (SC6V) model and its fusion. Our main result is an integral expression for natural observables of this model --- joint q-moments of height functions. This generalises a recent result of Borodin-Wheeler.  The key technical ingredient is a new relation of height functions of SC6V model in neighboring points. This relation is of independent interest; we refer to it as a local relation. As applications, we give a new proof of certain symmetries of height functions of SC6V model recently established by Borodin-Gorin-Wheeler and Galashin, and new formulas for joint moments of delayed partition functions of Beta polymer. 
\end{abstract}

\maketitle

\tableofcontents

\section{Introduction}

\subsection{Summary} In the last twenty years there was a lot of progress in understanding large scale, long time asymptotics of nonequilibrium stochastic particle systems belonging to the Kardar-Parisi-Zhang (KPZ) universality class. Much of this progress has been obtained by discovering exact formulas for certain \textit{observables} of these systems, and subsequently analysing their asymptotic behavior. In particular, in pioneering works of Tracy-Widom integral expressions for q-moments of the height function of asymmetric simple exclusion process (ASEP) were established (\cite{TW08a}, \cite{TW08b}) and then used (\cite{TW09}) for studying the asymptotic fluctuations of the height function. A prominent role in the KPZ universality class is played by the stochastic six vertex (S6V) model, since it can be degenerated into a number of systems in the class, including ASEP, q-TASEP, and some polymer models. The S6V model was first introduced by Gwa-Spohn in \cite{GS92} and its asymptotics were analyzed via exact formulas by Borodin-Corwin-Gorin in \cite{BCG16}. Numerous related works on obtaining q-moments of the height functions of various probabilistic systems were published, see e.g. Borodin-Corwin-Petrov-Sasamoto \cite{BCPS15b},\cite{BCPS15a}, Corwin-Petrov \cite{CP16}, Bufetov-Matveev \cite{BM18}, and references therein.

More recently, stochastic \textit{colored} vertex models were introduced and studied by Kuniba-Mangazeev-Maruyama-Okado \cite{KMMO16}, see also Borodin-Wheeler \cite{BW18}, Bosnjak-Mangazeev \cite{BM16}. On the language of interacting particle systems, they correspond to multi-species (or multi-type) particle systems. A natural goal is to find reasonable formulas for observables of such systems.

The main motivation for our current work is a recent paper of Borodin-Wheeler \cite{BW20}, in which integral expressions for certain observables of stochastic colored six vertex model (SC6V) were obtained. Our main result is Theorem \ref{mainTheorem} (and its fused version Theorem \ref{fusedMainTheorem}) which generalises the formula of Borodin-Wheeler in the following directions:
\begin{itemize}
\item Observables in our formula have a simple and standard form: They are q-moments of height functions. Observables in the Borodin-Wheeler formula are fairly complicated linear combinations of such q-moments. We demonstrate in Section \ref{sec:shiftedQmoments} how to derive the Borodin-Wheeler formula from ours.

\item Our formula is applicable to more general domains.

\item Our formula is applicable in the case of inhomogeneous spectral parameters of the model. 
\end{itemize}

Our method for obtaining the formula for observables is completely new and is of independent interest. It is based on a certain non-trivial relation between q-moments of height function in neighboring points which we prove in Proposition \ref{pLocalRelation}. We refer to it as a \textit{local relation}. This relation is a significant generalisation of the so called four point relation which played a key role in \cite{BG18}, see Remark \ref{rem:historyLocRel} for more detail. The main idea behind the local relation is that it allows to track the change in q-moments of height functions after one more vertex is added into the domain in which the SC6V model is considered. Crucially, the q-moments in the updated domain depend only on q-moments before the update. The existence of such a relation is non-trivial and emphasises a special role played by q-moments of height functions for ASEP/ S6V-type models. With the use of the local relation we are able to prove the formulas for q-moments of height functions via the induction on the size of the domain.

We give two applications of our formula for observables. First, we give an independent proof (see Theorem \ref{shiftInv} and its proof in Section \ref{shiftSec}) of some of the recent results of Galashin \cite{Gal20}, who established a conjecture from \cite{BGW19} about certain symmetries in joint distributions of height functions of SC6V model (this result can be degenerated into symmetries of a number of models, including directed polymers and the Airy sheet, see also \cite{D20} for related results). We are able to do this because q-moments of height functions completely determine their distribution, so it is sufficient to check the equivalence of formulas for the q-moments of these functions (this is not an easy check, though). Second, we establish similar formulas for observables for \textit{fused} vertex models (Theorems \ref{fusedMainTheorem} and \ref{th:qHahnFormulae}) and degenerate them to formulas for moments of delayed partition functions in Beta polymer (Theorem \ref{th:BetaPol}).

There is a significant difference between the approach of Borodin-Wheeler (\cite{BW18}, \cite{BW20}) to the SC6V model and the approach used recently in \cite{Buf20}, \cite{Gal20}, and in the current paper. The approach of Borodin-Wheeler is based on a family of nonsymmetric functions defined as partition functions of certain configurations of the SC6V model, and a subsequent study of the properties of these functions. The second approach seems to be more direct and is based on an interpretation of the SC6V model as a random walk on the Hecke algebra. One advantage of the second approach is that it immediately implies the so called color-position symmetry in the SC6V model (see \cite{Buf20}, \cite{Gal20}, and \cite{K20}), which is useful for 
asymptotic applications (see e.g. \cite{AAV08}, \cite{BB19}). In the current paper we also use the second approach, which arguably leads to more direct proofs. 

\subsection{Statement of result} Let us formulate a version of our main result. We will present an integral expression for q-moments of height functions of the higher spin version of the SC6V model.

The higher spin SC6V model can be defined as a random collection of up-right paths in the quadrant $\Z_{\ge 1} \times \Z_{\ge 1}$; each of paths has color (labeled by an integer) associated with it. Vertical edges of the quadrant can contain arbitrary number of paths, while no horizontal edge can contain more than one path. At the boundary of the quadrant, we set the following conditions. Let us fix a sequence of integers $0=l_0 \le l_1 \le l_2 \le \dots$. Incoming horizontal edges at rows $l_{c-1}+1, l_{c-1}+2, \dots, l_c$ have color $c$, for all $c=1,2,\dots$. No paths enter the quadrant through the vertical edges at the bottom of the quadrant.

\begin{figure}
	\centering
	\begin{tikzpicture}[scale=1.3,baseline={([yshift=-15]current bounding box.center)},>={Stealth [scale=1.2]}]]
	\foreach\x in {1,...,3}{
		\draw[line width=0.7pt,dotted] (\x,0) -- (\x,5.9);
		\node[below, text=black] at (\x,-0.2) {\x};
	}
	\foreach \y in {1,...,5}{
		\draw[line width=0.7pt,dotted] (0,\y) -- (3.9,\y);
		\node[left, text=black] at (-0.2,\y) {\y};
	}
	
	\draw[white, line width = 2] (3,3) -- (3,4);
	\draw[white, line width = 2] (1,5) -- (1,6);
	
	\draw[red, line width = 1.5pt] (0,1) -- (3,1) -- (3,3);
	\node[above, text=red] at (0.5,1) {1};
	\draw[red, line width = 1.5pt] (2.93,3) -- (2.93,4);
	\draw[red, line width = 1.5pt, ->] (3,4) -- (3,5) -- (4,5);
	
	\draw[red, line width = 1.5pt, ->] (0,2) -- (2,2) -- (2,6);
	\node[above, text=red] at (0.5,2) {1};

	\draw[blue, line width = 1.5pt] (0,3) -- (3,3);
	\node[above, text=blue] at (0.5,3) {2};
	\draw[blue, line width = 1.5pt] (3.07,3) -- (3.07,4);
	\draw[blue, line width = 1.5pt, ->] (3,4) -- (4,4);

	\draw[blue, line width = 1.5pt] (0,4) -- (1,4) -- (1,5);
	\node[above, text=blue] at (0.5,4) {2};
	\draw[blue, line width = 1.5pt, ->] (0.93,5) -- (0.93,6);	
	
	\draw[green, line width = 1.5pt] (0,5) -- (1,5);
	\node[above, text=green] at (0.5,5) {3};
	\draw[green, line width = 1.5pt, ->] (1.07,5) -- (1.07,6);	
	
\end{tikzpicture}
	\caption{A possible configuration of the higher spin SC6V model (best viewed in color). For such a configuration, one has $h_{>0}^{(3/2,1/2)}=0$, $h_{>1}^{(5/2,7/2)}=1$,               	$h_{>0}^{(3/2,9/2)}=3$. }
	\label{Fig-intro}
\end{figure}

The higher spin SC6V model involves the main quantization parameter $q\in [0;1]$ and three sequences of real valued parameters. A sequence $(u_1,u_2,u_3 \dots)$ is referred to as rapidities of rows, a sequence $(y_1,y_2,y_3 \dots)$ is referred to as rapidities of columns, and a sequence $(s_1,s_2,s_3 \dots)$ is referred to as spin parameters (they are also associated with columns). The random colored up-right paths propagate according to the following Markov recipe, which is applied to points $(\mathbf{x},\mathbf{y}) \in \Z_{\ge 1} \times \Z_{\ge 1}$ with $\mathbf{x}+\mathbf{y}=2,3,4, \dots$ step by step. At each step, the filling of the left and bottom edges in vertices along the diagonal $\mathbf{x}+\mathbf{y}=const$ is already determined either by boundary conditions or by previous steps. Let us encode this data in the following way. At the left horizontal edge, we can have either a path of color $i \ge 1$, or no path, which we encode by $0$. At the bottom vertical edge, we have a collection of colored paths which we encode by a vector $\bI = (I_1, I_2, I_3,\dots)$, where $I_k$ is the number of paths of color $k$ present in the edge. 
The filling of up and right edges is defined with the use of \textit{vertex weights}. They are given by the table \eqref{LweightsIntro}, where $\mathbf{e}^i$ is a standard basis vector with 1 as its $i$th coordinate and all other its coordinates are equal to 0, $I_{[a;b]} := \sum_{j=a}^b I_j$, and all unlisted configurations have weight 0.   
The vertex weights are used as probabilities that paths incoming from left and below will go straight or will make a turn in a vertex (indeed, note the sum of all vertex weights with any fixed left and bottom filling is equal to 1), eventually forming a vertex of one of the types listed in the table \eqref{LweightsIntro}. At vertex $(\mathbf{x},\mathbf{y})$ we use the parameters $u=u_{\mathbf{x}}$, $y=y_{\mathbf{y}}$, and $s=s_{\mathbf{y}}$ when using the vertex weights from the table. The random choices at different vertices are jointly independent. 

\begin{align}
\label{LweightsIntro}
\begin{tabular}{|c|c|c|}
\hline
\quad
\tikz{0.7}{
	\draw[lgray,line width=1.5pt,->] (-1,0) -- (1,0);
	\draw[lgray,line width=4pt,->] (0,-1) -- (0,1);
	\node[left] at (-1,0) {\tiny $0$};\node[right] at (1,0) {\tiny $0$};
	\node[below] at (0,-1) {\tiny $\bI$};\node[above] at (0,1) {\tiny $\bI$};
}
\quad
&
\quad
\tikz{0.7}{
	\draw[lgray,line width=1.5pt,->] (-1,0) -- (1,0);
	\draw[lgray,line width=4pt,->] (0,-1) -- (0,1);
	\node[left] at (-1,0) {\tiny $i$};\node[right] at (1,0) {\tiny $i$};
	\node[below] at (0,-1) {\tiny $\bI$};\node[above] at (0,1) {\tiny $\bI$};
}
\quad
&
\quad
\tikz{0.7}{
	\draw[lgray,line width=1.5pt,->] (-1,0) -- (1,0);
	\draw[lgray,line width=4pt,->] (0,-1) -- (0,1);
	\node[left] at (-1,0) {\tiny $0$};\node[right] at (1,0) {\tiny $i$};
	\node[below] at (0,-1) {\tiny $\bI$};\node[above] at (0,1) {\tiny $\bI-\e^i$};
}
\quad
\\[1.3cm]
\quad
$\dfrac{1-s u y^{-1} q^{I_{[1;n]}}}{1-s u y^{-1}}$
\quad
&
\quad
$\dfrac{(s^2q^{I_i}-s u y^{-1} ) q^{I_{[i+1;n]}}}{1-s u y^{-1}}$
\quad
&
\quad
$\dfrac{s u y^{-1} (q^{I_i}-1) q^{I_{[i+1;n]}}}{1-su y^{-1}}$
\quad
\\[0.7cm]
\hline
\quad
\tikz{0.7}{
	\draw[lgray,line width=1.5pt,->] (-1,0) -- (1,0);
	\draw[lgray,line width=4pt,->] (0,-1) -- (0,1);
	\node[left] at (-1,0) {\tiny $i$};\node[right] at (1,0) {\tiny $0$};
	\node[below] at (0,-1) {\tiny $\bI$};\node[above] at (0,1) {\tiny $\bI+\e^i$};
}
\quad
&
\quad
\tikz{0.7}{
	\draw[lgray,line width=1.5pt,->] (-1,0) -- (1,0);
	\draw[lgray,line width=4pt,->] (0,-1) -- (0,1);
	\node[left] at (-1,0) {\tiny $i$};\node[right] at (1,0) {\tiny $j$};
	\node[below] at (0,-1) {\tiny $\bI$};\node[above] at (0,1)
	{\tiny $\bI+\e^i-\e^j$};
}
\quad
&
\quad
\tikz{0.7}{
	\draw[lgray,line width=1.5pt,->] (-1,0) -- (1,0);
	\draw[lgray,line width=4pt,->] (0,-1) -- (0,1);
	\node[left] at (-1,0) {\tiny $j$};\node[right] at (1,0) {\tiny $i$};
	\node[below] at (0,-1) {\tiny $\bI$};\node[above] at (0,1) {\tiny $\bI+\e^j-\e^i$};
}
\quad
\\[1.3cm]
\quad
$\dfrac{1-s^2 q^{I_{[1;n]}}}{1-su y^{-1}}$
\quad
&
\quad
$\dfrac{s u y^{-1} (q^{I_j}-1) q^{I_{[j+1;n]}}}{1-su y^{-1}}$
\quad
&
\quad
$\dfrac{s^2(q^{I_i}-1)q^{I_{[i+1;n]}}}{1-su y^{-1}}$
\quad
\\[0.7cm]
\hline
\end{tabular}
\end{align}

This step-by-step procedure produces a random collection of up-right colored paths in $\Z_{\ge 1} \times \Z_{\ge 1}$. A quantitative way to describe it uses \textit{height functions}. By definition, for $(\alpha, \beta) \in \left( \Z_{\ge 0} + \frac{1}{2} \right) \times \left( \Z_{\ge 0}  + \frac{1}{2} \right)$ and $c \in \Z_{\ge 0}$ the height function $h_{>c}^{(\alpha,\beta)}$ is the total number of paths of color $>c$ which pass directly below the point $(\alpha,\beta)$. In other words, $h_{>c}^{(\alpha,\beta)}$ is the total number of edges connecting $(\alpha-\frac12, j)$ and $(\alpha+\frac12, j)$, $j=1,2,\dots,\beta-\frac12$, which contain a path of color $>c$. See Figure \ref{Fig-intro} for an example.

In order to formulate our result, we also need to recall the Demazure-Lustig representation of the Hecke algebra. Let $t_{i}$ be an operator which acts on the space of rational functions in $w_1, w_2, \dots, w_k$ by swapping $w_i$ and $w_{i+1}$. Define
$$
T_i :=q+ \frac{w_{i+1} - q w_i}{w_{i+1}-w_i} \left( t_i -1 \right). 
$$
Let $\pi \in S_k$ be a permutation of $\{1,2,\dots,k\}$, let $\sigma_i$ be a transposition of $i$ and $i+1$, and let $\pi = \sigma_{i_1} \sigma_{i_2} \dots \sigma_{i_{l(\pi)}}$ be a shortest decomposition of $\pi$ into a product of nearest neighbor transpositions. Define
$$
T_{\pi} := T_{i_1} T_{i_2} \dots T_{i_{l(\pi)}}.
$$
It is well-known that the construction produces a representation of the Hecke algebra, in particular, the operator $T_{\pi}$ does not depend on the choice of a shortest decomposition of $\pi$. 

The following theorem is Theorem \ref{fusedMainTheorem}.

\begin{theo}
\label{fusedMainTheoremIntro}
Assume that the row rapidities $u_i$ satisfy $u_i\neq qu_j$ for all $i,j$. Then for any $k$-tuples $(\alpha_1, \dots, \alpha_k), (\beta_1, \dots, \beta_k), (c_1, \dots, c_k)$ satisfying
\be
0<\alpha_1\leq \alpha_2\leq \dots\leq \alpha_k,\quad \beta_1\geq \beta_2\geq \dots\geq \beta_k>0,\qquad   \alpha_i,\beta_i\in\Z_{\ge 0}+\frac12,
\ee
\be
0\leq c_1\leq c_2\leq \dots\leq c_k\qquad c_i\in \Z_{\ge 0}
\ee
and any permutation $\pi\in S_k$ we have
\begin{multline}
\label{qmomentseqIntro}
\mathbb E\left(q^{h_{>c_1}^{\(\alpha_{\pi(1)},\beta_{\pi(1)}\)}+h_{>c_2}^{\(\alpha_{\pi(2)},\beta_{\pi(2)}\)}+\dots+h_{>c_k}^{\(\alpha_{\pi(k)},\beta_{\pi(k)}\)}}\right) = q^{\frac{k(k-1)}{2}-l(\pi)} \int_{\Gamma[1|\bu^{-1}]}\cdots\int_{\Gamma[k|\bu^{-1}]}\prod_{a<b}\frac{w_b-w_a}{w_b-qw_a}\\
T_\pi\left( \prod_{a=1}^k\prod_{i=1}^{l_{c_a}}\frac{1-u_iw_a}{1-qu_iw_a}\right)\prod_{a=1}^k\left(\prod_{i=1}^{i<\beta_a}\frac{1-qu_iw_a}{1-u_iw_a}\prod_{j=1}^{j<\alpha_a}\frac{s_j(w_a s_j-y^{-1}_j)}{w_a-s_jy^{-1}_j}\frac{dw_a}{2\pi i w_a}\right).
\end{multline}
where the integral with respect to $w_a$ is taken over the contour $\Gamma[a|\bu^{-1}]$ and the contours $\Gamma[a|\bu^{-1}]$ are $q$-nested around $0$, encircle $\{u_i^{-1}\}_i$ and encircle no other singularities of the integrand, that is, $\{s_jy_j^{-1}\}_j$ and $\{q^{-1}u_i^{-1}\}_i$.
\end{theo}

One possible choice of contours is constructed as follows. Let $\Gamma[\bu^{-1}]$ be a small contour encircling all points $\{u_i^{-1}\}_i$, with all points $0, \{q^{-1}u^{-1}_i\}_i, \{s_jy_j^{-1}\}_j$ being outside of $\Gamma[\bu^{-1}]$. Let $c_0$ be a small circle around $0$ such that all points  $\{u_i^{-1}\}_i, \{q^{-1}u^{-1}_i\}_i, \{s_jy_j^{-1}\}_j$ are outside of $c_0$. Define $\Gamma[k|\bu^{-1}]$ as a union of $\Gamma[\bu^{-1}]$ and $q^{2k}c_0$.

Theorem \ref{fusedMainTheoremIntro} is related to \cite[Theorem 6.1]{BW20} and improves upon it in the following ways. First, and perhaps most significantly, the left-hand side contains q-moments of height functions rather than their fairly complicated linear combinations. Second, \cite[Theorem 6.1]{BW20} addresses only the case when all $\beta_i$'s are equal to each other. Distinct $\beta_i$'s are allowed in our setting, moreover, we are analyzing more general domains than formulated above, see Theorem \ref{mainTheorem} for details. Third, \cite[Theorem 6.1]{BW20} is applicable only to the case when all $s_j$'s are equal to each other and $y_j$'s are equal to each other (in our notations). Thus, Theorem \ref{fusedMainTheoremIntro} allows for more inhomogeneities in the model.

\subsection{Outline} In Section 2 we recall the necessary definitions about SC6V model, largely following \cite{BW18}. In Section 3 we recall the necessary notions about Hecke algebra and establish some auxiliary properties related to it. In Section 4 we prove auxiliary facts about integral expressions that appear in formulas for q-moments. In Section 5 we establish a key new technical ingredient --- the local relation for q-moments of height functions of SC6V model. In Section 6 we prove our main result (Theorem \ref{mainTheorem}). In Section 7 we give a new proof of the shift invariance property for the SC6V model. In Section 8 we extend our main result to \textit{fused} vertex models and give an independent proof for the Borodin-Wheeler formula for observables. Finally, in Section 9 we obtain a formula for moments of delayed partition functions of the Beta polymer.

\textbf{Acknowledgements}. We are grateful to A. Borodin for many very helpful discussions. The work of A.~Bufetov was partially supported by the Deutsche Forschungsgemeinschaft (DFG, German Research Foundation) under Germany's Excellence Strategy -- EXC 2047 ``Hausdorff Center for Mathematics''. S.~Korotkikh was partially supported by the NSF FRG grant DMS-1664619.

\section{Stochastic colored vertex models}

In this section we summarize the terminology of vertex models and previously proved results which will be used in this work. We largely follow the exposition from \cite{BGW19}, with slight modifications.

In this work we will mostly treat the vertex models as a way to define a random collection of colored up-right paths on a finite subdomain of the square lattice $\Z\times\Z$. Colors of paths are always labeled by positive natural numbers, and, depending on the model, each lattice edge either may be occupied by an arbitrary number paths or have a restriction on number of paths along it.

Each vertex in the model produces a weight, depending solely on the colors of paths occupying adjacent lattice edges. The probability of a configuration of paths is given by the product of the weights of all vertices of the model. The choice of permitted configurations and weights of vertices will guarantee that all such products of weights will sum to $1$.

\subsection{Stochastic colored 6-vertex weights.} 

We start with a vertex model with no more than one path occupying any given lattice edge. In this case we can encode the collection of paths by an assignment of labels to the lattice edges: Assume that all paths have colors labeled by $\{1, \dots, n\}$. Then we assign $0$ to a lattice edge if there is no path along the edge and assign $i$ if the edge is occupied by a path of color $i$.\footnote{Throughout this work we will often treat $0$ as an additional color denoting the absence of a path.} We call such labelings \emph{configurations}.

Following the labeling notation above, the weights of vertices of the \emph{stochastic colored six-vertex model} are denoted by
\be
R_z(i,j;k,l)=\tikzbase{0.6}{-0.56ex}{
	\draw[lgray,line width=1.5pt,->] (-1,0) -- (1,0);
	\draw[lgray,line width=1.5pt,->] (0,-1) -- (0,1);
	\node[left] at (-1,0) {\tiny $j$};\node[right] at (1,0) {\tiny $l$};
	\node[below] at (0,-1) {\tiny $i$};\node[above] at (0,1) {\tiny $k$};
},\qquad i,j,k,l\in\{0, \dots, n\}
\ee
and are given in the following table where, by assumption, $0\leq i<j\leq n$ and all unlisted weights are equal to $0$:
\begin{align}
\label{weightsR}
\begin{tabular}{|c|c|c|}
\hline
\quad
\tikz{0.6}{
	\draw[lgray,line width=1.5pt,->] (-1,0) -- (1,0);
	\draw[lgray,line width=1.5pt,->] (0,-1) -- (0,1);
	\node[left] at (-1,0) {\tiny $i$};\node[right] at (1,0) {\tiny $i$};
	\node[below] at (0,-1) {\tiny $i$};\node[above] at (0,1) {\tiny $i$};
}
\quad
&
\quad
\tikz{0.6}{
	\draw[lgray,line width=1.5pt,->] (-1,0) -- (1,0);
	\draw[lgray,line width=1.5pt,->] (0,-1) -- (0,1);
	\node[left] at (-1,0) {\tiny $i$};\node[right] at (1,0) {\tiny $i$};
	\node[below] at (0,-1) {\tiny $j$};\node[above] at (0,1) {\tiny $j$};
}
\quad
&
\quad
\tikz{0.6}{
	\draw[lgray,line width=1.5pt,->] (-1,0) -- (1,0);
	\draw[lgray,line width=1.5pt,->] (0,-1) -- (0,1);
	\node[left] at (-1,0) {\tiny $i$};\node[right] at (1,0) {\tiny $j$};
	\node[below] at (0,-1) {\tiny $j$};\node[above] at (0,1) {\tiny $i$};
}
\quad
\\[1.3cm]
\quad
$1$
\quad
& 
\quad
$\dfrac{q(z-1)}{z-q}$
\quad
& 
\quad
$\dfrac{z(1-q)}{z-q}$
\quad
\\[0.7cm]
\hline
&
\quad
\tikz{0.6}{
	\draw[lgray,line width=1.5pt,->] (-1,0) -- (1,0);
	\draw[lgray,line width=1.5pt,->] (0,-1) -- (0,1);
	\node[left] at (-1,0) {\tiny $j$};\node[right] at (1,0) {\tiny $j$};
	\node[below] at (0,-1) {\tiny $i$};\node[above] at (0,1) {\tiny $i$};
}
\quad
&
\quad
\tikz{0.6}{
	\draw[lgray,line width=1.5pt,->] (-1,0) -- (1,0);
	\draw[lgray,line width=1.5pt,->] (0,-1) -- (0,1);
	\node[left] at (-1,0) {\tiny $j$};\node[right] at (1,0) {\tiny $i$};
	\node[below] at (0,-1) {\tiny $i$};\node[above] at (0,1) {\tiny $j$};
}
\quad
\\[1.3cm]
& 
\quad
$\dfrac{z-1}{z-q}$
\quad
&
\quad
$\dfrac{(1-q)}{z-q}$
\quad 
\\[0.7cm]
\hline
\end{tabular}
\end{align}
The parameter $q$ in the weights above is called a \emph{quantization parameter} and throughout this work we assume that it is a real number from $(0,1)$ equal for all vertices. The parameter $z$ is called a \emph{spectral parameter} and it will vary between vertices.

The vertex weights $R_z$ are related to the stochastic version of the $U_q(\widehat{\frak{sl}_n})$ $R$-matrix from \cite{Jim86}, and they are distinguished by several properties, \emph{cf.} \cite{BW18}.\footnote{In \cite{BW18} and other related works the spectral parameter $z$ in the weights $R_z$ is the inverse of ours. This difference will be offset later, by setting spectral parameter of a vertex equal to the ratio $x/y$ of the row and column rapidities, rather than $y/x$ like in \cite{BW18}.} First of all, they satisfy the following Yang-Baxter equation: for any fixed  $a_1,a_2,a_3,b_1,b_2,b_3\in\{0,\dots, n\}$ and for any $x,y,z\in\mathbb C$ we have

\begin{multline*}
\sum_{k_1,k_2,k_3}R_{x/y}(a_2,a_3;k_2,k_3)R_{x/z}(a_1,k_3;k_1,b_3)R_{y/z}(k_1,k_2;b_1,b_2)\\
=\sum_{k_1,k_2,k_3}R_{y/z}(a_1,a_2;k_1,k_2)R_{x/z}(k_1,a_3;b_1,k_3)R_{x/y}(k_2,k_3;b_2,b_3),
\end{multline*}
where the sums are taken over $k_1,k_2,k_3=\{0,1,\dots, n\}$. Graphically, the same identity can be written as follows:
\begin{equation}
\label{YB}
\sum_{k_1,k_2,k_3}
\tikzbase{0.9}{3ex}{
	\draw[lgray,line width=1.5pt,->]
	(-2,0.5) node[above,scale=0.6] {\color{black} $a_3$} -- (-1,-0.5) node[below,scale=0.6] {\color{black} $k_3$} -- (1,-0.5) node[right,scale=0.6] {\color{black} $b_3$};
	\draw[lgray,line width=1.5pt,->] 
	(-2,-0.5) node[below,scale=0.6] {\color{black} $a_2$} -- (-1,0.5) node[above,scale=0.6] {\color{black} $k_2$} -- (1,0.5) node[right,scale=0.6] {\color{black} $b_2$};
	\draw[lgray,line width=1.5pt,->] 
	(0,-1.5) node[below,scale=0.6] {\color{black} $a_1$} -- (0,0) node[right, scale=0.6] {\color{black} $k_1$} -- (0,1.5) node[above,scale=0.6] {\color{black} $b_1$};
	\node[left] at (-2.2,0.5) {$(x) \rightarrow$};
	\node[left] at (-2.2,-0.5) {$(y) \rightarrow$};
	\node[below] at (0,-1.9) {$\uparrow$};
	\node[below] at (0,-2.4) {$(z)$};
}
\quad
=
\quad
\sum_{k_1,k_2,k_3}
\tikzbase{0.9}{3ex}{
	\draw[lgray,line width=1.5pt,->] 
	(-1,1) node[left,scale=0.6] {\color{black} $a_3$} -- (1,1) node[above,scale=0.6] {\color{black} $k_3$} -- (2,0) node[below,scale=0.6] {\color{black} $b_3$};
	\draw[lgray,line width=1.5pt,->] 
	(-1,0) node[left,scale=0.6] {\color{black} $a_2$} -- (1,0) node[below,scale=0.6] {\color{black} $k_2$} -- (2,1) node[above,scale=0.6] {\color{black} $b_2$};
	\draw[lgray,line width=1.5pt,->] 
	(0,-1) node[below,scale=0.6] {\color{black} $a_1$} -- (0,0.5) node[right, scale=0.6] {\color{black} $k_1$} -- (0,2) node[above,scale=0.6] {\color{black} $b_1$};
	\node[left] at (-1.5,1) {$(x) \rightarrow$};
	\node[left] at (-1.5,0) {$(y) \rightarrow$};
	\node[below] at (0,-1.4) {$\uparrow$};
	\node[below] at (0,-1.9) {$(z)$};
}
\end{equation}
Here the vertices are denoted by transversal intersection of lines, with each line carrying a parameter, called a \emph{rapidity}. The spectral parameter of a vertex is equal to the ratio $x/y$ of the row rapidity $x$ and the column rapidity $y$.

Additionally, the weights $R_z$ are \emph{stochastic} in the following sense:

\begin{equation}
\label{stoch6V}
\sum_{k,l}
\tikzbase{0.6}{-0.56ex}{
	\draw[lgray,line width=1.5pt,->] (-1,0) -- (1,0);
	\draw[lgray,line width=1.5pt,->] (0,-1) -- (0,1);
	\node[left] at (-1,0) {\tiny $j$};\node[right] at (1,0) {\tiny $l$};
	\node[below] at (0,-1) {\tiny $i$};\node[above] at (0,1) {\tiny $k$};
}
=1.
\end{equation}

To describe domains of the vertex model, we will use the following notational convention: set
\be
\dZ:=\left\{i+\frac{1}{2}\ \Big |\ i\in \Z\right\}.
\ee
Then the grid $\ddZ$ can be interpreted as a grid dual to $\Z\times\Z$, with a point $(\alpha,\beta)\in\ddZ$ corresponding to the face of $\Z\times\Z$ bounded by the $\(\alpha\pm\frac{1}{2}\)$th columns and the $\(\beta\pm\frac{1}{2}\)$th rows.

\begin{defin}
An \emph{up-left path} is a sequence of lattice points
\be
P=(\alpha_0, \beta_0)\to\dots\to(\alpha_l,\beta_l),
\ee
where $(\alpha_i,\beta_i)\in\ddZ$ and for each $i=1,\dots, l$ we have
\be
(\alpha_i,\beta_i)=(\alpha_{i-1}-1,\beta_{i-1}),\quad \text{or}\quad (\alpha_i,\beta_i)=(\alpha_{i-1},\beta_{i-1}+1).
\ee
In the first case we say that the step $(\alpha_{i-1},\beta_{i-1})\to(\alpha_i,\beta_i)$ is \emph{horizontal}, while in the latter case the step $(\alpha_{i-1},\beta_{i-1})\to(\alpha_i,\beta_i)$ is called \emph{vertical}. A \emph{coloring} of an up-left path $P$ is a sequence of colors $c_1,\dots, c_l$, where the color $c_i$ is called the color of the $i$th step $(\alpha_{i-1},\beta_{i-1})\to(\alpha_i,\beta_i)$.
\end{defin}

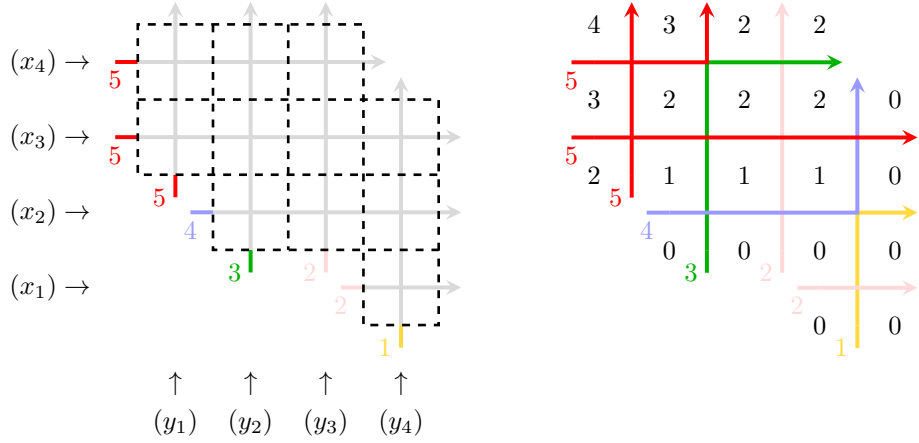
\begin{figure}
\begin{tikzpicture}[scale=1,baseline={([yshift=0]current bounding box.center)},>=stealth]

	\draw[yellow,line width=1.5pt] (4, 0.2) -- (4, 0.5);
	\node[left, text=yellow] at (4, 0.2) {1};
	\draw[lgray,line width=1.5pt, ->] (4, 0.5) -- (4, 3.8);
	\node[below] at (4,0) {$\uparrow$};
	\node[below] at (4,-0.5) {$(y_4)$};
	
	\draw[lred,line width=1.5pt] (3, 1.2) -- (3, 1.5);
	\node[left, text=lred] at (3, 1.2) {2};
	\draw[lgray,line width=1.5pt, ->] (3, 1.5) -- (3, 4.8);
	\node[below] at (3,0) {$\uparrow$};
	\node[below] at (3,-0.5) {$(y_3)$};
	
	\draw[dgreen,line width=1.5pt] (2, 1.2) -- (2, 1.5);
	\node[left, text=dgreen] at (2, 1.2) {3};
	\draw[lgray,line width=1.5pt, ->] (2, 1.5) -- (2, 4.8);
	\node[below] at (2,0) {$\uparrow$};
	\node[below] at (2,-0.5) {$(y_2)$};

	\draw[red,line width=1.5pt] (1, 2.2) -- (1, 2.5);
	\node[left, text=red] at (1, 2.2) {5};	
	\draw[lgray,line width=1.5pt, ->] (1, 2.5) -- (1, 4.8);
	\node[below] at (1,0) {$\uparrow$};
	\node[below] at (1,-0.5) {$(y_1)$};
	
	\draw[lred,line width=1.5pt] (3.2, 1) -- (3.5, 1);
	\node[below, text=lred] at (3.2, 1) {2};
	\draw[lgray,line width=1.5pt, ->] (3.5, 1) -- (4.8, 1);
	\node[left] at (0,1) {$(x_1) \rightarrow$};
	
	\draw[lblue,line width=1.5pt] (1.2, 2) -- (1.5, 2);
	\node[below, text=lblue] at (1.2, 2) {4};
	\draw[lgray,line width=1.5pt, ->] (1.5, 2) -- (4.8, 2);
	\node[left] at (0,2) {$(x_2) \rightarrow$};
	
	\draw[red, line width=1.5pt] (0.2, 3) -- (0.5, 3);
	\node[below, text=red] at (0.2, 3) {5};
	\draw[lgray,line width=1.5pt, ->] (0.5, 3) -- (4.8, 3);
	\node[left] at (0,3) {$(x_3) \rightarrow$};
	
	\draw[red, line width=1.5pt] (0.2, 4) -- (0.5, 4);
	\node[below, text=red] at (0.2, 4) {5};
	\draw[lgray,line width=1.5pt, ->] (0.5, 4) -- (3.8, 4);
	\node[left] at (0,4) {$(x_4) \rightarrow$};
	
	\draw[style=dual] (0.5,3.5) -- (4.5,3.5) -- (4.5,0.5) -- (3.5,0.5) -- (3.5, 4.5) -- (0.5,4.5) -- (0.5, 2.5) -- (4.5, 2.5);
	\draw[style=dual] (1.5, 4.5) -- (1.5, 1.5) -- (4.5,1.5);
	\draw[style=dual] (2.5, 4.5) -- (2.5, 1.5);
\end{tikzpicture}
\hspace{1cm}
\begin{tikzpicture}[scale=1,baseline={([yshift=-15]current bounding box.center)},>=stealth]

	\draw[yellow,line width=1.5pt] (4, 0.2) -- (4, 0.5);
	\node[left, text=yellow] at (4, 0.2) {1};
	\draw[yellow,line width=1.5pt, ->] (4, 0.5) -- (4,2) -- (4.8,2);
	
	\draw[lred,line width=1.5pt] (3, 1.2) -- (3, 1.5);
	\node[left, text=lred] at (3, 1.2) {2};
	\draw[lred,line width=1.5pt, ->] (3, 1.5) -- (3, 4.8);

	\draw[lred,line width=1.5pt] (3.2, 1) -- (3.5, 1);
	\node[below, text=lred] at (3.2, 1) {2};
	\draw[lred,line width=1.5pt, ->] (3.5, 1) -- (4.8, 1);
	
	\draw[dgreen,line width=1.5pt] (2, 1.2) -- (2, 1.5);
	\node[left, text=dgreen] at (2, 1.2) {3};
	\draw[dgreen,line width=1.5pt, ->] (2, 1.5) -- (2,4) -- (3.8, 4);
	
	\draw[lblue,line width=1.5pt] (1.2, 2) -- (1.5, 2);
	\node[below, text=lblue] at (1.2, 2) {4};
	\draw[lblue,line width=1.5pt, ->] (1.5, 2) -- (4, 2) -- (4, 3.8);
	
	\draw[red,line width=1.5pt] (1, 2.2) -- (1, 2.5);
	\node[left, text=red] at (1, 2.2) {5};	
	\draw[red,line width=1.5pt, ->] (1, 2.5) -- (1, 4.8);
	
	\draw[red, line width=1.5pt] (0.2, 3) -- (0.5, 3);
	\node[below, text=red] at (0.2, 3) {5};
	\draw[red,line width=1.5pt, ->] (0.5, 3) -- (4.8, 3);
	
	\draw[red, line width=1.5pt] (0.2, 4) -- (0.5, 4);
	\node[below, text=red] at (0.2, 4) {5};
	\draw[red,line width=1.5pt, ->] (0.5, 4) -- (2,4) --(2, 4.8);
	
	\node[] at (4.5, 0.5) {0};
	\node[] at (3.5, 0.5) {0};
	\node[] at (3.5, 1.5) {0};
	\node[] at (4.5, 1.5) {0};
	\node[] at (4.5, 2.5) {0};
	\node[] at (4.5, 3.5) {0};
	
	\node[] at (3.5, 2.5) {1};
	\node[] at (3.5, 3.5) {2};
	\node[] at (3.5, 4.5) {2};
	
	\node[] at (2.5, 1.5) {0};
	\node[] at (2.5, 2.5) {1};
	\node[] at (2.5, 3.5) {2};
	\node[] at (2.5, 4.5) {2};
	
	\node[] at (1.5, 1.5) {0};
	\node[] at (1.5, 2.5) {1};
	\node[] at (1.5, 3.5) {2};
	\node[] at (1.5, 4.5) {3};	
	
	\node[] at (0.5, 2.5) {2};
	\node[] at (0.5, 3.5) {3};
	\node[] at (0.5, 4.5) {4};	
\end{tikzpicture}
\caption{
\label{skewDomain}
Left: an example of the data for a SC6V model, namely, a skew domain, rapidities of rows and columns and a monotone coloring. Here the coloring is $(1,2,2,3,4,5,5,5)$. Right: an example of a path configuration satisfying the boundary condition on the left picture, as well as the values of the height function $h^{(\alpha,\beta)}_{\geq 4}$ with respect to this configuration.}
\end{figure}

For two up-left paths $P,P'$ of length $l$ with the same endpoints we write $P\leq P'$ if for all $0\leq i\leq l$ we have $\alpha_i+\beta_i\leq \alpha_i'+\beta_i'$. Equivalently, $P\leq P'$ if the path $P$ is to the down-left of the path  $P'$. For a pair of paths $Q\leq P$ the \emph{skew domain} $P-Q$ is defined as the collection of vertices of the lattice $\Z\times\Z$ to the left of path $P$ and to the right of $Q$, as well as all edges, adjacent to these vertices, see Figure \ref{skewDomain}.

Fix a skew domain $P-Q$ and let $\mathbf{c}=(c_1,\dots, c_l)$ be a coloring of $Q$. We will assume that the coloring is \emph{monotonic}, i.e.
\be
c_1\leq c_2\leq\dots\leq c_l.
\ee 
Note that there is a one-to-one correspondence between lattice edges entering the domain from the down-left and the steps of $Q$. Using this correspondence, we can color the incoming lattice edges using the coloring $\mathbf c$. More precisely, we assign color $c_i$ to the lattice edge intersecting the $i$th step of $Q$, see Figure \ref{skewDomain}.

For each row and column intersecting the region we assign a rapidity, with $x_i$ denoting the rapidity of the $i$th row counting from the bottom and $y_j$ denoting the rapidity of the $j$th column counting from the left. Note that each column and row uniquely corresponds to the step of $Q$ it intersects, so we can instead assign parameters $\z_i$ to the steps of the path $Q$, with $\z_i=x_j$ if the $i$th step is vertical and intersects the $j$th row, and similarly for columns. 

The data above defines boundary conditions for the incoming edges and the parameters of each vertex\footnote{with the convention that the spectral parameter of a vertex is the ratio of the row and column rapidities}. Let $\Omega(Q,P;c_1,\dots, c_l)$ denote the set of configurations (\emph{i.e.} labelings of the edges) on the skew domain $P-Q$ such that the labels of the incoming edges are given by the coloring $(c_1,\dots,c_l)$. Finally, we define a complex valued probability measure on $\Omega(Q,P;c_1,\dots, c_l)$, setting the probability of a configuration equal to the product of the $R$-weights of the vertices inside $P-Q$. To sum the whole construction up, we give the following definition.

\begin{defin}
\label{skewDomainDef}
Let $Q\leq P$ be a pair of up-left paths of length $l$, and let $\mathbf{c}=(c_1, c_2,\dots, c_l), \mathbf z=(\z_1, \z_2,\dots, \z_l)$  be parameters assigned to the steps of $Q$ satisfying
\be
c_1\leq c_2\leq\dots\leq c_l.
\ee 
Define $\Omega(Q,P;\mathbf c)$ as the space of all configurations on the skew domain $P-Q$ (alternatively, up-right colored path ensembles on $P-Q$ with no edge being occupied by more than one path) satisfying the boundary conditions defined by $\mathbf{c}$: the color of the incoming edge intersecting the $i$th step of $Q$ coincides with $c_i$. A \emph{SC6V model} $\mathcal M=\mathcal M(Q,P,\mathbf c,\mathbf z)$ associated to this data is a complex probability measure on the space $\Omega(Q,P;\mathbf{c})$, with the probability of a configuration being equal to the product of weights $R_z$ over all the vertices inside $P-Q$, where for a vertex $(i,j)$ the spectral parameter is set to $z=x_i/y_j$.
\end{defin}

The constructed probability measure can be interpreted as a result of the following Markovian sampling procedure: take the minimal $m$ such that the diagonal $i+j=m$ intersects the region of the skew domain. For any vertex on this diagonal, the colors of the incoming edges are defined by the boundary conditions. Now, for each vertex on the diagonal we independently choose one of the possible colorings of the outgoing edges, with probabilities given by the weights of resulting configurations around the vertex. The stochasticity \eqref{stoch6V} of weights implies that all applicable weights sum to $1$. On the next step, we take vertices on the diagonal $i+j=m+1$, whose incoming edges are already determined by the boundary conditions and the choices made on the previous step. Repeating the sampling procedure for the subsequent diagonals we arrive at a random collection of paths on the skew domain.

The stochasticity of the weights $R_z$ also implies the following property: let $Q\leq P\leq P'$, and let $\mathbf c, \mathbf z$ be parameters of a vertex model (assigned to steps of $Q$). Let $\M$ and $\M'$ be probability measures defined on $P-Q$ and $P'-Q$ using $\mathbf c,\mathbf{z}$. Then, by the stochasticity, the restriction of the measure $\M'$ on $\Omega(Q,P;\mathbf c)$ is the measure $\M$. Hence, the up-left path $P$ is not actually essential in the definition of the model, and will be sometimes omitted.

\subsection{Height function} Instead of assigning labels to edges as we did before, one can describe a collection of up-right paths using \emph{colored height functions}. Assume that $\mathcal M$ is a SC6V vertex model on a skew domain $P-Q$, with a monotone coloring $\mathbf{c}=(c_1, \dots, c_l)$ defining the colors of the incoming edges. Let $Q_i$ denote the points of $Q$:
\be
Q=Q_0\to\dots\to Q_l.
\ee 
For each $c\geq 0$ the corresponding colored height function is denoted by $h^{(\alpha,\beta)}_{>c}(\Sigma)$ and it depends on a point $(\alpha,\beta)\in\ddZ$ inside the skew domain $P-Q$ and a configuration $\Sigma\in\Omega(Q, P; \mathbf{c})$ satisfying the boundary conditions given by $\mathbf{c}$. The values of the height function are defined recursively, using the initial condition
\be
h_{>c}^{Q_0}(\Sigma)=0
\ee
and the following local relations: if in the configuration $\Sigma$ the lattice edge between $(\alpha,\beta)$ and $(\alpha+1,\beta)$ has color $i$, then
\begin{equation}
\label{localH1}
h_{>c}^{(\alpha+1,\beta)}(\Sigma)=h_{>c}^{(\alpha,\beta)}(\Sigma)-\1_{i>c}.
\end{equation}
Similarly, if the lattice edge between $(\alpha,\beta)$ and $(\alpha,\beta+1)$ is of color $i$, then
\begin{equation} 
\label{localH2}
h_{>c}^{(\alpha,\beta+1)}(\Sigma)=h_{>c}^{(\alpha,\beta)}(\Sigma)+\1_{i>c}
\end{equation}
One can readily see that the conservation law for vertices\footnote{The number of the incoming paths of a given color coincides with the number of the outgoing paths of a given color.} ensures the existence of the unique well-defined function satisfying the initial condition and relations \eqref{localH1}, \eqref{localH2}. On the other hand, the values of the height functions $h_{>c}^{(\alpha,\beta)}(\Sigma)$ for a fixed configuration $\Sigma$ uniquely determine this configuration because of the local relations \eqref{localH1}, \eqref{localH2}. See Figure \ref{skewDomain} for an example of values of a height function for a given configuration.

\subsection{Exchange relations.} The SC6V model can be used to define a certain family of operators, called \emph{row operators}, with an explicitly described exchange relations between them implied by the Yang-Baxter equation. Fix an integer $M>0$ and consider a vector space $V_M$ spanned by $M$-tuples of colors:
\be
V_M:=\mathrm{Span}_{\mathbb C}\{|\bi\rangle\}=\mathrm{Span}_{\mathbb C}\{|i_1,\dots,i_M\rangle\}_{i_1,\dots,i_M\in\{0,\dots,n\}}\cong \(\mathbb C^{n+1}\)^{\otimes M}.
\ee
We will interpret $V_M$ as the space of configurations of $M$ columns in the SC6V model.  Using this interpretation, we can define row operators $\C_k(x\mid y_1,\dots,y_M):V_M\to V_M$ by specifying their matrix coefficients
\be
\langle \bj|\ \C_k(x\mid y_1,\dots,y_M)\ |\bi\rangle:= 
\tikzbase{0.8}{13}{
	\draw[lgray,line width=1.5pt,->] (1,2) -- (8,2);
	\foreach\x in {2,...,7}{
		\draw[lgray,line width=1.5pt,->] (\x,1) -- (\x,3);
	}
	%top labels
	\node[above] at (7,3) {$j_{M}$};
	\node[above] at (5,3) {$\cdots$};
	\node[above] at (4,3) {$\cdots$};
	\node[above] at (2,3) {$j_1$};
	%left labels
	\node[left] at (0.5,2) {$(x)\rightarrow$}; \node[left] at (1,2) {$k$};
	%bottom labels
	\node[below] at (7,-0.3) {$(y_M)$};\node[below] at (7,0.2) {$\uparrow$}; \node[below] at (7,1) {$i_{M}$};
	\node[below] at (5,1) {$\cdots$};
	\node[below] at (4,1) {$\cdots$};
	\node[below] at (2,-0.3) {$(y_1)$}; \node[below] at (2,0.2) {$\uparrow$}; \node[below] at (2,1) {$i_1$};
	%right labels
	\node[right] at (8,2) {$0$};
},
\ee
where $\{\langle \bj|\}_{\bj}$ is the basis dual to $\{|\bi\rangle\}_{\bi}$. 

Using a ``zipper"-like argument one can show directly from the Yang-Baxter relation \eqref{YB}  that the operators $\C_k$ satisfy the following exchange relation, \emph{cf.} \cite[Section 3.2]{BW18}:
\begin{equation}
\label{exchange}
\C_{k_1}(x_1|\by)\C_{k_2}(x_2|\by)=\frac{x_2-qx_1}{x_2-x_1}\C_{k_2}(x_2|\by)\C_{k_1}(x_1|\by)-\frac{x_1(1-q)}{x_2-x_1}\C_{k_2}(x_1|\by)\C_{k_1}(x_2|\by),
\end{equation}
where $k_1<k_2$ and $\by=(y_1,\dots, y_M)$. Beside the equation above, there exist similar relations for $k_1=k_2$ and $k_1>k_2$, \emph{cf.} \cite{BW18}; in this work we will only need the relation for $k_1<k_2$.

\subsection{Fused vertex weights.} The vertex weights $R_z$ can be used to get a more general vertex model, with relaxed restrictions on the number of paths along lattice edges. The weights of this model are denoted by
\be
W_{x/y}^{(\NN,\MM)}\(
\tikzbase{0.4}{-0.5ex}{
	\draw[lgray,line width=1pt,->] (-1,0) -- (1,0);
	\draw[lgray,line width=1pt,->] (0,-1) -- (0,1);
	\node[left] at (-1,0) {\tiny $\bB$};\node[right] at (1,0) {\tiny $\bD$};
	\node[below] at (0,-1) {\tiny $\bA$};\node[above] at (0,1) {\tiny $\bC$};
}
\)=
\tikzbase{0.9}{2ex}{
	\draw[lgray,line width=4pt,->] (-1,0) -- (1,0);
	\draw[lgray,line width=4pt,->] (0,-1) -- (0,1);
	\node[left] at (-1,0) {\tiny $\bB$};\node[right] at (1,0) {\tiny $\bD$};
	\node[below] at (0,-1) {\tiny $\bA$};\node[above] at (0,1) {\tiny $\bC$};
	%%%%
	\node[left] at (-1.3,0) {$(x,\NN) \rightarrow$};
	\node[below] at (0,-1.4) {$\uparrow$};
	\node[below] at (0,-1.9) {$(y,\MM)$};
}
\ee
Here the horizontal edges can carry no more than $\NN$ paths and the vertical edges carry no more that $\MM$ paths. The colors of paths along edges are encoded by \emph{color compositions} $\bA,\bB,\bC,\bD\in\mathbb \Z_{\geq 0}^n$: the lattice edge labeled by the color composition $\bA=(A_1,\dots, A_n)$ has $A_i$ paths of color $i$, and similarly for the other compositions $\bB,\bC,\bD$. In terms of the color compositions, the constraint on the number of paths can be written as
\begin{equation}
\label{constraint}
|\bA|,|\bC|\leq \MM,\quad |\bB|,|\bD|\leq \NN.
\end{equation}
where $|\bP|=\sum_{i=1}^n P_i$ for $\bP=(P_1, \dots, P_n)$. Finally, $x$ and $y$ denote the rapidities of the row and the column respectively. 

The weights $W^{(\NN,\MM)}_z$ are constructed from the basic weights $R_z$ by a procedure called \emph{stochastic fusion}\footnote{Fusion in an algebraic setting was originally introduced in the works \cite{KRS81} and \cite{KR83}.}. For a vector $(i_1,\dots, i_{\NN})\in\{0,\dots, n\}^{\NN}$ define its color composition by
\be
\Comp(i_1,\dots, i_{\NN}):=(I_1,\dots, I_n), \quad I_k=\#\{a: i_a=k\}.
\ee
Set 
\be
\inv(i_1,\dots, i_{\NN})=\#\{1\leq a<b\leq N: i_a>i_b\},\quad \tinv(i_1,\dots, i_{\NN})=\#\{1\leq a<b\leq N: i_a<i_b\},
\ee
\be
Z_q(\NN,\bI)=\sum_{\Comp(i_1,\dots, i_{\NN})=\bI}q^{\inv(i_1,\dots, i_{\NN})}=\sum_{\Comp(i_1,\dots, i_{\NN})=\bI}q^{\tinv(i_1,\dots, i_{\NN})}=\frac{(q;q)_{\NN}}{(q;q)_{I_0} \dots (q;q)_{I_n}}
\ee
where $\bI=(I_1,\dots, I_n)$ and $I_0=\NN-\sum_{k=1}^n I_k$.

Having established the needed notation above, for $\NN,\MM\geq 1$ and color compositions $\bA,\bB,\bC,\bD$ satisfying constraints \eqref{constraint} we set
\begin{multline*}
W^{(\NN,\MM)}_{x/y}\(\vert{\bA}{\bB}{\bC}{\bD}\):=\frac{Z_q(\MM; \bC) Z_q(\NN;\bD)q^{-{\inv}(k_1,\dots,k_{\MM})}q^{-\tinv(l_1,\dots,l_{\NN})}}{Z_q(\MM; \bA) Z_q(\NN;\bB)}\\
\times \sum_{\substack{\Comp(i_1,\dots, i_{\MM})=\bA\\ \Comp(j_1,\dots, j_{\NN})=\bB}}q^{{\inv}(i_1,\dots,i_{\MM})}q^{\tinv(j_1,\dots,j_{\NN})}
\tikz{0.8}{
	\foreach\y in {2,...,5}{
		\draw[lgray,line width=1.5pt,->] (1,\y) -- (8,\y);
	}
	\foreach\x in {2,...,7}{
		\draw[lgray,line width=1.5pt,->] (\x,1) -- (\x,6);
	}
	%top labels
	\node[above] at (7,6) {$k_{\MM}$};
	\node[above] at (5,6.1) {\tiny $\cdots$};
	\node[above] at (4,6.1) {\tiny $\cdots$};
	\node[above] at (2,6) {$k_1$};
	%left labels
	\node[left] at (0,2) {$(x)\rightarrow$}; \node[left] at (1,2) {$j_1$};
	\node[left] at (0.5,3) {\tiny $\vdots$};
	\node[left] at (0.5,4) {\tiny$\vdots$};
	\node[left] at (0,5) {$(q^{\NN-1}x)\rightarrow$}; \node[left] at (1,5) {$j_{\NN}$};
	%bottom labels
	\node[below] at (7,-1) {$(y)$};\node[below] at (7,0) {$\uparrow$}; \node[below] at (7,1) {$i_{\MM}$};
	\node[below] at (5.2,0.7) {\tiny $\cdots$};
	\node[below] at (3.8,0.7) {\tiny $\cdots$};
	\node[below] at (2,-0.8) {$(q^{\MM-1}y)$}; \node[below] at (2,0.2) {$\uparrow$}; \node[below] at (2,1) {$i_1$};
	%right labels
	\node[right] at (8,2) {$l_1$};
	\node[right] at (8.1,3) {\tiny $\vdots$};
	\node[right] at (8.1,4) {\tiny $\vdots$};
	\node[right] at (8,5) {$l_{\NN}$};
}
\end{multline*}
where $(k_1,\dots, k_{\MM}), (l_1,\dots, l_{\NN})$ are any vectors such that $\Comp(k_1,\dots, k_{\MM})=\bC, \Comp(l_1,\dots, l_{\NN})=\bD$, and the picture in the right hand side denotes the sum of the products of the vertex weights $R_z$ over all configurations with colors of the boundary edges given by $i,j,k,l$.\footnote{As before, the spectral parameter $z$ of a vertex is the ratio of the row and the column rapidities.} Using the Yang-Baxter relation one can show that the right-hand side does not depend on the choice of $(k_1,\dots, k_{\MM})$ and $(l_1,\dots, l_{\NN})$; this property is called \emph{$q$-exchangeability} and it allows stacking such vertices to get a well-defined vertex model, see \cite[Section~8]{BGW19} for more details on this construction.

Following \cite{BM16} (see also \cite[Theorem C.1.1]{BW18} and \cite[Theorem 8.5]{BGW19}) the weights $W^{(\NN,\MM)}_z$ can be given by an explicit expression
\begin{multline*}
W_{z}^{(\NN,\MM)}\(
\tikzbase{0.4}{-0.5ex}{
	\draw[lgray,line width=1pt,->] (-1,0) -- (1,0);
	\draw[lgray,line width=1pt,->] (0,-1) -- (0,1);
	\node[left] at (-1,0) {\tiny $\bB$};\node[right] at (1,0) {\tiny $\bD$};
	\node[below] at (0,-1) {\tiny $\bA$};\node[above] at (0,1) {\tiny $\bC$};
}
\)=
\1_{|\bA|+|\bB|=|\bC|+|\bD|} z^{|\bD|-|\bB|}q^{|\bA|\NN-|\bD|\MM}\\
\times \sum_{\bP}\Phi(\bC-\bP, \bC+\bD-\bP;q^{\NN-\MM}z;q^{-\MM}z)\Phi(\bP,\bB;q^{-\NN}/z;q^{-\NN}),
\end{multline*}
where the sum is taken over $\bP=(P_1,\dots, P_n)\in\Z_{\geq 0}^n$ such that $P_i\leq \min(B_i,C_i)$ for all $1\leq i\leq n$, and for any pair of compositions $\bA,\bB$ we define
\be
\Phi(\bA,\bB;x,y):=\frac{(x;q)_{|\bA|}(y/x;q)_{|\bB-\bA|}}{(y;q)_{|\bB|}}(y/x)^{|\bA|}q^{\sum_{i<j}(B_i-A_i)A_j}\prod_{i=1}^n\binom{B_i}{A_i}_q.
\ee
Here we use the $q$-Pochhammer symbol and the $q$-binomial coefficient:
\be
(x;q)_n:=\prod_{i=1}^n(1-q^{i-1}x);\qquad \binom{n}{m}_q:=\frac{(q;q)_n}{(q;q)_m(q;q)_{n-m}}.
\ee 
Note that for fixed compositions of the edges, the weight $W_z^{(\NN,\MM)}$ depends rationally on $q^{\NN}$ and $q^{\MM}$ with no other dependence on $\NN,\MM$.

The weights $W_z^{(\NN,\MM)}$ satisfy the following version of the Yang-Baxter equation, \emph{cf.} \cite[(C.1.2)]{BW18}:
\be
\sum_{\bK_1,\bK_2,\bK_3}
\tikzbase{0.9}{3ex}{
	\draw[lgray,line width=4pt,->]
	(-2,0.5) node[above,scale=0.6] {\color{black} $\bA_3$} -- (-1,-0.5) node[below,scale=0.6] {\color{black} $\bK_3$} -- (1,-0.5) node[right,scale=0.6] {\color{black} $\bB_3$};
	\draw[lgray,line width=4pt,->] 
	(-2,-0.5) node[below,scale=0.6] {\color{black} $\bA_2$} -- (-1,0.5) node[above,scale=0.6] {\color{black} $\bK_2$} -- (1,0.5) node[right,scale=0.6] {\color{black} $\bB_2$};
	\draw[lgray,line width=4pt,->] 
	(0,-1.5) node[below,scale=0.6] {\color{black} $\bA_1$} -- (0,0) node[right, scale=0.6] {\color{black} $\bK_1$} -- (0,1.5) node[above,scale=0.6] {\color{black} $\bB_1$};
	\node[left] at (-2.2,0.5) {$(x,\NN) \rightarrow$};
	\node[left] at (-2.2,-0.5) {$(y,\MM) \rightarrow$};
	\node[below] at (0,-1.9) {$\uparrow$};
	\node[below] at (0,-2.4) {$(z,\LL)$};
}
\quad
=
\quad
\sum_{\bK_1,\bK_2,\bK_3}
\tikzbase{0.9}{3ex}{
	\draw[lgray,line width=4pt,->] 
	(-1,1) node[left,scale=0.6] {\color{black} $\bA_3$} -- (1,1) node[above,scale=0.6] {\color{black} $\bK_3$} -- (2,0) node[below,scale=0.6] {\color{black} $\bB_3$};
	\draw[lgray,line width=4pt,->] 
	(-1,0) node[left,scale=0.6] {\color{black} $\bA_2$} -- (1,0) node[below,scale=0.6] {\color{black} $\bK_2$} -- (2,1) node[above,scale=0.6] {\color{black} $\bB_2$};
	\draw[lgray,line width=4pt,->] 
	(0,-1) node[below,scale=0.6] {\color{black} $\bA_1$} -- (0,0.5) node[right, scale=0.6] {\color{black} $\bK_1$} -- (0,2) node[above,scale=0.6] {\color{black} $\bB_1$};
	\node[left] at (-1.5,1) {$(x,\NN) \rightarrow$};
	\node[left] at (-1.5,0) {$(y,\MM) \rightarrow$};
	\node[below] at (0,-1.4) {$\uparrow$};
	\node[below] at (0,-1.9) {$(z,\LL)$};
}
\ee
Moreover, they are stochastic:
\be
\sum _{\bC,\bD}W_{z}^{(\NN,\MM)}\(
\tikzbase{0.4}{-0.5ex}{
	\draw[lgray,line width=1pt,->] (-1,0) -- (1,0);
	\draw[lgray,line width=1pt,->] (0,-1) -- (0,1);
	\node[left] at (-1,0) {\tiny $\bB$};\node[right] at (1,0) {\tiny $\bD$};
	\node[below] at (0,-1) {\tiny $\bA$};\node[above] at (0,1) {\tiny $\bC$};
}
\)=1.
\ee
Both properties above can be seen as the consequence of the analogous properties for the weights $R_z$, via the fusion construction. 

In Section \ref{fusionSection} we will use two specializations of the weights $W^{(\NN,\MM)}_z$.  The \emph{stochastic colored higher spin six-vertex} weights $L^{(s)}_z$ are defined by
\be
L_{z}^{(s)}\(
\tikzbase{0.4}{-0.5ex}{
	\draw[lgray,line width=1pt,->] (-1,0) -- (1,0);
	\draw[lgray,line width=1pt,->] (0,-1) -- (0,1);
	\node[left] at (-1,0) {\tiny $j$};\node[right] at (1,0) {\tiny $l$};
	\node[below] at (0,-1) {\tiny $\bI$};\node[above] at (0,1) {\tiny $\bK$};
}
\):=
\restr{W_{z/s}^{(1,\MM)}\(
\tikzbase{0.4}{-0.5ex}{
	\draw[lgray,line width=1pt,->] (-1,0) -- (1,0);
	\draw[lgray,line width=1pt,->] (0,-1) -- (0,1);
	\node[left] at (-1,0) {\tiny $\e^j$};\node[right] at (1,0) {\tiny $\e^l$};
	\node[below] at (0,-1) {\tiny $\bI$};\node[above] at (0,1) {\tiny $\bK$};
}
\)}{q^{\MM}=s^{-2}}
\ee
where we use that $W_{z}^{(\NN,\MM)}$ depend rationally on $q^{\MM}$ to analytically continue the weights and set $q^{\MM}=s^{-2}$. An explicit expression for these weights is summarised in the following table:
\begin{align}
\label{Lweights}
\begin{tabular}{|c|c|c|}
\hline
\quad
\tikz{0.7}{
	\draw[lgray,line width=1.5pt,->] (-1,0) -- (1,0);
	\draw[lgray,line width=4pt,->] (0,-1) -- (0,1);
	\node[left] at (-1,0) {\tiny $0$};\node[right] at (1,0) {\tiny $0$};
	\node[below] at (0,-1) {\tiny $\bI$};\node[above] at (0,1) {\tiny $\bI$};
}
\quad
&
\quad
\tikz{0.7}{
	\draw[lgray,line width=1.5pt,->] (-1,0) -- (1,0);
	\draw[lgray,line width=4pt,->] (0,-1) -- (0,1);
	\node[left] at (-1,0) {\tiny $i$};\node[right] at (1,0) {\tiny $i$};
	\node[below] at (0,-1) {\tiny $\bI$};\node[above] at (0,1) {\tiny $\bI$};
}
\quad
&
\quad
\tikz{0.7}{
	\draw[lgray,line width=1.5pt,->] (-1,0) -- (1,0);
	\draw[lgray,line width=4pt,->] (0,-1) -- (0,1);
	\node[left] at (-1,0) {\tiny $0$};\node[right] at (1,0) {\tiny $i$};
	\node[below] at (0,-1) {\tiny $\bI$};\node[above] at (0,1) {\tiny $\bI-\e^i$};
}
\quad
\\[1.3cm]
\quad
$\dfrac{1-s z q^{I_{[1;n]}}}{1-sz}$
\quad
& 
\quad
$\dfrac{(s^2q^{I_i}-sz) q^{I_{[i+1;n]}}}{1-sz}$
\quad
& 
\quad
$\dfrac{sz(q^{I_i}-1) q^{I_{[i+1;n]}}}{1-sz}$
\quad
\\[0.7cm]
\hline
\quad
\tikz{0.7}{
	\draw[lgray,line width=1.5pt,->] (-1,0) -- (1,0);
	\draw[lgray,line width=4pt,->] (0,-1) -- (0,1);
	\node[left] at (-1,0) {\tiny $i$};\node[right] at (1,0) {\tiny $0$};
	\node[below] at (0,-1) {\tiny $\bI$};\node[above] at (0,1) {\tiny $\bI+\e^i$};
}
\quad
&
\quad
\tikz{0.7}{
	\draw[lgray,line width=1.5pt,->] (-1,0) -- (1,0);
	\draw[lgray,line width=4pt,->] (0,-1) -- (0,1);
	\node[left] at (-1,0) {\tiny $i$};\node[right] at (1,0) {\tiny $j$};
	\node[below] at (0,-1) {\tiny $\bI$};\node[above] at (0,1) 
	{\tiny $\bI+\e^i-\e^j$};
}
\quad
&
\quad
\tikz{0.7}{
	\draw[lgray,line width=1.5pt,->] (-1,0) -- (1,0);
	\draw[lgray,line width=4pt,->] (0,-1) -- (0,1);
	\node[left] at (-1,0) {\tiny $j$};\node[right] at (1,0) {\tiny $i$};
	\node[below] at (0,-1) {\tiny $\bI$};\node[above] at (0,1) {\tiny $\bI+\e^j-\e^i$};
}
\quad
\\[1.3cm] 
\quad
$\dfrac{1-s^2 q^{I_{[1;n]}}}{1-sz}$
\quad
& 
\quad
$\dfrac{sz(q^{I_j}-1) q^{I_{[j+1;n]}}}{1-sz}$
\quad
&
\quad
$\dfrac{s^2(q^{I_i}-1)q^{I_{[i+1;n]}}}{1-sz}$
\quad
\\[0.7cm]
\hline
\end{tabular} 
\end{align}
where $\bI=(I_1,\dots, I_n)$, all unlisted weights are $0$, and we use notation $I_{[a;b]}=\sum_{k=a}^bI_k$.

The \emph{$q$-Hahn} weights are obtained from the weights $W^{qH}_{z,s}$ by
\be
W^{qH}_{s,z}\(
\tikzbase{0.4}{-0.5ex}{
	\draw[lgray,line width=1pt,->] (-1,0) -- (1,0);
	\draw[lgray,line width=1pt,->] (0,-1) -- (0,1);
	\node[left] at (-1,0) {\tiny $\bB$};\node[right] at (1,0) {\tiny $\bC$};
	\node[below] at (0,-1) {\tiny $\bA$};\node[above] at (0,1) {\tiny $\bD$};
}
\):=
\restr{W_{1}^{(\NN,\MM)}\(
\tikzbase{0.4}{-0.5ex}{
	\draw[lgray,line width=1pt,->] (-1,0) -- (1,0);
	\draw[lgray,line width=1pt,->] (0,-1) -- (0,1);
	\node[left] at (-1,0) {\tiny $\bB$};\node[right] at (1,0) {\tiny $\bC$};
	\node[below] at (0,-1) {\tiny $\bA$};\node[above] at (0,1) {\tiny $\bD$};
}
\)}{\substack{q^{\NN}=z^{-2}\\q^{\MM}=s^{-2}}}
\ee
and they have the form
\begin{equation}
\label{eq:qHahExplicit}
W_{s,z}^{qH} \(
\tikzbase{0.4}{-0.5ex}{
	\draw[lgray,line width=1pt,->] (-1,0) -- (1,0);
	\draw[lgray,line width=1pt,->] (0,-1) -- (0,1);
	\node[left] at (-1,0) {\tiny $\bB$};\node[right] at (1,0) {\tiny $\bD$};
	\node[below] at (0,-1) {\tiny $\bA$};\node[above] at (0,1) {\tiny $\bC$};
}
\)  =
\left( \frac{s^2}{z^2} \right)^{|\bD|} \frac{(s^2/z^2;q)_{|\bA|-|\bD|} (z^2;q)_{|\bD|}}{(s^2;q)_{|\bA|}} q^{\sum_{i<j} D_i (A_j-D_j)} \prod_{i\ge 1} \binom{A_i}{A_i-D_i}_q.
\end{equation}
Note that if $A_i-D_i<0$, then (by definition) the weight of the vertex is equal to 0. One can interpret this by saying that the paths entering a vertex from the left must turn to the top, while arrows coming from the bottom continue to the top or to the right depending on the weight of the vertex (note that this weight depends on $\bA$ and $\bD$ only).

The vertex models with the weights $W_z^{(\NN,\MM)}, L^{(s)}_z$ and $W^{qH}_{s,z}$ can be constructed in a similar fashion as for the weights $R_z$: one defines a domain of vertices on the square lattice, assigns rapidities and an additional spin parameters $\NN,\MM$ or $s$ to rows and columns, specifies the boundary conditions and constructs a probability measure on the configurations. For the weights $L_z^{(s)}$ a configuration assigns integer labels to horizontal edges and color compositions to vertical ones, while for the weights $W_z^{(\NN,\MM)}$ and $W^{qH}_{s,z}$ a configuration is an assignment of color compositions to all edges.

The definition of the height function $h^{(\alpha,\beta)}_{>c}(\Sigma)$ can also be extended to the fused models. To do it, we set the value $h^{(\alpha,\beta)}_{>c}(\Sigma)$ for a single point $(\alpha,\beta)$ and use the following modification of the recurrence relations \eqref{localH1},\eqref{localH2}:
\be
h_{>c}^{(\alpha+1,\beta)}(\Sigma)=h_{>c}^{(\alpha,\beta)}(\Sigma)-I_{[c+1,n]},
\ee
\be
h_{>c}^{(\alpha,\beta+1)}(\Sigma)=h_{>c}^{(\alpha,\beta)}(\Sigma)+I_{[c+1, n]},
\ee
where $\bI$ is a color composition assigned to the corresponding edge (here we treat integer labels $i$ as color compositions $\e^i$).

\subsection{Color merging.} Different versions of color merging and color-blindness statements were described and proved in \cite{BW18},\cite{BGW19}, \cite{BW20}. Here we formulate the statement applicable to our setting. 

Let $\theta:\{1,\dots, n\}\to\{0, \dots, m\}$ be a monotone map, where $n,m$ are positive integers. It induces a map on compositions $\theta_*:\Z_{\geq 0}^n\to\Z_{\geq 0}^m$ defined by
\be
\theta_*:\bI=(I_1,\dots, I_n)\mapsto \bJ=(J_1,\dots, J_m),\quad \text{where}\ \ J_j=\sum_{i\in\theta^{-1}(j)}I_i
\ee
In terms of colored paths the map $\theta_*$ corresponds to merging all colors in $\theta^{-1}(j)$ into a new color labeled $j$, while all paths of colors $\theta^{-1}(0)$ are removed.

\begin{prop}
For any $\NN,\MM\geq 1$ and any color compositions $\bA,\bB\in \Z_{\geq 0}^n$ $\bC',\bD'\in \Z_{\geq 0}^m$ satisfying constraints \eqref{constraint}  we have
\be
\sum_{\bC=\theta_*(\bC'), \bD=\theta_*(\bD')} W_{z}^{(\NN,\MM)}\(
\tikzbase{0.4}{-0.5ex}{
	\draw[lgray,line width=1pt,->] (-1,0) -- (1,0);
	\draw[lgray,line width=1pt,->] (0,-1) -- (0,1);
	\node[left] at (-1,0) {\tiny $\bB$};\node[right] at (1,0) {\tiny $\bD$};
	\node[below] at (0,-1) {\tiny $\bA$};\node[above] at (0,1) {\tiny $\bC$};
}
\)=W_{z}^{(\NN,\MM)}\(
\tikzbase{0.4}{-0.5ex}{
	\draw[lgray,line width=1pt,->] (-1,0) -- (1,0);
	\draw[lgray,line width=1pt,->] (0,-1) -- (0,1);
	\node[left] at (-1,0) {\tiny $\theta_*(\bB)$};\node[right] at (1,0) {\tiny $\bD'$};
	\node[below] at (0,-1) {\tiny $\theta_*(\bA)$};\node[above] at (0,1) {\tiny $\bC'$};
}
\)
\ee
\end{prop}
\begin{proof}
This property is readily seen for $R$-weights from the explicit expression for them and monotonicity of $\theta$. Then one can perform stochastic fusion to extend the result to $W$-weights.
\end{proof}
\begin{rem} Note that in the left-hand side the weights are taken in the model with $n$ colors, while in the right-hand side the weights are taken in the model with $m$ colors. This proposition, in particular, shows that there is no ambiguity in not specifying the total number of possible colors.
\end{rem}
\begin{rem} One application of color merging is color-blindness, which merges all nonzero colors into a single color. This reduces the colored vertex models to uncolored ones, which were studied before in \cite{BP16} and other works.
\end{rem}

\section{Hecke algebras and Demazure-Lusztig operators} \label{symmetricGroupSection}
In this section we establish the notation and provide several facts related to symmetric groups, Hecke algebras and Demazure-Lusztig operators.
\subsection{Symmetric group.} Let $S_k$ denote the symmetric group of rank $k$, that is, the group of permutations of $\{1, \dots, k\}$. The multiplication is defined by composition of permutations: $\pi\tau(i)=\pi(\tau(i))$. Let $\sigma_i$ denote the simple transposition exchanging $i$ and $i+1$ while leaving other elements intact. A \emph{reduced expression} of a permutation $\pi\in S_k$ is a decomposition of $\pi$ into a product of simple transpositions
\be
\pi=\sigma_{i_1}\dots\sigma_{i_l}
\ee 
with the minimal amount of simple transpositions used.  This number of simple transpositions is called the \emph{length} of $\pi$ and it is denoted by $l(\pi)$.

The \emph{(strong) Bruhat order} on $S_k$ is a partial order defined by the covering relations $\pi\preceq \pi\cdot \tau_{a,b}$ whenever $l(\pi)<l(\pi\cdot \tau_{a,b})$. Here $\tau_{a,b}$ denotes the transposition exchanging $a$ and $b$. Alternatively, $\rho\preceq \pi$ if and only if
\be
\rho=\sigma_{j_1}\dots\sigma_{j_r}
\ee for some reduced expression 
\be
\pi=\sigma_{i_1}\dots\sigma_{i_l}
\ee
and a subsequence $(j_1,\dots,j_k)$ of the sequence $(i_1,\dots, i_l)$.

Define a left action of $S_k$ on $k$-tuples by
\be
\pi.(w_1,\dots, w_k):=(w_{\pi^{-1}(1)},\dots, w_{\pi^{-1}(k)}).
\ee
Similarly, define an action $t_\pi$ of $S_k$ on the space of functions in $k$ variables $w_1,\dots, w_k$:
\be
t_\pi f(\bw)=t_\pi f(w_1,\dots, w_k):=f(w_{\pi(1)},\dots, w_{\pi(k)})=f(\pi^{-1}.\mathbf{w}),
\ee
where $\mathbf {w}:=(w_1,\dots, w_k)$.

In the following sections we will also use the following notation. For $a\leq b$ set
\be
\sigma^+_{[a,b]}:=\sigma_{a}\sigma_{a+1}\dots \sigma_{b-1},\qquad \sigma^-_{[a,b]}:=\sigma_{b-1}\sigma_{b-2}\dots \sigma_{a},
\ee
so $\sigma^+_{[a,b]}$ is a cycle sending $b\mapsto a$ and $i\mapsto i+1$ for $a\leq i<b$, while $\sigma^-_{[a,b]}$ cycles the same elements in the opposite direction. We say that a permutation $\pi\in S_k$ is \emph{$[a,b]$-ordered} if 
\begin{equation}
\label{ordered}
\pi^{-1}(i)<\pi^{-1}(j),\quad \text{for}\quad a\leq i<j\leq b.
\end{equation}

\subsection{Demazure-Lusztig operators.} We define the Hecke algebra $H_k$ as an associative algebra over $\mathbb C$ with the basis $\{T_\pi\}$ enumerated by partitions $\pi\in S_k$ and satisfying the following relations
\be
T_{\pi}T_{\tau}=T_{\pi\tau}\quad\text{if}\ \ l(\pi\tau)=l(\pi)+l(\tau); \qquad\qquad (T_i-q)(T_i+1)=0\quad\text{for}\ 1\leq i< k,
\ee
where $T_i:=T_{\sigma_i}$.  In this work we identify the algebra $H_k$ with a certain representation of it, given by the \emph{Demazure-Lusztig} operators
\begin{equation}
\label{defT}
T_i=q+\frac{w_{i+1}-qw_i}{w_{i+1}-w_i}(t_i-1)=\frac{(q-1)w_{i+1}}{w_{i+1}-w_i}+\frac{w_{i+1}-qw_i}{w_{i+1}-w_i}t_i,
\end{equation}
which act on the space of rational functions in $w_1,\dots, w_k$. Here $t_i:=t_{\sigma_i}$ is an operator permuting $w_i$ and $w_{i+1}$. Note that these operators preserve the space of polynomials in $w_1,\dots, w_k$, giving a polynomial representation of $H_k$.

\subsection{Coefficients of $T_\pi$} \label{coefficientsSection} Later, in Section \ref{shiftSec}, we will need a more detailed description of the operators $T_\pi$. 
\begin{prop}
\label{kappaDef}
For any permutation $\pi\in S_k$ we have
\be
T_\pi=\sum_{\rho}\kappa_\pi^\rho(\bw)t_\rho,
\ee
where $\kappa_\rho^\pi(\bw)$ are rational functions in $\bw=(w_1,\dots, w_k)$. Moreover, $\kappa_\rho^\pi(\bw)=0$ unless $\rho\preceq\pi$.
\end{prop}
\begin{proof}
Choose a reduced expression $\pi=\sigma_{i_1}\dots\sigma_{i_l}$. Then we have
\be
T_\pi=T_{i_1}\dots T_{i_l}=\(\frac{(q-1)w_{i_1+1}}{w_{i_1+1}-w_{i_1}}+\frac{w_{i_1+1}-qw_{i_1}}{w_{i_1+1}-w_{i_1}}t_{i_1}\)\dots\(\frac{(q-1)w_{i_l+1}}{w_{i_l+1}-w_{i_l}}+\frac{w_{i_l+1}-qw_{i_l}}{w_{i_l+1}-w_{i_l}}t_{i_l}\)
\ee
Multiplying all terms in the right-hand side and moving operators $t_i$ to the right using the commutation relation 
\be
t_\pi f(\bw)=f(\pi^{-1}.\bw)t_\pi
\ee 
we obtain an expression of the following form
\begin{equation}
\label{reducedT}
T_\pi=\sum_{j_1,\dots,j_r}\mathcal A_{j_1,\dots j_r}(\bw)t_{j_1}\dots t_{j_r},
\end{equation}
where the sum is taken over subsequences $(j_1,\dots, j_r)$ of the sequence $(i_1,\dots, i_l)$ and all $\mathcal A(\bw)$ are rational functions. Thus the first part of the claim holds, with $\kappa^\rho_\pi(\bw)$ being sums of appropriate $\mathcal A_{j_1,\dots j_r}(\bw)$.

Note that $\kappa_\pi^\rho(\bw)\neq 0$ only for permutations $\rho$ such that $\rho=\sigma_{j_1}\dots\sigma_{j_r}$ for some subsequence $(j_1,\dots, j_r)$ of $(i_1,\dots, i_l)$. Using the alternative description of the Bruhat order, we have
\be
\rho=\sigma_{j_1}\dots \sigma_{j_s}\preceq \pi
\ee
for all such $\rho$. Hence we can have $\kappa_\pi^\rho(\bw)\neq 0$ only if $\rho\preceq \pi$.
\end{proof}

In this section we will describe the coefficients $\kappa_\pi^\rho(\bw)$. Recall that for $l(\pi)<l(\pi\sigma_{i})$ we have $T_{\pi\sigma_{i}}=T_\pi T_i$, hence
\be
\sum_{\rho}\kappa^{\rho}_{\pi\sigma_i}(\bw)t_{\rho}=\(\sum_{\rho}\kappa^{\rho}_{\pi}(\bw)t_{\rho}\)\(\frac{(q-1)w_{i+1}}{w_{i+1}-w_i}+\frac{w_{i+1}-qw_{i}}{w_{i+1}-w_{i}}t_i\).
\ee
which, after moving $t_\rho$ to the right, produces the following relation
\begin{equation}
\label{localZ}
\kappa^\rho_{\pi\sigma_{i}}(\bw)=\frac{(q-1)w_{\rho(i+1)}}{w_{\rho(i+1)}-w_{\rho(i)}} \kappa^\rho_{\pi}(\bw)+\frac{w_{\rho(i)}-qw_{\rho(i+1)}}{w_{\rho(i)}-w_{\rho(i+1)}}\kappa^{\rho\sigma_{i}}_{\pi}(\bw).
\end{equation}

Note that the relation \eqref{localZ} resembles the exchange relation \eqref{exchange}. To make this connection more explicit, for a pair of permutations $\rho,\pi$ define a partition function
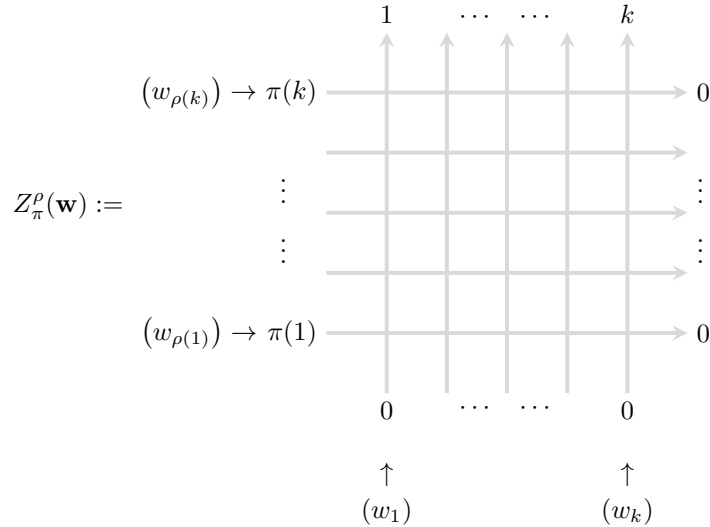
\be
Z^\rho_\pi(\bw):=
\tikzbase{0.8}{20}{
	\foreach\y in {2,...,6}{
		\draw[lgray,line width=1.5pt,->] (1,\y) -- (7,\y);
	}
	\foreach\x in {2,...,6}{
		\draw[lgray,line width=1.5pt,->] (\x,1) -- (\x,7);
	}
	%top labels
	\node[above] at (6,7) {$k$};
	\node[above] at (4.5,7) {$\cdots$};
	\node[above] at (3.5,7) {$\cdots$};
	\node[above] at (2,7) {$1$};
	%left labels
	\node[left] at (0,2) {$\(w_{\rho(1)}\)\rightarrow$}; \node[left] at (1,2) {$\pi(1)$};
	\node[left] at (0.5,3.5) {$\vdots$};
	\node[left] at (0.5,4.5) {$\vdots$};
	\node[left] at (0,6) {$\(w_{\rho(k)}\)\rightarrow$}; \node[left] at (1,6) {$\pi(k)$};
	%bottom labels
	\node[below] at (6,-0.6) {$(w_k)$};\node[below] at (6,0) {$\uparrow$}; \node[below] at (6,1) {$0$};
	\node[below] at (4.5,1) {$\cdots$};
	\node[below] at (3.5,1) {$\cdots$};
	\node[below] at (2,-0.6) {$(w_1)$}; \node[below] at (2,0) {$\uparrow$}; \node[below] at (2,1) {$0$};
	%right labels
	\node[right] at (7,2) {$0$};
	\node[right] at (7,3.5) {$\vdots$};
	\node[right] at (7,4.5) {$\vdots$};
	\node[right] at (7,6) {$0$};
}
\ee
or, alternatively,
\be
Z^\rho_\pi(\bw):=\langle 1,2,\dots,k|\ \C_{\pi(k)}(w_{\rho(k)}|\bw)\dots\C_{\pi(1)}(w_{\rho(1)}|\bw)\ |0,0,\dots, 0\rangle.
\ee

\begin{prop}
\label{kappaZ}
For any permutations $\pi,\rho$ we have
\be
\kappa^\rho_\pi(\bw)=(-1)^{l(\pi)-l(\rho)}\prod_{a<b}\frac{w_b-qw_a}{w_b-w_a}\cdot Z^\rho_\pi(\bw).
\ee
\end{prop}
\begin{proof}
We start with $\pi=id$ case. It can be readily seen that in this case the partition function $Z_\pi^\rho(\bw)$ has only one configuration of paths satisfying the boundary conditions, and we obtain
\be
Z_{id}^\rho(\bw)=\prod_{i}\frac{1-q}{w_{\rho(i)}-qw_i}\prod_{a<b}\frac{w_{\rho(b)}-w_a}{w_{\rho(b)}-qw_a}.
\ee
Note that $Z_{id}^\rho(\bw)=0=\kappa_{id}^\rho(\bw)$ for all permutations $\rho\neq id$, while for $\rho=id$ we have
\be
\prod_{a<b}\frac{w_b-qw_a}{w_b-w_a}\cdot Z^{id}_{id}(\bw)=\prod_{a<b}\frac{w_b-qw_a}{w_b-w_a}\cdot\prod_{a<b}\frac{w_{b}-w_a}{w_{b}-qw_a}=1=\kappa^{id}_{id}(\bw).
\ee

For general $\pi$ we use the recurrent relations for $Z^\rho_\pi(\bw)$ and $\kappa^\rho_\pi(\bw)$. Assume that $l(\pi)<l(\pi\sigma_i)$ for some $i$. Then $\pi(i)<\pi(i+1)$, and we can apply the exchange relation to rows $i$ and $i+1$ in the partition function $Z^\rho_\pi(\bw)$, setting $k_1=\pi(i), k_2=\pi(i+1), x_1=w_{\rho(i+1)}, x_2=w_{\rho(i)}$ in \eqref{exchange} to get
\be
Z^\rho_{\pi\sigma_i}(\bw)=\frac{w_{\rho(i)}-qw_{\rho(i+1)}}{w_{\rho(i)}-w_{\rho(i+1)}}Z^{\rho\sigma_i}_\pi(\bw)-\frac{w_{\rho(i+1)}(1-q)}{w_{\rho(i)}-w_{\rho(i+1)}}Z^{\rho}_\pi(\bw).
\ee
Rescaling by $(-1)^{l(\pi)-l(\rho)+1}\prod_{a<b}\frac{w_b-qw_a}{w_b-w_a}$ we get
\be
\tilde{Z}^\rho_{\pi\sigma_i}(\bw)=\frac{w_{\rho(i+1)}(1-q)}{w_{\rho(i)}-w_{\rho(i+1)}}\tilde{Z}^{\rho}_\pi(\bw)+\frac{w_{\rho(i)}-qw_{\rho(i+1)}}{w_{\rho(i)}-w_{\rho(i+1)}}\tilde{Z}^{\rho\sigma_i}_\pi(\bw),
\ee
where $\tilde Z^\rho_\pi(\bw)$ denotes the right-hand side of the claim. This relation coincides with the recurrent relation \eqref{localZ} for the left-hand side of the claim, hence we can finish the proof by induction.
\end{proof}

The following statement summarizes several other properties of the coefficients $\kappa^\rho_\pi(\bw)$.
\begin{prop}
\label{Zprop}

1) $\kappa^\rho_\pi(\bw)$ is a rational function with at most simple poles at $w_i=w_j$ for $i<j$ and no other singularities. In other words, $\prod_{i<j}(w_j-w_i) \kappa^\rho_\pi(\bw)$ is a polynomial in $w_1,\dots, w_k$.

2) If for $a<b$ we have $\rho^{-1}(a)>\rho^{-1}(b)$ then $\restr{\kappa^\rho_\pi(\bw)}{\substack{w_a=z\\w_b=qz}}=0$.

3)  $\restr{\kappa^\rho_\pi(\bw)}{w_a=0}=0$ if $\pi(\rho^{-1}(a))<a$. Similarly, $\restr{\kappa^\rho_\pi(\bw)}{w_a\to\infty}=0$ if $\pi(\rho^{-1}(a))>a$.
\end{prop}
\begin{proof}
1) First we will use the same argument as in Proposition \ref{kappaDef} to prove that $\kappa^\rho_\pi(\bw)$ is a rational function with possible poles at $w_i=w_j$ and no other singularities. Note that all coefficients in  \eqref{defT} are rational functions with possible poles at $w_i=w_j$ and no other singularities. Hence in \eqref{reducedT} all coefficients $\A_{j_1,\dots j_s}(\bw)$ are rational functions satisfying the same restriction on  singularities. Hence $\kappa^\rho_\pi(\bw)$ satisfies this property as well. 

To prove that $\kappa^{\rho}_\pi(\bw)$ can have only simple poles at $w_i=w_j$ observe that Proposition \ref{kappaZ} implies
\be
\prod_{i<j}(w_j-w_i) \kappa^\rho_\pi(\bw)=(-1)^{l(\pi)-l(\rho)}\prod_{i<j}(w_j-qw_i) Z^\rho_\pi(\bw).
\ee
The vertex weights participating in the partition function $Z^\rho_\pi(\bw)$ only can have singularities at $w_i=qw_j$, hence the righthand side of the equation above has no singularities at $w_i=w_j$. Thus $\prod_{i<j}(w_j-w_i) \kappa^\rho_\pi(\bw)$ also has no singularities at $w_i=w_j$ and we are done.

2) We use the recurrent relation \eqref{localZ} and induction on $l(\pi)$. For $\pi=id$ we have $\kappa^\rho_{id}(\bw)=0$ for all $\rho\neq id$, hence, if $\rho^{-1}(a)>\rho^{-1}(b)$ for some $a<b$ then $\kappa^{\rho}_{id}(\bw)=0$ and the claim follows.

To prove the step, assume that we have proved the claim for $\pi$ and we want to prove it for $\pi\sigma_{i}$ with $l(\pi)<l(\pi\sigma_i)$. Fix $a<b$ and take any permutation $\rho$ such that $\rho^{-1}(a)>\rho^{-1}(b)$. Then by \eqref{localZ} and the inductive hypothesis
\begin{multline*}
\restr{\kappa^\rho_{\pi\sigma_{i}}(\bw)}{\substack{w_a=z\\ w_b=qz}}=\restr{\(\frac{(q-1)w_{\rho(i+1)}}{w_{\rho(i+1)}-w_{\rho(i)}} \kappa^\rho_{\pi}(\bw)\)}{\substack{w_a=z\\ w_b=qz}}+\restr{\(\frac{w_{\rho(i)}-qw_{\rho(i+1)}}{w_{\rho(i)}-w_{\rho(i+1)}}\kappa^{\rho\sigma_{i}}_{\pi}(\bw)\)}{\substack{w_a=z\\w_b=qz}}\\
=\restr{\(\frac{w_{\rho(i)}-qw_{\rho(i+1)}}{w_{\rho(i)}-w_{\rho(i+1)}}\kappa^{\rho\sigma_{i}}_{\pi}(\bw)\)}{\substack{w_a=z\\w_b=qz}}.
\end{multline*}
If $\sigma_i(\rho^{-1}(a))>\sigma_i(\rho^{-1}(b))$ then we are done because $\restr{\kappa^{\rho\sigma_{i}}_{\pi}(\bw)}{\substack{w_a=z\\w_b=qz}}=0$ by the inductive hypothesis. On the other hand, $\sigma_i(\rho^{-1}(a))<\sigma_i(\rho^{-1}(b))$ is possible only if $\rho^{-1}(a)=i+1$ and $\rho^{-1}(b)=i$. But in this case we have
\be
\restr{\(\frac{w_{\rho(i)}-qw_{\rho(i+1)}}{w_{\rho(i)}-w_{\rho(i+1)}}\kappa^{\rho\sigma_{i}}_{\pi}(\bw)\)}{\substack{w_a=z\\w_b=qz}}=\restr{\(\frac{w_{b}-qw_{a}}{w_{b}-w_{a}}\kappa^{\rho\sigma_{i}}_{\pi}(\bw)\)}{\substack{w_a=z\\w_b=qz}}=0.
\ee
Thus, the claim holds for $\pi\sigma_i$ and the step follows.

3) The proof is again by induction on $l(\pi)$. The base is trivial because $\kappa_{id}^\rho(\bw)=0$ unless $\rho=id$.

To prove the step, it is enough to show that the claim for $\pi$ implies the claim for $\pi\sigma_i$ if $l(\pi)<l(\pi\sigma_i)$. First consider the first part of the claim and take $\rho$ such that $\pi\sigma_i(\rho^{-1}(a))<a$. Then by \eqref{localZ} we have
\be
\restr{\kappa^\rho_{\pi\sigma_{i}}(\bw)}{w_a=0}=\restr{\(\frac{(q-1)w_{\rho(i+1)}}{w_{\rho(i+1)}-w_{\rho(i)}} \kappa^\rho_{\pi}(\bw)\)}{w_a=0}+\restr{\(\frac{w_{\rho(i)}-qw_{\rho(i+1)}}{w_{\rho(i)}-w_{\rho(i+1)}}\kappa^{\rho\sigma_{i}}_{\pi}(\bw)\)}{w_a=0}.
\ee
Note that $(\pi\sigma_i)\rho^{-1}=\pi(\rho\sigma_i)^{-1}$, so by the inductive hypothesis $\restr{\kappa^{\rho\sigma_{i}}_{\pi}(\bw)}{w_a=0}=0$ and it is enough to prove that 
\be
\restr{\(\frac{(q-1)w_{\rho(i+1)}}{w_{\rho(i+1)}-w_{\rho(i)}} \kappa^\rho_{\pi}(\bw)\)}{w_a=0}=0.
\ee

We have three cases. If $\rho^{-1}(a)\notin \{i,i+1\}$, then $\pi(\rho^{-1}(a))=\pi\sigma_i(\rho^{-1}(a))<a$ and we have $\restr{\kappa^{\rho}_{\pi}(\bw)}{w_a=0}=0$, finishing the proof in this case. If $\rho^{-1}(a)=i$ we have
\be
\pi(\rho^{-1}(a))=\pi(i)<\pi(i+1)=\pi\sigma_i(\rho^{-1}(a))<a,
\ee
where for the first inequality we have used that $l(\pi)<l(\pi\sigma_i)$. This implies $\restr{\kappa^{\rho}_{\pi}(\bw)}{w_a=0}=0$ and we are done. Finally, if $\rho^{-1}(a)=i+1$ then 
\be
\restr{\(\frac{(q-1)w_{\rho(i+1)}}{w_{\rho(i+1)}-w_{\rho(i)}} \kappa^\rho_{\pi}(\bw)\)}{w_a=0}=\restr{\(\frac{(q-1)w_{a}}{w_{a}-w_{\rho(i)}} \kappa^\rho_{\pi}(\bw)\)}{w_a=0}=0.
\ee

The proof of the inductive step for the second part of the claim is identical, with the only difference being the cases $\rho^{-1}(a)=i,i+1$. For the case $\rho^{-1}(a)=i$ note that
\be
\restr{\frac{(q-1)w_{\rho(i+1)}}{w_{\rho(i+1)}-w_a}}{w_a\to\infty}=0
\ee
and for the case $\rho^{-1}(a)=i+1$ we have
\be
\pi(\rho^{-1}(a))=\pi(i+1)>\pi(i)=\pi\sigma_i(\rho^{-1}(a))>a,
\ee
hence $\restr{\kappa^\rho_\pi(\bw)}{w_a\to\infty}=0$ by induction.
\end{proof}

\begin{rem}
\label{coefrem}
In the context of the stochastic colored six-vertex model, the coefficients $\kappa^\rho_\pi(\bz)$ have first appeared in \cite[Chapter 6]{BW18}. There they are described in terms of \emph{permutation graphs} or, equivalently, in terms of the \emph{SC6V model on an arbitrary wiring diagram}: instead of vertical and horizontal lines constituting a skew domain, one considers an arbitrary wiring diagram with intersections of lines serving as vertices. Similarly to the SC6V model on a skew domain, boundary conditions and rapidities assigned to each line define weights for possible configurations of the model. The coefficient $\kappa_\pi^\rho(\bz)$ corresponds to the partition function for a wiring digram with shape $\pi$, rapidities $\bz$ and boundary conditions determined by $\rho$. An exact matching of our notation with \cite{BW18} can be obtained by comparing Proposition \ref{kappaZ} with \cite[Proposition 6.7.6]{BW18}.
\end{rem}

\section{Iterated contour integrals involving the Hecke algebra action} 
\label{integrals}
For an ordered collection of contours $\Gamma=(\Gamma_1,\dots, \Gamma_k)$ and for functions $\Phi(\bw)=\Phi(w_1,\dots, w_k)$ and $\Psi(\bw)=\Psi(w_1,\dots, w_k)$ set
 \be
 \langle \Phi(\bw), \Psi(\bw) \rangle_\Gamma^k=\oint_{w_1\in\Gamma_1}\dots\oint_{w_k\in\Gamma_k}\prod_{a<b}\frac{w_b-w_a}{w_b-qw_a}\Phi(w_1, \dots, w_k)\Psi(w_1,\dots, w_k)\prod_{a=1}^k\frac{dw_a}{2\pi \i w_a}.
 \ee

In view of our main results, the expressions of the form $\langle T_\pi\Phi(\bw),\Psi(\bw)\rangle$ constitute one of the significant pieces of the current work. Working with such iterated integrals turns out to be challenging and our approach follows the following outline: In this section we directly compute $\langle T_\pi\Phi(\bw),\Psi(\bw)\rangle$ for a certain choice of $\Phi$ and $\Psi$ when each integral has one nonzero residue either outside or inside the contour. This computation serves as the initial point for the argument in the later sections, where we will reach the desired integral representations using recurrence relations and avoiding any additional computation of the whole iterated integral.

 We start by specifying the family of contours $\Gamma=(\Gamma_1,\dots, \Gamma_k)$ often used throughout the remainder of the text. Assume that $\bz=(\z_1,\dots,\z_N)$ is a family of parameters such that $q\z_i\neq\z_j$ for every $1\leq i,j\leq N$. Let $\Gamma[\bz^{-1}]$ be a union of small simple contours encircling $\z_i^{-1}$ for every $i$, while $0$ and the points $\{q^{-1}\z^{-1}_i\}_i$ are outside of $\Gamma[\bz^{-1}]$. Additionally we assume that the contours forming $\Gamma[\bz^{-1}]$ are small enough so that they do not encircle any points from $q^{\pm 1}\Gamma[\bz^{-1}]$. Let $c_0$ be a small circle around $0$ such that all the points  $\{\z_i^{-1}\}_i, \{q^{-1}\z^{-1}_i\}_i$ are outside of $c_0$. Define $\Gamma[k|\bz^{-1}]$ as a union of $\Gamma[\bz^{-1}]$ and $q^{2k}c_0$. For the rest of this section, unless stated otherwise, we set $\Gamma=(\Gamma[1|\bz^{-1}],\dots,\Gamma[k|\bz^{-1}])$.

 \begin{prop}[{\cite[Proposition 8.1.3]{BW18}}]
 \label{selfadjoint}
Assume that the functions $\Phi(\bw)$ and $\Psi(\bw)$ have no singularities at $w_\alpha=\lambda w_\beta$ for any fixed $\lambda$. Then for any $\pi\in S_k$ we have
 \be
  \langle T_\pi\Phi(\bw), \Psi(\bw) \rangle_\Gamma^k=\langle \Phi(\bw),T_{\pi^{-1}}\Psi(\bw) \rangle_\Gamma^k.
 \ee
 \end{prop}
 \begin{proof}
 It is enough to consider the case $\pi=\sigma_{[i,i+1]}$ for $1\leq i<k$. Recall that
 \be
 T_{\sigma_{[i,i+1]}}=T_i=\frac{(q-1)w_{i+1}}{w_{i+1}-w_i}+\frac{w_{i+1}-qw_i}{w_{i+1}-w_i}t_i.
 \ee
From the definition of $\langle\cdot,\cdot\rangle_\Gamma^k$ it is clear that
 \begin{equation}
 \label{selfadjpart1}
 \left \langle \frac{(q-1)w_{i+1}}{w_{i+1}-w_i} \Phi, \Psi \right\rangle_\Gamma^k=\left\langle \Phi, \frac{(q-1)w_{i+1}}{w_{i+1}-w_i}\Psi \right\rangle_\Gamma^k.
 \end{equation}
 On the other hand, note that the function
 \be
\(\prod_{a<b}\frac{w_b-w_a}{w_b-qw_a} \)\frac{w_{i+1}-qw_i}{w_{i+1}-w_i}\Phi(\sigma_i.\bw)\Psi(\bw)
 \ee
 has no singularities at $w_i=\lambda w_{i+1}$, so in the integral
  \be
\oint_{\Gamma[1|\bz^{-1}]}\dots\oint_{\Gamma[k|\bz^{-1}]}\(\prod_{\alpha<\beta}\frac{w_\beta-w_\alpha}{w_\beta-qw_\alpha}\)\frac{w_{i+1}-qw_i}{w_{i+1}-w_i}\Phi(\sigma_i.\bw)\Psi(\bw)\prod_{\alpha=1}^k\frac{dw_\alpha}{2\pi iw_\alpha}
 \ee
 we can exchange the contours $\Gamma[i|\bz^{-1}]$ and $\Gamma[i+1|\bz^{-1}]$ without crossing any singularities of the integrand. Changing variables $\tilde \bw:=\sigma_i. \bw$ we arrive at
 \begin{equation}
  \label{selfadjpart2}
 \left \langle \frac{w_{i+1}-qw_i}{w_{i+1}-w_i} t_i \Phi(\bw), \Psi(\bw) \right\rangle_\Gamma^k=\left \langle \frac{w_{i+1}-qw_i}{w_{i+1}-w_i} \Phi(\sigma_i.\bw), \Psi(\bw) \right\rangle_{\sigma_i\Gamma}^k=\left\langle \Phi(\widetilde{\bw}), \frac{\tilde w_{i+1}-q\tilde w_i}{\tilde w_{i+1}-\tilde w_i} t_i \Psi(\widetilde{\bw})\right\rangle_\Gamma^k.
 \end{equation}
The proof is finished by combining \eqref{selfadjpart1} and \eqref{selfadjpart2}.
 \end{proof}

To state the main result of this section we use the ramp function denoted by $\R$ and defined by
\be
\R(a):=\max(a,0).
\ee

\begin{prop} Let $(f_i)_{i=1}^k$ and $(l_i)_{i=1}^k$ be $k$-tuples of integers such that
\be
 f_1\geq f_2\geq \dots \geq f_k\geq 0, \qquad\qquad 0\leq l_1\leq l_2\leq\dots\leq l_k.
\ee 
Then for any permutation $\pi\in S_k$ we have
\label{base}
\begin{multline*}
q^{\R(f_{\pi(1)}-l_1)+\R(f_{\pi(2)}-l_2)+\dots+\R(f_{\pi(k)}-l_k)}= q^{\frac{k(k-1)}{2}-l(\pi)} \oint_{\Gamma[1|\bz^{-1}]}\cdots\oint_{\Gamma[k|\bz^{-1}]}\prod_{a<b}\frac{w_b-w_a}{w_b-qw_a}\\
T_\pi\left( \prod_{a=1}^k\prod_{j=1}^{l_{a}}\frac{1-\z_jw_a}{1-q\z_iw_a}\right)\cdot\prod_{a=1}^k\prod_{j=1}^{f_a}\frac{1-q\z_jw_a}{1-\z_jw_a}\prod_{a=1}^k\frac{dw_a}{2\pi i w_a}.
\end{multline*}
\end{prop}

The proof is based on two lemmas. Recall that $\sigma^{+}_{[a,b]}$ and $\sigma^{-}_{[a,b]}$ denote cycles in $S_k$, and we define the action of $S_k$ on $k$-tuples by
\be
\pi.\bw=\pi.(w_1,\dots, w_k)=(w_{\pi^{-1}(1)},\dots, w_{\pi^{-1}(k)}).
\ee 

\begin{lem}
\label{computation-inside}
Suppose that a collection of contours $\Gamma_1, \dots, \Gamma_k$  and rational functions 
\be
\Phi(\bw)=\Phi(w_1, \dots, w_k),\quad \Psi(\bw)=\Psi(w_1, \dots, w_k)
\ee
satisfy:

--For any fixed points $w_1,\dots, w_{k-1}$ on the contours $\Gamma_1, \dots, \Gamma_{k-1}$ respectively the function $\Psi(\bw)$ is holomorphic in a neighbourhood of the interior of $\Gamma_k$ as a function of $w_k$,

--Function $\Phi(\bw)$ is holomorphic in $w_1, \dots, w_k$ inside $U^k$ for a certain neighbourhood $U$ of the interior of $\Gamma_k$.

-- $\Gamma_k$ encircles $0$ and encircles no points from $q\Gamma_i$ for $1\leq i\leq k-1$.

Then for any $1\leq j<k$ we have
\begin{multline*}
\oint_{\Gamma_1}\cdots\oint_{\Gamma_k}\prod_{1\leq a<b\leq k}\frac{w_b-w_a}{w_b-qw_a}\(T_{\sigma^-_{[k-j,k]}}\Phi(\bw)\right)\Psi(\bw)\prod_{a=1}^k\frac{dw_a}{2\pi i w_a}\\
=q^{j-k+1}\oint_{\Gamma_1}\cdots\oint_{\Gamma_{k-1}}\prod_{1\leq a<b\leq k-1}\frac{w_b-w_a}{w_b-qw_a}\ \restr{\(\Phi({\sigma^+_{[k-j,k]}}. \bw)\Psi(\bw)\right)}{w_k=0}\prod_{a=1}^{k-1}\frac{dw_a}{2\pi i w_a}
\end{multline*}
or, in other words, we have
\be
\left\langle T_{\sigma^{-}_{[k-j,k]}}\Phi(\bw),\Psi(\bw)\right\rangle_\Gamma^k=q^{j-k+1} \left\langle \restr{\Phi(\sigma^+_{[k-j,k]}.\bw)}{w_k=0},\restr{\Psi(\bw)}{w_k=0}\right\rangle_{\tilde\Gamma}^{k-1},
\ee
where $\tilde\Gamma=(\Gamma_1,\dots, \Gamma_{k-1})$.
\end{lem}
\begin{proof}
We will take the integral with respect to $w_k$ in the left-hand side. Since $\Phi(\bw)$ is holomorphic for $w_1,\dots, w_k$ inside $\Gamma_k$, for any $\pi$ the function $T_\pi \Phi(\bw)$ has no singularities inside $\Gamma_k$ as a function of $w_k$. Hence the integrand, as a function of $w_k$, has only one non-vanishing residue which is at $w_k=0$.

To compute the residue at ${w_k=0}$ note that taking $w_{i+1}=0$ in $T_i$ we get
\be
\restr{T_{i}F(\bw)}{w_{i+1}=0}=\restr{(q+q(t_{i}-1))F(\bw)}{w_{i+1}=0}=qt_i\(\restr{F(\bw)}{w_{i}=0}\).
\ee
Repeating for $i=k-1,\dots, k-j$ we get
\begin{multline*}
\restr{T_{k-1}\dots T_{k-j}\Phi(\bw)}{w_k=0}=q\restr{t_{k-1}T_{k-2}\dots T_{k-j}\Phi(\bw)}{w_k=0}=\cdots=q^j\restr{t_{k-1}\dots t_{k-j}\Phi(\bw)}{w_k=0}\\
=q^j\restr{\Phi(\sigma^+_{[k-j,k]}.\bw)}{w_k=0}.
\end{multline*}
Additionally, we have
\be
\restr{\prod_{1\leq a<b\leq k}\frac{w_b-w_a}{w_b-qw_a}}{w_k=0}=q^{-k+1}\prod_{1\leq a< b\leq k-1}\frac{w_b-w_a}{w_b-qw_a}.
\ee
The claim follows by multiplying the last two equalities.
\end{proof}

\begin{lem}
\label{computation-outside}
Suppose that a pair of functions 
\be
\Phi(\bw)=\Phi(w_1, \dots, w_k),\quad \Psi(\bw)=\Psi(w_2, \dots, w_k)
\ee
satisfies:

-- The function $\Phi(\bw)$ does not depend on $w_1$.

-- The function $\Psi(\bw)$ is holomorphic in $U^k$ for a neighbourhood $U$ of the exterior of $\Gamma[\bz^{-1}]$. Moreover, $\Psi(\bw)$ has a finite limit at infinity with respect to any variable $w_i$.

-- For any $a\neq b$ the function $\Psi(\bw)$ has no singularity at $w_b=qw_a$, that is, $\restr{\Psi(\bw)}{w_b=qw_a}$ is a well-defined rational function.

-- For any $a< b$ and for any fixed $w_1, \dots, \widehat{w_a}, \dots, \widehat{w_b},\dots, w_k$ not in the interior of $\Gamma[\bz^{-1}]$ the function $\restr{\Psi(\bw)}{w_a=q^{-1}z, w_b=z}$ is holomorphic with respect to $z$ inside a neighbourhood of the interior of $\Gamma[\bz^{-1}]$.

Then for any $0\leq j<k$ we have
\begin{multline*}
\oint_{\Gamma[1|\bz^{-1}]}\cdots\oint_{\Gamma[k|\bz^{-1}]}\prod_{1\leq a<b\le k}\frac{w_b-w_a}{w_b-qw_a}\Phi(\bw)\(T_{\sigma^+_{[1,j]}}\Psi(\bw)\)\prod_{a=1}^k\frac{dw_a}{2\pi i w_a}\\
=q^{j-k+1}\oint_{\Gamma[2|\bz^{-1}]}\cdots\oint_{\Gamma[k|\bz^{-1}]}\prod_{2\leq a<b\le k}\frac{w_b-w_a}{w_b-qw_a}\restr{\(\Phi(\bw)\Psi(\sigma^-_{[1,j]}.\bw)\right)}{w_1\to\infty}\prod_{a=2}^{k}\frac{dw_a}{2\pi i w_a}
\end{multline*}
or, in other words, we have
\be
\left\langle \Phi(\bw),T_{\sigma^{+}_{[1,j]}}\Psi(\bw)\right\rangle_\Gamma^k=q^{j-k+1} \left\langle \restr{\Phi(\bw)}{w_1\to\infty},\restr{\Psi(\sigma^-_{[1,j]}.\bw)}{w_1\to\infty}\right\rangle_{\widetilde\Gamma}^{k-1},
\ee
where $\widetilde\Gamma=(\Gamma_2,\dots,\Gamma_k)$
\end{lem}
\begin{proof}
First we will prove the claim for $j=0$, that is, $T_{\sigma^+_{[1,j]}}=id$. Similarly to Lemma \ref{computation-inside} we are computing the integral with respect to $w_1$ in the left-hand side, taking residues outside of $\Gamma[1|\bz^{-1}]$ this time. We have a residue at $w_1=\infty$, which is equal to the right-hand side of the claim, so we need to prove that all other residues outside of $\Gamma[1|\bz^{-1}]$ give no contribution to the integral.

Note that due to the conditions on the functions, the only nonzero residues with respect to $w_1$ outside of $\Gamma[1|\bz^{-1}]$ are $\res_{w_1=q^{-1}w_a}$ for $a\geq 2$. Due to the geometry of the contours, $q^{-1}w_a$ is outside of $\Gamma[1|\bz^{-1}]$ only for $w_a\in\Gamma[\bz^{-1}]$, so the contribution of $\res_{w_1=q^{-1}w_a}$ is equal to
\be
(1-q^{-1})\oint\cdots\oint\prod_{\substack{2\leq b\leq k \\ b\neq a}}\frac{w_b-q^{-1}w_a}{w_b-w_a}\prod_{2\leq b<c\leq k}\frac{w_c-w_b}{w_c-qw_b}\Phi(\bw)\Psi\(\restr{\bw}{w_1=q^{-1}w_a}\)\prod_{b=2}^k\frac{dw_b}{2\pi i w_b},
\ee
where the integral with respect to $w_a$ is taken over the $\Gamma[\bz^{-1}]$ and for $b\neq a$ the integral with respect to $w_b$ is over $\Gamma[b|\bz^{-1}]$. By the last condition on $\Psi(\bw)$ applied to $w_1$ and $w_a$, the integral with respect to $w_a$ has no nonzero residues inside $\Gamma[\bz^{-1}]$, so the whole integral vanishes, as desired. Hence the contribution of $\res_{w_1=q^{-1}w_a}$ in the computation of the integral with respect to $w_1$ vanishes, which finishes the proof of the case $j=0$.

For general $j>0$ we use induction on $j$, with the base case $j=0$. To prove the step we claim that $T_i\Psi(\bw)$ also satisfies the conditions on the function $\Psi$ from the statement of the lemma. Indeed, the only  condition which does not readily follow is the last one: the function
\be
\restr{T_i\Psi(\bw)}{w_a=q^{-1}z, w_b=z}=\restr{\(q\Psi(\bw)+\frac{w_{i+1}-qw_i}{w_{i+1}-w_i}(t_i-1)\Psi(\bw)\)}{w_a=q^{-1}z, w_b=z}
\ee
is holomorphic inside $\Gamma[\bz^{-1}]$. To prove this last condition, note that it is enough to show that the function 
\be
\restr{\(\frac{w_{i+1}-qw_i}{w_{i+1}-w_i}(t_i-1)\Psi(\bw)\)}{w_a=q^{-1}z, w_b=z}=\restr{\(\frac{w_{i+1}-qw_i}{w_{i+1}-w_i}(\Psi(\sigma_i.\bw)-\Psi(\bw))\)}{w_a=q^{-1}z, w_b=z}
\ee
has no residues inside $\Gamma[\bz^{-1}]$. Due to the assumptions on $\Psi(\bw)$, the function $\restr{\Psi(\sigma_i.\bw)}{w_a=q^{-1}z, w_b=z}$ is holomorphic inside $\Gamma[\bz^{-1}]$ as long as $\sigma_i(a)<\sigma_i(b)$. Then if $w_a, w_b$ do not coincide with $w_i, w_{i+1}$  we are immediately done. Otherwise, say, we have $b=i$, $a\neq i+1$. Then both functions $\restr{\(\frac{1}{w_{i+1}-w_b}\Psi(\sigma_i.\bw)\)}{w_a=q^{-1}z, w_b=z}$ and $\restr{\(\frac{1}{w_{i+1}-w_b}\Psi(\bw)\)}{w_a=q^{-1}z, w_b=z}$ have no poles inside $\Gamma[\bz^{-1}]$, except for $z=w_{i+1}$. But for $z=w_{i+1}=w_i$ we have $\Psi(\sigma_i .\bw)-\Psi(\bw)=0$, so the residue at $z=w_{i+1}$ vanishes and we are done. Other cases when $\sigma_i(a)<\sigma_i(b)$ are treated similarly.

The only remaining case is $a=i, b=i+1$, when we have
\be
\restr{\(\frac{w_{i+1}-qw_i}{w_{i+1}-w_i}(t_i-1)\Psi(\bw)\)}{\substack{w_i=q^{-1}z,\\ w_{i+1}=z}}=\restr{\(\frac{z-z}{z-q^{-1}z}(t_i-1)\Psi(\bw)\)}{\substack{w_i=q^{-1}z,\\ w_{i+1}=z}} =0.
\ee
Thus the pair of functions $\Phi(\bw),T_i\Psi(\bw)$ satisfies the conditions of the lemma.

Finally, note that
\begin{multline*}
\restr{t_1\dots t_{j-1}T_j\Psi(\bw)}{w_1\to\infty}\\
=\restr{\(\(\frac{(q-1)w_{j+1}}{w_{j+1}-w_1}\)t_1\dots t_{j-1}\Psi(\bw)+\frac{w_{j+1}-qw_1}{w_{j+1}-w_1}t_1\dots t_j\Psi(\bw)\)}{w_1\to\infty}\\
=q\restr{t_1\dots t_j\Psi(\bw)}{w_1\to\infty},
\end{multline*}
so, replacing $\Psi(\bw)$ by $T_j\Psi(\bw)$ in the claim of the lemma for $j-1$, we get the claim of the lemma for $j$.
\end{proof}

\begin{proof}[Proof of Proposition \ref{base}]
We proceed by induction, with the base case $k=0$ being trivial. 

Denote the integral in the right-hand side of the claim of Proposition \ref{base} by 
\be
\mathcal I(\pi, f, l):=q^{\frac{k(k-1)}{2}-l(\pi)}\left\langle T_{\pi}\left( \prod_{a=1}^k\prod_{j\geq 1}^{l_{a}}\frac{1-\z_jw_a}{1-q\z_iw_a}\right), \(\prod_{a=1}^{k}\prod_{j\geq 1}^{f_a}\frac{1-q\z_jw_a}{1-\z_jw_a} \)\right\rangle_\Gamma^k.
\ee
 We have two cases:

\paragraph{\bfseries Case 1.} Assume first that $0<f_k<l_1$. Let $j:=\pi^{-1}(k)$ and $\tilde\pi=\pi\sigma^+_{[j,k]}$, so that $\tilde\pi(k)=k$. Moreover, $l(\pi)=l(\tilde\pi)+l(\sigma^-_{[j,k]})$, so $T_\pi=T_{\tilde\pi}T_{\sigma^-_{[j,k]}}$. Hence by Proposition \ref{selfadjoint} we have
\be
\mathcal I(\pi, f, l)=q^{\frac{k(k-1)}{2}-l(\pi)}\left\langle T_{\sigma^-_{[j,k]}}\left( \prod_{a=1}^k\prod_{j>f_k}^{l_{a}}\frac{1-\z_jw_a}{1-q\z_iw_a}\right), T_{\tilde\pi^{-1}}\(\prod_{a=1}^{k}\prod_{j>f_k}^{f_a}\frac{1-q\z_jw_a}{1-\z_jw_a} \)\right\rangle_\Gamma^k,
\ee
where we have multiplied the left and right parts of $\langle\cdot, \cdot\rangle$ by $\(\prod_{a}\prod_{j=1}^{f_k}\frac{1-q\z_jw_a}{1-\z_jw_a}\)^{\pm 1}$ (note that the multiplication by a factor symmetric in $\bw$ commutes with the action of $T_\pi$). 

Now we can apply Lemma \ref{computation-inside} for 
\be
\Phi(\bw)=\prod_{a=1}^k\prod_{j>f_k}^{l_{a}}\frac{1-\z_jw_a}{1-q\z_iw_a},\quad \Psi(\bw)=T_{\tilde\pi^{-1}}\(\prod_{a=1}^{k}\prod_{j>f_k}^{f_a}\frac{1-q\z_jw_a}{1-\z_jw_a}\).
\ee
Noting that 
\be
\restr{t_{\sigma^-_{[j,k]}}\prod_{a=1}^k\prod_{j>f_k}^{l_{a}}\frac{1-\z_jw_a}{1-q\z_iw_a}}{w_k=0}=\restr{\prod_{a=1}^k\prod_{j>f_k}^{l_{\sigma^+_{[j,k]}(a)}}\frac{1-\z_jw_{a}}{1-q\z_iw_{a}}}{w_k=0}=\prod_{a=1}^{k-1}\prod_{j>f_k}^{l_{{\sigma^+_{[j,k]}}(a)}}\frac{1-\z_jw_a}{1-q\z_iw_a},
\ee
Lemma \ref{computation-inside} gives
\begin{multline*}
\mathcal I(\pi, f, l)= q^{\frac{k(k-1)}{2}-l(\pi)-k+j+1} \left\langle \restr{t_{\sigma^-_{[j,k]}}\prod_{a=1}^k\prod_{j>f_k}^{l_{a}}\frac{1-\z_jw_a}{1-q\z_iw_a}}{w_k=0};\restr{T_{\tilde\pi^{-1}}\(\prod_{a=1}^{k}\prod_{j>f_k}^{f_a}\frac{1-q\z_jw_a}{1-\z_jw_a}\)}{w_k=0} \right\rangle_{\tilde\Gamma}^{k-1}\\
=q^{\frac{(k-1)(k-2)}{2}-l(\tilde\pi)} \left\langle \prod_{a=1}^{k-1}\prod_{j>f_k}^{l_{{\sigma^+_{[j,k]}}(a)}}\frac{1-\z_jw_a}{1-q\z_iw_a};T_{\tilde\pi^{-1}}\(\prod_{a=1}^{k-1}\prod_{j>f_k}^{f_a}\frac{1-q\z_jw_a}{1-\z_jw_a}\) \right\rangle_{\tilde\Gamma}^{k-1}=\mathcal I(\tilde\pi, \tilde f,  \tilde l),
\end{multline*}
where $\tilde \Gamma=(\Gamma[1|\bz^{-1}],\dots, \Gamma[k-1|\bz^{-1}])$ and
\be
\tilde f=(f_1,\dots f_{k-1}),\qquad \tilde l=\(l_{\sigma^+_{[j,n]}(1)},\dots l_{\sigma^+_{[j,n]}(k-1)}\).
\ee
Hence
\be
\mathcal I(\pi, f, l)=\mathcal I(\tilde\pi, \tilde f, \tilde l)=q^{\sum_{a=1}^{k-1} \R\(f_{\tilde\pi(a)}-l_{\sigma^+_{[j,k]}(a)}\)}=q^{\sum_{a\neq j}^k \R(f_{\pi(a)}-l_{a})}=q^{\sum_{a=1}^k \R(f_{\pi(a)}-l_{a})},
\ee
where in the last equality we have used $f_k<l_1\leq l_j$ to get $\R(f_{\pi(j)}-l_j)=\R(f_{k}-l_j)=0$.

\paragraph{\bfseries Case 2.} Assume now that $l_1\leq f_k$. Similarly to the previous case, let $\pi=\sigma^-_{[1,j]}\tilde\pi$ for $j:=\pi(1)$ and a permutation $\tilde\pi$ s.t. $\tilde\pi(1)=1$. Then we have
\be
\mathcal I(\pi, f, l)=q^{\frac{k(k-1)}{2}-l(\pi)}\left\langle T_{\tilde\pi}\left( \prod_{a=1}^k\prod_{j>l_1}^{l_{a}}\frac{1-\z_jw_a}{1-q\z_iw_a}\right), T_{\sigma^+_{[1,j]}}\(\prod_{a=1}^{k}\prod_{j>l_1}^{f_a}\frac{1-q\z_jw_a}{1-\z_jw_a}\) \right\rangle_\Gamma^k,
\ee
where we have used Proposition \ref{selfadjoint} and multiplied the left and right sides of $\langle\cdot, \cdot\rangle$ by $\(\prod_{a}\prod_{j=1}^{l_1}\frac{1-\z_jw_a}{1-q\z_jw_a}\)^{\pm 1}$. Now we can use Lemma \ref{computation-outside} for 
\be
\Phi(\bw)=T_{\tilde\pi}\left( \prod_{a=1}^k\prod_{j>l_1}^{l_{a}}\frac{1-\z_jw_a}{1-q\z_iw_a}\right),\quad \Psi(\bw)=\(\prod_{a=1}^{k}\prod_{j>l_1}^{f_a}\frac{1-q\z_jw_a}{1-\z_jw_a}\).
\ee
Note that 
\be
\restr{\prod_{j>l}^f\frac{1-q\z_jw}{1-\z_jw}}{w\to\infty}=q^{f-l},
\ee
so after applying Lemma \ref{computation-outside} we obtain
\begin{multline*}
\mathcal I(\pi, f, l)= q^{\frac{(k-1)(k-2)}{2}-l(\tilde\pi)} \left\langle T_{\tilde\pi}\left( \prod_{a=2}^{k}\prod_{j>l_1}^{l_{a}}\frac{1-\z_jw_a}{1-q\z_iw_a}\right), q^{f_j-l_1}\(\prod_{a=2}^{k}\prod_{j>l_1}^{f_{\sigma^-_{[1,j]}(a)}}\frac{1-q\z_jw_a}{1-\z_jw_a}\) \right\rangle_{\tilde\Gamma}^{k-1}\\
\\
=q^{f_j-l_1}\mathcal I(\tilde\pi, \tilde f,  \tilde l),
\end{multline*}
where $\tilde\Gamma=(\Gamma_2,\dots,\Gamma_k)$ and
\be
\tilde f=(f_{\sigma^-_{[1,j]}(2)},\dots f_{\sigma^-_{[1,j]}(k)}),\qquad\tilde l=(l_2,\dots,l_k).
\ee
Hence,
\be
\mathcal I(\pi, f, l)=q^{f_j-l_1}\mathcal I(\tilde\pi, \tilde f,  \tilde l)=q^{f_j-l_1+\sum_{\alpha=2}^{k} \R(f_{\sigma^-_{[1,j]}\tilde\pi(\alpha)}-l_{\alpha})}=q^{\sum_{\alpha=1}^k \R(f_{\pi(\alpha)}-l_{\alpha})}
\ee
where we have used $l_1\leq f_k\leq f_j$ to get $f_j-l_1=\R(f_j-l_1)$.
\end{proof}

 \begin{rem}
The contours $\Gamma_i$ are called \emph{$q$-nested} if the interior of the contour $\Gamma_i$ contains the contour $q^{-1}\Gamma_j$ for any $i<j$. In view of this terminology, we can refer to the family of contours $\{\Gamma[i|\bz^{-1}]\}_i$ as the family of sufficiently small contours encircling $\z_i^{-1}$ for all $i$, $q$-nested around $0$ and encircling no other singularity of an integrand. 
 
Instead of $q$-nestedness only around $0$, we could have considered the $q$-nested contours encircling $0$, $\bz^{-1}$ and no other singularities. All results above hold for this family of contours as well, moreover, some of the proofs become less involved. But, since the integrands we consider have singularities at points $q^{-1}\z^{-1}_i$, the existence and an explicit construction of a family of $q$-nested contours having $\z_i^{-1}$ inside and $q^{-1}\z_i^{-1}$ outside is much less trivial than for the contours $\Gamma[i|\bz^{-1}]$.
 \end{rem}

\section{Local relation}

\label{LocalRelationSec}

In this section we describe a certain recurrence relation between special observables of a SC6V model expressed in terms of the height functions. In the following text we extensively use the powers of $q$, so, to simplify such expressions, set
\be
\exp_q(x):=q^x.
\ee
We also use the following notation for the height functions:
\begin{equation}
\label{defH}
\H_{(c_1, c_2, \dots, c_k)}^{\nu_1\nu_2,\dots \nu_k}(\Sigma)=\H_{\c}^{(\A,\B)}(\Sigma):=h_{> c_1}^{(\alpha_1,\beta_1)}(\Sigma)+\dots+h_{> c_k}^{(\alpha_k,\beta_k)}(\Sigma),
\end{equation}
where $\Sigma$ is a configuration of a SC6V model, $\nu_1=(\alpha_1,\beta_1), \dots, \nu_k=(\alpha_k,\beta_k)\in\ddZ$ are points of the dual lattice and $\A=(\alpha_1,\dots, \alpha_k), \B=(\beta_1,\dots, \beta_k), \c=(c_1,\dots, c_k)$. Let $[\nu_1]^{a_1}\dots [\nu_t]^{a_t}$ denote the sequence of points starting with $\nu_1$ repeated $a_1$ times, then $\nu_2$ taken $a_2$ times and so on. 

The relation, which we call the \emph{local relation}, is essentially a statement about the values of the height functions around a single vertex. But, for the further use, we phrase it in terms of configurations on a skew domain of a SC6V model. Let $P-Q$ be a skew domain, and $v$ be a vertex inside it. Define $P_{\swarrow v}$ as the maximal up-left path $Q\leq P_{\swarrow v}\leq P$ such that $v$ is not inside the skew domain $P_{\swarrow v}-Q$. We can describe this path explicitly: assume that $P_i\in\ddZ$ are the points lying on the path $P$, so that
\be
P=P_0\to\dots\to P_l.
\ee
Let $P_i\to P_{i+1}$ be the step crossing the row of $v$, while $P_j\to P_{j+1}$ be the step of $P$ crossing the column containing $v$. Then the path $P_{\swarrow v}$ coincides with $P$ until it reaches $P_i$, then it goes to the left until it crosses the column containing $v$ and goes up until $P_{j+1}$, when it continues along $P$. Alternatively, we can characterize the path $P_{\swarrow v}$ by the skew domain $P_{\swarrow v}-Q$, which is obtained from $P-Q$ by keeping only the vertices below or to the left of $v$. 

For a configuration $\Sigma$ on a skew domain $P-Q$ and a vertex $v$ inside the skew domain $P-Q$ let $\Sigma_{\swarrow v}$ denote the restriction of the configuration $\Sigma$ to the skew domain $P_{\swarrow v}-Q$.
Alternatively, $\Sigma_{\swarrow v}$ is the restriction to the edges and vertices below or to the left of $v$.  %If $\Sigma$ is distributed according to a SC6V model $\M(Q,P,\c,\bz)$, then, due to the stochasticity of the vertex weights, the restriction $\Sigma_{\swarrow}$ is distributed according to the model with the same parameters $\M(Q,P_{\swarrow},\c,\bz)$. 

\begin{prop} 
\label{pLocalRelation} 
Consider a SC6V model $\M$ on a skew domain $P-Q$.  Let $v$ be a vertex inside $P-Q$ with a spectral parameter $z$. Then for any collection of colors $c_1\leq\dots\leq c_r$ and any configuration $\Sigma^0$ of $\M$ we have
\begin{multline}
\label{LocalRelation}
\E\[\exp_q\({\H_{(c_1,\dots, c_r)}^{[v_{\mnnearrow}]^r}(\Sigma)}\)\Big|\Sigma_{\swarrow v}=\Sigma^0_{\swarrow v}\]=\frac{q-q^rz}{q-z}\exp_q\({\H_{(c_1,\dots, c_r)}^{[v_{\mnsearrow}]^r}(\Sigma^0)}\)\\
+\frac{qz-1}{q-z}\sum_{i=0}^{r-1}\exp_q\({i+\H_{(c_1,\dots, c_r)}^{[v_{\mnsearrow}]^iv_{\mnswarrow}[v_{\mnsearrow}]^{r-i+1}}(\Sigma^0)}\)
+\frac{1-z}{q-z}\sum_{i=0}^{r-1}\exp_q\({i+\H_{(c_1,\dots, c_r)}^{[v_{\mnsearrow}]^iv_{\mnnwarrow}[v_{\mnsearrow}]^{r-i+1}}(\Sigma^0)}\)
\end{multline}
where the expectation is taken with respect to the model $\M$ and for $v=(a,b)\in\Z^2$ we set 
\begin{align*} 
v_{\mnnearrow}=\(a+\frac{1}{2},b+\frac{1}{2}\),\quad v_{\mnsearrow}=\(a+\frac{1}{2},b-\frac{1}{2}\),\\
v_{\mnswarrow}=\(a-\frac{1}{2},b-\frac{1}{2}\),\quad v_{\mnnwarrow}=\(a-\frac{1}{2},b+\frac{1}{2}\).
\end{align*}
\end{prop}

\begin{proof}
Let $\Sigma$ be a configuration satisfying $\Sigma_{\swarrow v}=\Sigma^0_{\swarrow v}$. Note that the values of the height functions at the points $v_{\mnsearrow}, v_{\mnswarrow}, v_{\mnnwarrow}$ are determined by $\Sigma^0_{\swarrow v}$, while $h_{> c}^{v_{\mnnearrow}}(\Sigma)$ depends only on $\Sigma^0_{\swarrow v}$ and the configuration around the vertex $v$. More explicitly, let $i,j$ denote the colors of the incoming vertical and horizontal edges of $v$, while $k,l$ denote the colors of the outgoing edges in the configuration $\Sigma$. Note that both $i$ and $j$ are determined by $\Sigma^0_{\swarrow v}$. Then, the recurrence relations \eqref{localH1}, \eqref{localH2} give 
\begin{equation}
\label{hRelation}
h_{> c}^{v_{\mnnearrow}}(\Sigma)=h_{> c}^{v_\mnsearrow}(\Sigma) + \1_{l> c};\qquad h_{> c}^{v_{\mnswarrow}}(\Sigma)=h_{> c}^{v_\mnsearrow}(\Sigma) + \1_{i> c}; \qquad h_{> c}^{v_{\mnnwarrow}}(\Sigma)=h_{> c}^{v_\mnswarrow}(\Sigma) + \1_{j> c}.
\end{equation}

\begin{figure}
\label{localVertex}
\begin{tikzpicture}[scale=1.5,baseline={([yshift=0]current bounding box.center)},>=stealth]
\draw[lgray,line width=1.5pt, ->] (-1.3, 0) -- (1, 0);
\draw[lgray,line width=1.5pt, ->] (0, -1.3) -- (0, 1);

\draw[style=dual] (-1, 1.5) -- (-1, -1) -- (1.5, -1);

\node[] at (0.5,0.5) {\large $v_{\mnnearrow}$};
\node[] at (0.5,-0.5) {\large $v_{\mnsearrow}$};
\node[] at (-0.5,0.5) {\large $v_{\mnnwarrow}$};
\node[] at (-0.5,-0.5) {\large $v_{\mnswarrow}$};

\node[above] at (-0.8, 0) {\small $j$};
\node[left] at (0,-0.8) {\small $i$};
\node[above] at (0.9,0) {\small $l$};
\node[left] at (0,0.9) {\small $k$};

\node[] at (-1.5,-1) {\huge $\Sigma^0_{\swarrow v}$};

\end{tikzpicture}
\caption{The configuration around the vertex $v$.}
\end{figure}
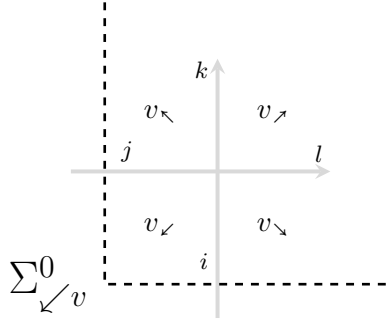

Dividing by $\exp_q\({\H_{(c_1,\dots, c_k)}^{[v_{\mnsearrow}]^r}(\Sigma^0)}\)$ and using relations \eqref{hRelation}, we reduce \eqref{LocalRelation} to
\begin{equation}
\label{reducedLocalRelation}
\E\(q^{\sum_{t=1}^r\1_{l> c_t}}|A_{(i,j)}\)=\frac{q-q^rz}{q-z}+\frac{qz-1}{q-z}\sum_{t=1}^{r}q^{t-1+\1_{i> c_t}}+\frac{1-z}{q-z}\sum_{t=1}^{r}q^{t-1+\1_{i> c_t}+\1_{j> c_t}},
\end{equation}
where $A_{(i,j)}$ denotes the condition that lattice edges entering $v$ have colors $i,j$, while $l$ denotes a random color of the lattice edge exiting $v$ to the right.

To prove \eqref{reducedLocalRelation} we use induction on $r$. The base case $r=0$ is trivial. For the step, note that it is enough to prove
\begin{equation}
\label{RelationStep}
\E\(q^{\sum_{t=1}^r\1_{l> c_t}}(q^{\1_{l> c_{r+1}}}-1)|A_{(i,j)}\)=\frac{(1-q)q^rz}{q-z}+\frac{qz-1}{q-z}q^{r+\1_{i> c_{r+1}}}+\frac{1-z}{q-z}q^{r+\1_{i> c_{r+1}}+\1_{j> c_{r+1}}}.
\end{equation}
The left-hand side is equal to
\be
\E\(q^{\sum_{t=1}^r\1_{l> c_t}}(q^{\1_{l> c_{r+1}}}-1)|A_{(i,j)}\)=q^r(q-1)\mathbb P(l> c_{r+1}|A_{(i,j)})
\ee
where we have used that $c_1\leq \dots\leq c_{r+1}$. To compute the probability $\mathbb P(l> c_{r+1}|A_{(i,j)})$ and to prove that \eqref{RelationStep} holds we have to consider four cases:

\emph{Case 1: $i,j\leq c_{r+1}$.} Then $\mathbb P(l> c_{r+1}|A_{(i,j)})=0$, and the right-hand side of \eqref{RelationStep} is equal to
\be
\frac{(1-q)q^rz}{q-z}+\frac{qz-1}{q-z}q^{r}+\frac{1-z}{q-z}q^{r}=0.
\ee

\emph{Case 2: $i\leq c_{r+1}< j$.} Then $l> c_{r+1}$ only if $j=l$. Using weights \eqref{weightsR} we obtain
\be
q^r(q-1)\mathbb P(l> c_{r+1}|A_{(i,j)})=q^r(q-1)\frac{1-z}{q-z},
\ee
while the right-hand side of \eqref{RelationStep} is equal to
\be
\frac{(1-q)q^rz}{q-z}+\frac{qz-1}{q-z}q^{r}+\frac{1-z}{q-z}q^{r+1}=q^r(q-1)\frac{1-z}{q-z}.
\ee

\emph{Case 3: $j\leq c_{r+1}< i$.} Here $l> c_{r+1}$ only if $i=l$. Using weights \eqref{weightsR} again we get
\be
q^r(q-1)\mathbb P(l> c_{r+1}|A_{(i,j)})=q^r(q-1)\frac{z(q-1)}{q-z},
\ee
while the right-hand side of \eqref{RelationStep} is equal to
\be
\frac{(1-q)q^rz}{q-z}+\frac{qz-1}{q-z}q^{r+1}+\frac{1-z}{q-z}q^{r+1}=q^r(q-1)\frac{z(q-1)}{q-z}.
\ee

\emph{Case 4: $c_{r+1}< i,j$.} Here $\mathbb P(l> c_{r+1}|A_{(i,j)})=1$, and the right-hand side of \eqref{RelationStep} is equal to
\be
\frac{(1-q)q^rz}{q-z}+\frac{qz-1}{q-z}q^{r+1}+\frac{1-z}{q-z}q^{r+2}=q^r(q-1).
\ee
\end{proof}

\begin{rem}
\label{rem:historyLocRel}
The $r=1$ case and the color-blind version of $r=2$ case of Proposition \ref{pLocalRelation} were proved previously in \cite[Theorem 3.1]{BG18} (a partial version of the $r=1$ case was also present in unpublished notes of M.~Wheeler). The full $r=2$ case was proved in \cite[Theorem 3.6]{BGW19}. Note that for fixed $q$ the knowledge of quantities $\mathbb{E} \left[ q^h \right]$ and $\mathbb{E} \left[ q^{2h} \right]$ is not enough for a complete control of the distribution of a random variable $h$ or its asymptotics (however, in the $q \to 1$ limit this information might allow to establish the law of large numbers in certain limit regimes, see \cite{BG18}). Proposition \ref{pLocalRelation} extends this relation to an arbitrary $r$, which allows to control $\mathbb{E} \left[ q^{rh} \right]$ through the $q$-moments of height fuctions in neigboring vertices. 
\end{rem}

For the later use we rewrite the local relation (actually, a weaker version of it) using different notation. As before, let $\{\e^i\}_{i=1}^k$ be the standard basis of $\mathbb Z^k$, with $\e^i$ being the vector with $1$ at $i$-th position and $0$ elsewhere. Set
\be
\e^{[a,b]}=\sum_{i=a}^{b} \e^i.
\ee
Recall that $\sigma^{\pm}_{[a,b]}$ denote cycles in $S_k$ and a permutation $\pi\in S_k$ is called $[a,b]$-ordered if \eqref{ordered} holds.

\begin{prop} 
\label{GlobalRelationProp}
Consider a SC6V model $\M$ with a skew domain $P-Q$. Let $\A=(\alpha_1,\dots, \alpha_k)\in\(\dZ\)^k$, $\B=(\beta_1,\dots, \beta_k)\in\(\dZ\)^k, \c=(c_1,\dots,c_k)\in\Z^k$ be $k$-tuples such that
\be
\alpha_1\leq \alpha_2\leq\dots\leq \alpha_{r-1}<\alpha=\alpha_r=\dots=\alpha_{r+t-1}<\alpha_{r+t}\leq\dots\leq \alpha_k,
\ee
\be
\beta_1\geq \beta_2\geq\dots\geq \beta_{r-1}>\beta=\beta_r=\dots=\beta_{r+t-1}>\beta_{r+t}\geq\dots\geq \beta_k,
\ee
\be
c_1\leq c_2\leq\dots\leq c_k
\ee
for some $r,t$. Assume that all points $(\alpha_i,\beta_i)$ are in the region $P-Q$, as well as the vertex $v=(\alpha-\frac{1}{2},\beta-\frac{1}{2})$. Let $z$ denote the spectral parameter of the vertex $v$. Then for any $[r, r+t-1]$-ordered permutation $\pi$ we have
\begin{multline}
\label{GlobalRelationS6V}
q^{l(\pi)}\E\[\exp_q\({\H_{\pi.\c}^{(\A,\B)}(\Sigma)}\)\]=\frac{q-q^{t}z}{q-z}q^{l(\pi)}\E\[\exp_q\({\H_{\pi.\c}^{\(\A,\B-\e^{[r,r+t-1]}\)}(\Sigma)}\)\]\\
+\sum_{i=0}^{t-1}\frac{qz-1}{q-z}q^{l(\sigma^+_{[r,r+i]}\pi)}\E\[\exp_q\({\displaystyle\H_{\sigma^+_{[r,r+i]}\pi.\c}^{\(\A-\e^r,\B-\e^{[r,r+t-1]}\)}(\Sigma)}\)\]\\
+\sum_{i=0}^{t-1}\frac{1-z}{q-z}q^{l(\sigma^+_{[r,r+i]}\pi)}\E\[\exp_q\({\displaystyle\H_{\sigma^+_{[r,r+i]}\pi.\c}^{\(\A-\e^r,\B-\e^{[r+1,r+t-1]}\)}(\Sigma)}\)\].
\end{multline}
\end{prop}
\begin{proof}
The claim follows directly from the local relation \eqref{LocalRelation} after a change of notation. Set
\be
\pi.\c=(\tilde c_1, \dots, \tilde c_k).
\ee 
\be
H^l(\Sigma)=\H^{(\alpha_1,\beta_1),\dots, (\alpha_{r-1}, \beta_{r-1})}_{(\tilde c_1, \dots, \tilde c_{r-1})}(\Sigma);\quad H^b(\Sigma)=\H^{(\alpha_{r+t},\beta_{r+t}),\dots, (\alpha_{k}, \beta_{k})}_{(\tilde c_{r+t}, \dots, \tilde c_{k})}(\Sigma).
\ee
Note that for any configuration $\Sigma$ both $H^l(\Sigma)$ and $H^b(\Sigma)$ are determined by $\Sigma_{\swarrow v}$, or, more precisely, $H^l(\Sigma)$ is determined by the configuration to the left of $v$, because $\alpha_1,\dots, \alpha_{r-1}<\alpha$, and $H^b(\Sigma)$ is determined by the configuration below $v$, because $\beta_{r+t}, \dots, \beta_k<\beta$. Hence
\be
\E\[\exp_q\({\displaystyle\H_{\pi.\c}^{(\A,\B)}(\Sigma)}\)\]=\E\[\exp_q\({H^l(\Sigma)+H^b(\Sigma)}\)\E\[\exp_q\({\H^{[v_{\mnnearrow}]^t}_{(\tilde c_{r},\dots, \tilde c_{r+t-1})}(\Sigma')}\)\big|\Sigma'_{\swarrow v}=\Sigma_{\swarrow v}\]\],
\ee
where we use the notation from Proposition \ref{pLocalRelation}. Since $\pi$ is $[r,r+t-1]$-ordered we have
\be
\tilde c_r=c_{\pi^{-1}(r)}\leq \tilde c_{r+1}=c_{\pi^{-1}(r+1)}\leq\dots\leq \tilde c_{r+t-1}=c_{\pi^{-1}(r+t-1)},
\ee
so we can apply the local relation \eqref{LocalRelation} to the conditional expectation inside obtaining a sum, which term-wise coincides with \eqref{GlobalRelationS6V}:
\be
\exp_q\({\displaystyle H^l+\H^{[v_{\mnsearrow}]^t}_{(\tilde c_{r},\dots, \tilde c_{r+t-1})}+H^b}\)=\exp_q\({\displaystyle \H_{\pi.\c}^{\(\A,\B-\e^{[r,r+t-1]}\)}}\),
\ee
\begin{multline*}
\exp_q\(i+ H^l+\H^{[v_{\mnsearrow}]^{i-1} v_{\mnswarrow} [v_{\mnsearrow}]^{t-i}}_{(\tilde c_{r},\dots, \tilde c_{r+t-1})}+H^b\)=q^{l(\sigma^+_{[r,r+i]}\pi)-l(\pi)}\exp_q\({\displaystyle\H_{\pi.\c}^{\(\A-\e^{r+i},\B-\e^{[r,r+t-1]}\)}}\)\\
=q^{l(\sigma^+_{[r,r+i]}\pi)-l(\pi)}\exp_q\({\displaystyle \H_{\pi.\c}^{\sigma^-_{[r,r+i]}\(\A-\e^{r},\B-\e^{[r,r+t-1]}\)}}\)=q^{l(\sigma^+_{[r,r+i]}\pi)-l(\pi)}\exp_q\({\displaystyle \H_{\sigma^+_{[r,r+i]}\pi.\c}^{\(\A-\e^{r},\B-\e^{[r,r+t-1]}\)}}\),
\end{multline*}
\be
\exp_q\(i+{\displaystyle H^l+\H^{[v_{\mnsearrow}]^{i-1} v_{\mnnwarrow} [v_{\mnsearrow}]^{t-i}}_{(\tilde c_{r},\dots, \tilde c_{r+t-1})}+H^b}\)=q^{l(\sigma^+_{[r,r+i]}\pi)-l(\pi)}\exp_q\({\displaystyle \H_{\sigma^+_{[r,r+i]}\pi.\c}^{\(\A-\e^{r},\B-\e^{[r+1,r+t-1]}\)}}\).
\ee
Note that for the last two equations we have used the relation $l(\sigma^+_{[r,r+i]}\pi)-l(\pi)=i$, which holds since $\pi$ is $[r,r+t-1]$-ordered.
\end{proof}

\begin{rem} One can perform a horizontal fusion in the local relation \eqref{LocalRelation} to get
\begin{multline}
\label{LocalRelationFused}
\E\[\exp_q\({\H_{(c_1,\dots, c_k)}^{[v_{\mnnearrow}]^k}(\Sigma)}\)\Big|\Sigma_{\swarrow v}=\Sigma^0_{\swarrow v}\]=\frac{1-q^ksu}{1-su}\exp_q\({\H_{(c_1,\dots, c_k)}^{[v_{\mnsearrow}]^k}(\Sigma^0)}\)\\
+\frac{-s^2+qsu}{1-su}\sum_{i=0}^{k-1}\exp_q\({i+\H_{(c_1,\dots, c_k)}^{[v_{\mnsearrow}]^iv_{\mnswarrow}[v_{\mnsearrow}]^{k-i+1}}(\Sigma^0)}\)
+\frac{s^2-su}{1-su}\sum_{i=0}^{k-1}\exp_q\({i+\H_{(c_1,\dots, c_k)}^{[v_{\mnsearrow}]^iv_{\mnnwarrow}[v_{\mnsearrow}]^{k-i+1}}(\Sigma^0)}\),
\end{multline}
where the vertex $v$ has the weights $L_{u}^{(s)}$ given in \eqref{Lweights}. It seems plausible that a similar relation holds in the fully fused model with weights $W^{(\NN,\MM)}_z$, but we do not know the exact form of it.
\end{rem}

\section{Integral formula}

In this section we prove the main result: the integral representation for $q$-moments of the height function. Consider a SC6V model $\mathcal M$ on a skew domain $P-Q$ with the endpoints of $P,Q$ being $\(M+\frac{1}{2}, \frac{1}{2}\)$ and $\(\frac{1}{2}, N+\frac{1}{2}\)$. 

Recall that we can identify the steps of $Q$ with the lattice edges entering the skew domain, and the boundary conditions of the model $\M$ are defined by a monotonic coloring of the steps of $Q$. We can describe this coloring as follows: since the colors of the steps monotonically increase, for each color there exists a unique point of $Q$ denoted by $\q(c)=(\gamma(c),\delta(c))\in\ddZ$ such that the colors of the steps of $Q$ before $\q(c)$ are $\leq c$, while the colors of the steps after $\q(c)$ are $> c$.

Recall from Definition \ref{skewDomainDef} that the row rapidities $(x_1,\dots, x_N)$ and the column rapidities $(y_1,\dots, y_M)$ can be encoded by  $\bz=(\z_1,\dots, \z_{N+M})$, where $\z_i=x_j$ if the $i$th step of $Q$ intersects the $j$th row, and $\z_i=y_j$ if  the $i$th step intersects the $j$th column. Assume that $\z_i\neq q\z_j$ for $1\leq i,j\leq N+M$. Let $\Gamma[i|\bz^{-1}]$ be the contours from Section~\ref{integrals}.
 
\begin{theo}
\label{mainTheorem} 
Fix a S6CV model on $P-Q$ with the coloring of $Q$ defined by $\q(c)$ and the rapidities of rows and columns defined by $\bz=(\z_1, \dots, \z_{N+M})$. For any collection $\{(\alpha_i,\beta_i)\}_{i=1}^k$ of points on $P$ such that
\label{S6Vqmoments}
\be
\alpha_1\geq\dots\geq\alpha_k,\quad \beta_1\geq\dots\geq\beta_k,\qquad \alpha_i,\beta_j\in\dZ,
\ee
 any $k$-tuple of colors $\c=(c_1,c_2,\dots, c_k)$ satsifying
\be
0\leq c_1\leq c_2\leq \dots\leq c_k, \qquad c_k\in\Z,
\ee
and any permutation $\pi\in S_k$ we have
\begin{multline}
\label{integralFormula}
\mathbb E\left[\exp_q\( \H^{(\A,\B)}_{\pi.\c}(\Sigma)\)\right] = q^{\frac{k(k-1)}{2}-l(\pi)} \oint_{\Gamma[1|\bz^{-1}]}\cdots\oint_{\Gamma[k|\bz^{-1}]}\prod_{a<b}\frac{w_b-w_a}{w_b-qw_a}\\
T_\pi\left( \prod_{a=1}^k\prod_{i=1}^{i<\delta(c_a)}\frac{1-x_{i}w_a}{1-qx_{i}w_a}\prod_{j>\gamma(c_a)}^{M}\frac{1-y_jw_a}{1-qy_jw_a}\right)\prod_{a=1}^k\left(\prod_{i=1}^{i<\beta_a}\frac{1-qx_{i}w_a}{1-x_{i}w_a}\prod_{j>\alpha_a}^{M}\frac{1-qy_jw_a}{1-y_jw_a}\frac{dw_a}{2\pi \i w_a}\right),
\end{multline}
where the products are taken over integers $i,j$, the integral with respect to $w_a$ is taken over $\Gamma[a|\bz^{-1}]$ and we set 
\be
\exp_q(x):=q^x,\qquad \H^{(\A,\B)}_{\pi.\c}(\Sigma)=h_{>c_{\pi^{-1}(1)}}^{(\alpha_1,\beta_1)}(\Sigma)+ h_{>c_{\pi^{-1}(2)}}^{(\alpha_2,\beta_2)}(\Sigma)+\dots + h_{>c_{\pi^{-1}(k)}}^{(\alpha_k,\beta_k)}(\Sigma).
\ee
\end{theo}
\begin{proof}
 We use induction on the number of vertices inside $P-Q$. 
 
 The base case $P=Q$ follows from Proposition \ref{base}. Let $Q_0,\dots, Q_{N+M}$ denote the points of the path $P=Q$, that is
 \be
 Q:Q_0\to Q_1\to\dots Q_{N+M},
 \ee
 and assume that $(\alpha_i,\beta_i)=Q_{f_i}$ and $(\gamma(c), \delta(c))=Q_{l(c)}$ for certain integers $f_i, l(c)$. Then the height function $h_{>c}^{(\alpha_i,\beta_i)}$ is deterministic and it is equal to $\R(f_i-l(c))$. Indeed, by definition $h_{>c}^{(\alpha_i,\beta_i)}$ is equal to the number of steps of $Q$ before $(\alpha_i,\beta_i)$ with the color $>c$. Hence, if $f_i<l(c)$ all the steps before $(\alpha_i,\beta_i)$ have color $\leq c$, contributing nothing to the height function. On the contrary, if $l(c)\leq f_i$, then there are $f_i-l(c)$ steps between $(\alpha_i,\beta_i)$ and $\q(c)$ contributing to the height function, which results in $h_{>c}^{(\alpha_i,\beta_i)}=f_i-l(c)$.
 
Note that
 \begin{align*}
 \prod_{i=1}^{i<\delta(c)}\frac{1-x_{i}w}{1-qx_{i}w}\prod_{j>\gamma(c)}^{M}\frac{1-y_jw}{1-qy_jw}=\prod_{i=1}^{l(c)}\frac{1-\z_iw}{1-q\z_iw};\\
  \prod_{i=1}^{i<\beta_i}\frac{1-qx_{i}w}{1-x_{i}w}\prod_{j>\alpha_i}^{M}\frac{1-qy_jw}{1-y_jw}=\prod_{i=1}^{f_i}\frac{1-q\z_iw}{1-\z_iw},
 \end{align*}
 so \eqref{S6Vqmoments} reduces to
 \begin{multline*}
\mathbb E\left(q^{\R(f_{\pi(1)}-l(c_1))+\dots+\R(f_{\pi(k)}-l(c_k))}\right) = q^{\frac{k(k-1)}{2}-l(\pi)} \oint_{\Gamma[1|\bz^{-1}]}\cdots\oint_{\Gamma[k|\bz^{-1}]}\prod_{a<b}\frac{w_b-w_a}{w_b-qw_a}\\
T_\pi\left( \prod_{a=1}^k\prod_{i=1}^{l(c)}\frac{1-\z_iw_a}{1-q\z_iw_a}\right)\prod_{a=1}^k\left(\prod_{i=1}^{f_a}\frac{1-q\z_iw_a}{1-\z_iw_a}\frac{dw_a}{2\pi \i w_a}\right).
\end{multline*}
This is exactly the claim of Proposition \ref{base}, so we are done with the proof of the inductive base.

The step of induction consists of changing one outer corner $(\alpha,\beta-1)\to(\alpha,\beta)\to(\alpha-1,\beta)$ of the path $P$ to the inner corner $(\alpha,\beta-1)\to(\alpha-1,\beta-1)\to(\alpha-1,\beta)$, removing one vertex $v=\(\alpha-\frac{1}{2},\beta-\frac{1}{2}\)$ from the skew domain. If $(\alpha_i,\beta_i)\neq(\alpha,\beta)$ for all $i$, the height function $\H_{\pi.\c}^{(\A,\B)}$ does not depend on the configuration around $v$ due to the stochasticity of the vertex weights, so such a change of the path $P$ does not affect both sides of \eqref{S6Vqmoments}.

In the other case, assume that $\alpha_r=\dots=\alpha_{r+t-1}=\alpha$ and $\beta_r=\dots=\beta_{r+t-1}=\beta$. Since $Q$ and $\mathbf c$ are fixed throughout the argument, it is more convenient to rewrite \eqref{integralFormula} as
\be
\E\(\exp_q\(\H_{\pi.\c}^{(\A,\B)}(\Sigma)\)\)=q^{\frac{k(k-1)}{2}-l(\pi)}\left\langle T_\pi\Phi(\bw), \Delta^{(\A,\B)}(\bw)\right\rangle^k_\Gamma
\ee
where  $\A=(\alpha_1,\dots,\alpha_k), \B=(\beta_1,\dots, \beta_k)$ and we set
\begin{align*}
\Phi(\bw):=\prod_{a=1}^k\prod_{i=1}^{i<\delta(c_a)}\frac{1-x_{i}w_a}{1-qx_{i}w_a}\prod_{j>\gamma(c_a)}^{M}\frac{1-y_jw_a}{1-qy_jw_a};\\
\Delta^{(\A,\B)}(\bw)=\prod_{a=1}^k\prod_{i=1}^{i<\beta_a}\frac{1-qx_{i}w_a}{1-x_{i}w_a}\prod_{j>\alpha_a}^{M}\frac{1-qy_jw_a}{1-y_jw_a}.
\end{align*}

For now assume that the permutation $\pi$ is $[r,r+t-1]$-ordered. Then, by the inductive hypothesis and Proposition \ref{GlobalRelationProp} we have
\begin{align}
\label{stepAfterRelation}
\begin{split}
q^{l(\pi)-\frac{k(k-1)}{2}}\E\(\exp_q\({\H_{\pi.\c}^{(\A,\B)}(\Sigma)}\)\)=&\frac{q-q^txy^{-1}}{q-xy^{-1}}\langle T_\pi\Phi(\bw), \Delta^{(\A,\B-\e^{[r,r+t-1]})}(\bw)\rangle_\Gamma\\
&+\sum_{i=0}^{t-1}\frac{qxy^{-1}-1}{q-xy^{-1}}\left\langle T_{\sigma^+_{[r,r+i]}\pi}\Phi(\bw),\Delta^{(\A-\e^r,\B-\e^{[r,r+t-1]})}(\bw)\right\rangle_\Gamma\\
&+\sum_{i=0}^{t-1}\frac{1-xy^{-1}}{q-xy^{-1}}\left\langle T_{\sigma^+_{[r,r+i]}\pi}\Phi(\bw),\Delta^{(\A-\e^r,\B-\e^{[r+1,r+t-1]})}(\bw)\right\rangle_\Gamma,
\end{split}
\end{align}
where we set $x=x_{\beta-1/2}, y=y_{\alpha-1/2}$. It can be readily verified by a direct computation that
\begin{multline*}
\frac{qxy^{-1}-1}{q-xy^{-1}}\Delta^{(\A-\e^r,\B-\e^{[r,r+t-1]})}(\bw)+\frac{1-xy^{-1}}{q-xy^{-1}}\Delta^{(\A-\e^r,\B-\e^{[r+1,r+t-1]})}(\bw)\\
=\(\frac{qxy^{-1}-1}{q-xy^{-1}}+\frac{1-xy^{-1}}{q-xy^{-1}}\frac{1-qxw_r}{1-xw_r}\)\Delta^{(\A-\e^r,\B-\e^{[r,r+t-1]})}(\bw)\\
=\frac{xy^{-1}(q-1)(1-qyw_r)}{(q-xy^{-1})(1-xw_r)}\Delta^{(\A,\B-\e^{[r,r+t-1]})}(\bw),
\end{multline*}
so we can combine the terms of the two sums in \eqref{stepAfterRelation} term-wise to obtain
\begin{multline*}
q^{l(\pi)-\frac{k(k-1)}{2}}\E\(\exp_q\(\H_{\pi.\c}^{(\A,\B)}(\Sigma)\)\)=\frac{q-q^txy^{-1}}{q-xy^{-1}}\langle T_\pi\Phi(\bw), \Delta^{(\A,\B-\e^{[r,r+t-1]})}(\bw)\rangle_\Gamma\\
+\sum_{i=0}^{t-1}\left\langle T_{\sigma^+_{[r,r+i]}\pi}\Phi(\bw),\frac{xy^{-1}(q-1)(1-qyw_r)}{(q-xy^{-1})(1-xw_r)}\Delta^{(\A,\B-\e^{[r,r+t-1]})}(\bw)\right\rangle_\Gamma.
\end{multline*}

Since $\pi$ is $[r,r+t-1]$-ordered we have $T_{\sigma^+_{[r,r+i]}\pi}=T_{\sigma^+_{[r,r+i]}}T_\pi$ for any $i=0,\dots, t-1$. Hence, using Proposition \ref{selfadjoint} and the fact that $\Delta^{(\A,\B-\e^{[r,r+t-1]})}(\bw)$ is symmetric with respect to $w_r,\dots, w_{r+t-1}$, we have
\begin{multline}
\label{beforeLem}
q^{l(\pi)-\frac{k(k-1)}{2}}\E\(\exp_q\(\H_{\pi.\c}^{(\A,\B)}(\Sigma)\)\)\\
=\left\langle T_\pi\Phi(\bw),\(\frac{q-q^txy^{-1}}{q-xy^{-1}}  +\sum_{i=0}^{t-1} T_{\sigma^-_{[r,r+i]}} \frac{xy^{-1}(q-1)(1-qyw_r)}{(q-xy^{-1})(1-xw_r)}\)\Delta^{(\A,\B-\e^{[r,r+t-1]})}(\bw)\right\rangle_\Gamma.
\end{multline}
\begin{lem}
\label{T-comp}
For any $\lambda,\mu$, and any $t\in\mathbb Z_{\geq 0}$ we have
\be
\frac{q-\lambda q^t}{q-\lambda}+\sum_{i=0}^{t-1}T_{\sigma^-_{[1,i+1]}}\frac{(q-1)(\lambda-q\mu w_1)}{(q-\lambda)(1-\mu w_1)}=\prod_{i=1}^t\frac{1-q\mu w_i}{1-\mu w_i}.
\ee
\end{lem}
\begin{proof}
Note that 
\be
T_1\(\frac{\eta w_1}{1-\eta w_1}\)=\(\frac{q\eta w_1}{1-\eta w_1}+\frac{w_2-qw_1}{w_2-w_1}\(\frac{\eta w_2}{1-\eta w_2}-\frac{\eta w_1}{1-\eta w_1}\)\)=\frac{\eta w_2}{1-\eta w_2}\frac{1-q\eta w_1}{1-\eta w_1}.
\ee
This identity can be iterated, leading to
\be
T_{\sigma^-_{[1,r]}}\(\frac{(1-q)\eta w_1}{1-\eta w_1}\)=T_{r-1}\dots T_1\(\frac{(1-q)\eta w_1}{1-\eta w_1}\)=\frac{(1-q)\eta w_{r}}{1-\eta w_r}\prod_{i=1}^{r-1}\frac{1-q\eta w_i}{1-\eta w_i}=\prod_{i=1}^{r}\frac{1-q\eta w_i}{1-\eta w_i}-\prod_{i=1}^{r-1}\frac{1-q\eta w_i}{1-\eta w_i}.
\ee
Hence, cancelling telescoping terms and using $T_i(1)=q$ we obtain
\begin{multline*}
\sum_{i=0}^{t-1} T_{\sigma^-_{[1,i+1]}}\frac{(q-1)(\lambda-q\eta w_1)}{(q-\lambda)(1-\eta w_1)}=\sum_{i=0}^{t-1} T_{\sigma^-_{[1,i+1]}}\(\frac{\lambda(q-1)}{q-\lambda} +\frac{(1-q)\eta w_1}{1-\eta w_1}\)\\
=\frac{q^t-1}{q-1}\cdot\frac{\lambda(q-1)}{q-\lambda} + \prod_{i=1}^{t}\frac{1-q\eta w_i}{1-\eta w_i} -1=\frac{\lambda q^t-q}{q-\lambda} + \prod_{i=1}^{t}\frac{1-q\eta w_i}{1-\eta w_i}.
\end{multline*}
\end{proof}

Setting $\lambda=xy^{-1}$ and $\eta=x$, Lemma \ref{T-comp} gives
\be
\frac{q-q^txy^{-1}}{q-xy^{-1}}+\sum_{i=0}^{t-1}T_{r+i-1}\dots T_{r}\frac{(q-1)xy^{-1}(1-qy w_r)}{(q-xy^{-1})(1-x w_r)}=\prod_{i=r}^{r+t-1}\frac{1-qx w_i}{1-x w_i}.
\ee
Plugging it into \eqref{beforeLem}, we get
\begin{multline*}
q^{l(\pi)-\frac{k(k-1)}{2}}\E\(\exp_q\({\H_{\pi\c}^{(\A,\B)}(\Sigma)}\)\)=\left\langle T_\pi\Phi(\bw), \prod_{i=r}^{r+t-1}\frac{1-qxw_i}{1-xw_i}\Delta^{(\A,\B-\e^{[r,r+t-1]})}(\bw)\right\rangle_\Gamma\\
=\left\langle T_\pi\Phi(\bw), \Delta^{(\A,\B)}(\bw)\right\rangle_\Gamma,
\end{multline*}
which proves the inductive step for a $[r,r+t-1]$-ordered permutation $\pi$.

To remove the restriction on the permutation $\pi$, note that any permutation can be uniquely written as $\pi=\tau\tilde\pi$, where $\tilde\pi$ is $[r,r+t-1]$-ordered and $\tau$ permutes only $\{r,\dots, r+t-1\}$. Then for the left-hand side of \eqref{integralFormula} we have
\be
\E\(\exp_q\(\H_{\pi.\c}^{(\A,\B)}(\Sigma)\)\)=\E\(\exp_q\(\H_{\tau\tilde\pi.\c}^{(\A,\B)}(\Sigma)\)\)=\E\(\exp_q\(\H_{\tilde\pi.\c}^{(\tau^{-1}.\A,\tau^{-1}.\B)}(\Sigma)\)\)=\E\(\exp_q\(\H_{\tilde\pi.\c}^{(\A,\B)}(\Sigma)\)\),
\ee
while for the right-hand side 
\begin{multline*}
q^{\frac{k(k-1)}{2}-l(\pi)} \left\langle T_{\tau\tilde\pi}\Phi(\bw);\Delta^{(\A,\B)}(\bw)\right\rangle_\Gamma^k=q^{\frac{k(k-1)}{2}-l(\tilde \pi)-l(\tau)} \left\langle T_{\tilde\pi}\Phi(\bw);T_{\tau^{-1}}\Delta^{(\A,\B)}(\bw)\right\rangle_\Gamma^k\\
=q^{\frac{k(k-1)}{2}-l(\tilde\pi)} \left\langle T_{\tilde\pi}\Phi(\bw);\Delta^{(\A,\B)}(\bw)\right\rangle_\Gamma^k,
\end{multline*}
where we have used that $\Delta^{(\A,\B)}(\bw)$ is symmetric in $w_r,\dots, w_{r+t-1}$ to deduce that $T_{\tau^{-1}}\Delta^{(\A,\B)}(\bw)=q^{l(\tau)}\Delta^{(\A,\B)}(\bw)$. Thus, the integral representation \eqref{integralFormula} for an arbitrary permutation $\pi$ follows from the integral representation for $\tilde\pi$, which is $[r,r+t-1]$-ordered.  
\end{proof}

%\begin{rem}
%Repeating the proof above verbatim, one can readily extend Theorem \ref{mainTheorem} to SC6V models on wiring diagrams from Remark \ref{coefrem}. The initial condition from Proposition \ref{base} is left unchanged, while the local relation can be rephrased in terms of wiring diagrams by considering a configuration around a single crossing on one of the ends of a wiring diagram instead of a vertex of a skew domain.
%\end{rem}

\section{Shift-invariance}
\label{shiftSec}
This section is devoted to a recently established symmetry of vertex models, called \emph{shift-invariance}: it turns out that multi-dimensional distributions of height functions are invariant under certain transformations. This symmetry was initially described in \cite{BGW19} and then it was generalized in \cite{Gal20} for a wider class of transformations. Here we give an alternative proof of some of the results from \cite{BGW19} and \cite{Gal20} by matching the integral representations from Theorem \ref{mainTheorem}.

To phrase the shift-invariance it is convenient to use additional notation related to the geometry of skew domains. For a couple of points $\p=(\alpha,\beta)\in\ddZ$ and $\q=(\gamma, \delta)\in\ddZ$ we write
\be
\q\nearrow\p,\quad \text{if}\ \gamma<\alpha\ \text{and}\ \delta<\beta; \quad\quad\quad\quad  \q\nwarrow\p,\quad \text{if}\ \gamma\geq\alpha\ \text{and}\ \delta\leq\beta;
\ee
\be
\q\searrow\p,\quad \text{if}\quad \gamma\leq\alpha\ \text{and}\ \delta\geq\beta.
\ee
Mnemonically, the notation above can be visualized as a relative position of the points $\p,\q$ on the dual lattice, with $\q$ and $\p$ corresponding to the beginning and the end of the arrow respectively. Note that the inequalities for $\q\nearrow\p$ are strict, while for $\q\nwarrow\p$  and $\q\searrow\p$ they are not.

The following notion of $(Q,P)$-cuts is borrowed from \cite{Gal20}. Let $Q\leq P$ be a pair of up-left paths defining a SC6V-model, with endpoints $\(M+\frac{1}{2},\frac{1}{2}\)$ and $\(\frac{1}{2},N+\frac{1}{2}\)$. For simplicity we assume that the coloring of $Q$ defining the boundary conditions is given by $(1,\dots, N+M)$. A \emph{$(Q,P)$-cut} $C=(\q, \p)$ is a pair of points on the dual lattice $\ddZ$ such that $\q\in Q$, $\p\in P$ and $\q\nearrow\p$. Our primary use of cuts is to visualize the data of the height function: for a $(Q,P)$-cut $C=(\q,\p)$ set
\be
h[C]=h[(\q,\p)]:=h^{\p}_{>c(\q)},
\ee
where $c(\q)$ is the unique color such that the point $\q\in Q$ is between incoming lattice edges with colors $c(\q)$ and $c(\q)+1$.

Two $(Q,P)$-cuts $C=(\q,\p)$ and $C'=(\q',\p')$ are called \emph{crossing} if
\be
\q\nwarrow\q', \p\searrow\p',\quad \text{or}\quad \q\searrow\q', \p\nwarrow\p'.
\ee
Otherwise the cuts are called \emph{non-crossing}. For non-crossing $(Q,P)$-cuts $C=(\q,\p)$ and $C'=(\q',\p')$ we write $ C'>C$ if $C'$ is strictly to the up-left from $C$, that is, $\q'\searrow\q$ and $\p'\searrow\p$.

For a cut $C=(\q, \p)$ with $\p=(\alpha,\beta)$ and $\q=(\gamma,\delta)$ define 
\be
\Row[C]=\Row[\q,\p]:=\left\{\delta+\frac{1}{2}, \dots, \beta-\frac{1}{2}\right\},\qquad\Col[C]=\Col[\q,\p]:=\left\{\gamma+\frac{1}{2}, \dots, \alpha-\frac{1}{2}\right\}.
\ee 
These are exactly rows and columns between the points $\q$ and $\p$, intersecting the cut $C=(\q,\p)$. Let $P-Q$ and $\widetilde P-\widetilde Q$ be a pair of skew domains with endpoints $\(\frac{1}{2},N+\frac{1}{2}\)$ and $\(M+\frac{1}{2}, \frac{1}{2}\)$.  We say that a collection of $(Q,P)$-cuts $(C_1,\dots, C_k)$ is \emph{shift-isomorphic} to a collection of $(\widetilde Q,\widetilde P)$-cuts $(\widetilde C_1, \dots, \widetilde C_k)$ if :
\begin{itemize}
\item For every $i,j$ we have $C_i> C_j$ if and only if $\widetilde C_i>\widetilde C_j$;
\item There exist bijections
\be
\phi:\{1,\dots,N\}\to\{1,\dots,N\} , \quad \psi:\{1,\dots,M\}\to \{1,\dots,M\}
\ee 
such that $\phi(\Row[C_i])=\Row[\widetilde C_i]$ and $\psi(\Col[C_i])=\Col(\widetilde C_i)$ for all $i=1,\dots, k$. The pair $(\phi,\psi)$ is called \emph{shift-isomorphism}.
\end{itemize}
 See Figure \ref{shiftInvarianceFig} for an example of shift-isomorphic cuts.

\begin{theo}
\label{shiftInv}
Let $\M$ and $\widetilde\M$ be SC6V models on skew domains $P-Q$ and $\widetilde P-\widetilde Q$ with endpoints $\(N+\frac{1}{2},\frac{1}{2}\)$ and $\(\frac{1}{2}, M+\frac{1}{2}\)$. Let $\bx=(x_1,\dots, x_N)$, $\by=(y_1,\dots, y_M)$ be respectively the row and column rapidities of $\M$, and similarly let $\widetilde{\bx}$ and $\widetilde{\by}$ denote the rapidities of $\widetilde\M$.  Let $C=(C_1,\dots, C_k)$ and $\widetilde C=(\widetilde C_1,\dots, \widetilde C_k)$ be collections of shft-isomorphic $(Q,P)$ and $(\widetilde Q,\widetilde P)$-cuts, with a shift-isomorphism given by permutations $\phi\in S_N$ and $\psi\in S_M$. Assume that $\widetilde{\bx}=\phi.\bx$ and $\widetilde{\by}=\psi.\by$. Then the random vector
\be
(h[C_1], h[C_2],\dots, h[C_k])
\ee
is equal in distribution to the random vector
\be
(h[\widetilde C_1], h[\widetilde C_2],\dots, h[\widetilde C_k]).
\ee
\end{theo}

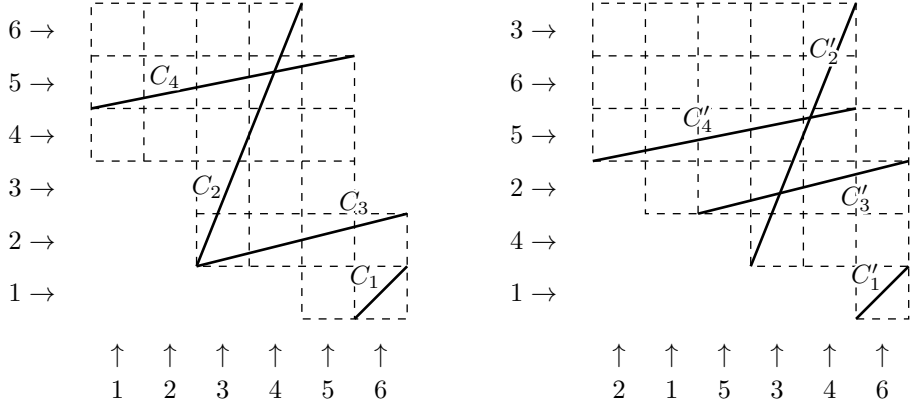
\begin{figure}
\begin{tikzpicture}[scale=0.7,baseline={([yshift=0]current bounding box.center)},>=stealth]

	\node[below] at (6,0.2) {$\uparrow$};
	\node[below] at (6,-0.5) {$6$};
	
	\node[below] at (5,0.2) {$\uparrow$};
	\node[below] at (5,-0.5) {$5$};

	\node[below] at (4,0.2) {$\uparrow$};
	\node[below] at (4,-0.5) {$4$};
	
	\node[below] at (3,0.2) {$\uparrow$};
	\node[below] at (3,-0.5) {$3$};
	
	\node[below] at (2,0.2) {$\uparrow$};
	\node[below] at (2,-0.5) {$2$};

	\node[below] at (1,0.2) {$\uparrow$};
	\node[below] at (1,-0.5) {$1$};
	
	\node[left] at (0,1) {$1 \rightarrow$};
	
	\node[left] at (0,2) {$2 \rightarrow$};
	
	\node[left] at (0,3) {$3 \rightarrow$};
	
	\node[left] at (0,4) {$4 \rightarrow$};
	
	\node[left] at (0,5) {$5 \rightarrow$};
		
	\node[left] at (0,6) {$6 \rightarrow$};
	
	\draw[style=lightdual] (1.5, 3.5) -- (1.5, 6.5);
	\draw[style=lightdual] (0.5, 4.5) -- (5.5, 4.5);
	\draw[style=lightdual] (0.5, 5.5) -- (5.5, 5.5) -- (5.5, 0.5);
	\draw[style=lightdual] (3.5, 1.5) -- (3.5, 6.5);
	\draw[style=lightdual] (2.5, 2.5) -- (6.5, 2.5) -- (6.5, 0.5) -- (4.5, 0.5) -- (4.5, 6.5) -- (0.5, 6.5) -- (0.5, 3.5) -- (5.5, 3.5);
	\draw[style=lightdual] (2.5, 6.5) -- (2.5, 1.5) -- (6.5, 1.5);
	
	\draw[style=cut] (0.5, 4.5) -- (5.5, 5.5);
	\draw[style=cut] (2.5, 1.5) -- (4.5, 6.5);
	\draw[style=cut] (2.5, 1.5) -- (6.5, 2.5);
	\draw[style=cut] (5.5, 0.5) -- (6.5, 1.5);
	
	\node at (5.7, 1.3) {\contour{white}{$C_1$}};
	\node at (2.7, 3) {\contour{white}{$C_2$}};
	\node at (5.5, 2.7) {\contour{white}{$C_3$}};
	\node at (1.9, 5.1) {\contour{white}{$C_4$}};
\end{tikzpicture}
\hspace{1cm}
\begin{tikzpicture}[scale=0.7,baseline={([yshift=0]current bounding box.center)},>=stealth]

	\node[below] at (6,0.2) {$\uparrow$};
	\node[below] at (6,-0.5) {$6$};
	
	\node[below] at (5,0.2) {$\uparrow$};
	\node[below] at (5,-0.5) {$4$};

	\node[below] at (4,0.2) {$\uparrow$};
	\node[below] at (4,-0.5) {$3$};
	
	\node[below] at (3,0.2) {$\uparrow$};
	\node[below] at (3,-0.5) {$5$};
	
	\node[below] at (2,0.2) {$\uparrow$};
	\node[below] at (2,-0.5) {$1$};

	\node[below] at (1,0.2) {$\uparrow$};
	\node[below] at (1,-0.5) {$2$};
	
	\node[left] at (0,1) {$1 \rightarrow$};
	
	\node[left] at (0,2) {$4 \rightarrow$};
	
	\node[left] at (0,3) {$2 \rightarrow$};
	
	\node[left] at (0,4) {$5 \rightarrow$};
	
	\node[left] at (0,5) {$6 \rightarrow$};
		
	\node[left] at (0,6) {$3 \rightarrow$};
	
	\draw[style=lightdual] (0.5, 4.5) -- (6.5, 4.5) -- (6.5, 0.5) -- (5.5, 0.5) -- (5.5, 6.5) -- (0.5, 6.5) -- (0.5, 3.5) -- (6.5, 3.5);
	\draw[style=lightdual] (0.5, 5.5) -- (5.5, 5.5);
	\draw[style=lightdual] (2.5, 2.5) -- (2.5, 6.5);
	\draw[style=lightdual] (4.5, 1.5) -- (4.5, 6.5);
	\draw[style=lightdual] (1.5, 6.5) -- (1.5, 2.5) -- (6.5, 2.5);
	\draw[style=lightdual] (3.5, 6.5) -- (3.5, 1.5) -- (6.5, 1.5);
		
	\draw[style=cut] (0.5, 3.5) -- (5.5, 4.5);
	\draw[style=cut] (3.5, 1.5) -- (5.5, 6.5);
	\draw[style=cut] (2.5, 2.5) -- (6.5, 3.5);
	\draw[style=cut] (5.5, 0.5) -- (6.5, 1.5);
	
	\node at (5.7,1.3) {\contour{white}{$C'_1$}};
	\node at (4.9, 5.6) {\contour{white}{$C'_2$}};
	\node at (5.5, 2.8) {\contour{white}{$C'_3$}};
	\node at (2.5, 4.3) {\contour{white}{$C'_4$}};
\end{tikzpicture}
\caption{\label{shiftInvarianceFig} A pair of shift-isomorphic collections of cuts. The maps $\phi$ and $\psi$ are given by renumbering of rows and columns, demonstrated on the right picture.}
\end{figure}

The proof is based on the following key lemma
\begin{lem}
\label{shiftInvarianceMatching}
In the setting of Theorem \ref{shiftInv}, we have
\be
\label{shiftInvarianceMatchingeq}
\E\(q^{h[C_1]+\dots+h[C_k]}\)=\E\(q^{h[\widetilde C_1]+\dots+h[\widetilde C_k]}\)
\ee
\end{lem}
First we will show how Theorem \ref{shiftInv} follows from Lemma \ref{shiftInvarianceMatching}, and then we will give the proof of Lemma \ref{shiftInvarianceMatching}.
\begin{proof}[Proof of Theorem]
Note that instead of proving the equality in distribution of the height functions we can equivalently prove that the following random vectors have the same distribution
\be
(q^{h[C_1]},\dots, q^{h[C_k]})=(q^{h[\widetilde C_1]},\dots, q^{h[\widetilde C_k]}).
\ee
Since $0<q<1$, it is enough to prove that the moments of these vectors coincide, namely, it is enough to check that for any $a_1,\dots,a_k\in\Z_{\geq 0}$ we have
\be
\E\(q^{a_1h[C_1]+\dots+a_kh[C_k]}\)=\E\(q^{a_1h[\widetilde C_1]+\dots+a_kh[\widetilde C_k]}\).
\ee
This is exactly the claim of Lemma \ref{shiftInvarianceMatching}, with the cuts $C_1$ and $\tilde C_1$ iterated for $a_1$ times, the cuts $C_2$ and $\tilde C_2$ iterated for $a_2$ time and so on.
 \end{proof}
 
 \begin{proof}[Proof of Lemma \ref{shiftInvarianceMatching}]
 We will prove that the integral expressions given by Theorem \ref{mainTheorem} are equal for the both sides of the equality in the lemma, using Proposition \ref{Zprop} and combinatorics of cuts. We start with the latter. 

Let $(C_1,\dots, C_k)$ be a collection of $(Q,P)$-cuts with $C_i=(\q^{pre}_i,\p^{pre}_i)$. We say that the collection $(C_1,\dots, C_k)$  is \emph{$Q$-ordered} if $\q^{pre}_1\nwarrow\dots\nwarrow\q^{pre}_k$ and for every $i<j$ such that $\q^{pre}_i=\q^{pre}_j$ we have $\p^{pre}_i\searrow\p^{pre}_j$. Similarly, the collection $(C_1,\dots, C_k)$ is inversely \emph{$P$-ordered} if $\p^{pre}_1\searrow\dots\searrow\p^{pre}_k$ and for every $i<j$ such that $\p^{pre}_i=\p^{pre}_j$ we have $\q^{pre}_i\nwarrow\q^{pre}_j$. \footnote{These definitions are motivated by the desire to reorder a sequence of cuts in a way such that the points $\q^{pre}_i$ go from the bottom-right endpoint of $Q$ in the up-left direction and the points $\p^{pre}_i$ go from the top-left endpoint of $P$ in the down-right direction, but since achieving this is impossible in general, we distinguish two extreme cases which separately prioritize the order of the points $\{\q^{pre}_i\}$ or the order of the points $\{\p^{pre}_i\}$.}

Let $(C_1,\dots, C_k)$ and $(\widetilde C_1,\dots, \widetilde C_k)$ be shift-isomorphic collections of $(Q,P)$ and $(\widetilde Q, \widetilde P)$-cuts, with shift-isomorphism given by $(\phi,\psi)$. 

\begin{lem}
\label{orders}
1) If $(C_1,\dots, C_k)$ is $Q$-ordered then  $(\widetilde C_1,\dots, \widetilde C_k)$ is also $\widetilde Q$-ordered.

2) If $(C_1,\dots, C_k)$ is inversely $P$-ordered then  $(\widetilde C_1,\dots, \widetilde C_k)$ is also inversely $\widetilde P$-ordered.
\end{lem}
\begin{proof}
We will prove only the first part, the proof of the second part is similar. Assume that the collection $(C_1,\dots, C_k)$ is $Q$-ordered and fix $i<j$. We have two cases:

If $\p^{pre}_i\searrow\p^{pre}_j$ then the cuts $C_i$ and $C_j$ are crossing and we have $\Row[C_j]\subset \Row[C_i]$ and $\Col[C_i]\subset \Col[C_j]$. Since $\phi$ and $\psi$ are bijections, we get $\Row[\widetilde C_j]\subset \Row[\widetilde C_i]$ and $\Col[\widetilde C_i]\subset \Col[\widetilde C_j]$, implying $\widetilde\q^{pre}_i\nwarrow\widetilde\q^{pre}_j$ and $\widetilde\p^{pre}_i\searrow\widetilde\p^{pre}_j$. In other words, $\widetilde\q^{pre}_i\nwarrow\widetilde \q^{pre}_j$ and if $\widetilde\q^{pre}_j=\widetilde\q^{pre}_i$ then $\widetilde\p^{pre}_i\searrow\widetilde\p^{pre}_j$.

On the other hand, if $\p^{pre}_i\nwarrow\p^{pre}_j, \p^{pre}_i=\p^{pre}_j$ we must have $\q^{pre}_i\neq\q^{pre}_j$ since the collection $(C_1,\dots, C_k)$ is $Q$-ordered. Then the cuts $C_i$ and $C_j$ are non-crossing and $C_j>C_i$. Hence $\widetilde C_j>\widetilde C_i$, so $\widetilde \q^{pre}_i\nwarrow\widetilde \q^{pre}_j$ with $\widetilde \q^{pre}_i\neq\widetilde \q^{pre}_j$.
\end{proof}

Note that simultaneous permutation of cuts in both collections $(C_1, \dots, C_k)$ and $(\widetilde C_1, \dots, \widetilde C_k)$ preserves shift-isomorphism and expectations in both sides of the lemma, so we can assume that the cuts $(C_1,\dots, C_k)$ are $Q$-ordered. Alternatively, we could have made this collection inversely $P$-ordered, so let $\pi$ be a permutation such that $(C_{\pi^{-1}(1)},\dots, C_{\pi^{-1}(k)})$ is an inversely $P$-ordered collection of cuts. Define
\be
\q_i:=\q_i^{pre},\qquad \p_{i}:=\p^{pre}_{\pi^{-1}(i)},
\ee 
so $C_i=(\q_i, \p_{\pi(i)})$. Then $\q_1,\dots, \q_k$ are ordered in the up-left direction along $Q$ and $\p_1.\dots, \p_k$ are ordered in the reverse down-right direction along $P$: $\q_1\nwarrow\dots\nwarrow\q_k$ and $\p_1\searrow\dots\searrow\p_k$. Similarly, we can define 
\be
\widetilde\q_i:=\widetilde\q_i^{pre},\qquad \widetilde\p_{i}:=\widetilde\p^{pre}_{\pi^{-1}(i)},
\ee 
for the cuts $(\widetilde C_1,\dots, \widetilde C_k)$. By Lemma \ref{orders} the points $\widetilde\q_i$ and $\widetilde\p_i$ are ordered in the same way as $\q_i,\p_i$.

Recall that since $C_{\pi^{-1}(i)}=(\q_{\pi^{-1}(i)},\p_{i})$ is a well-defined cut for any $i$, we have well-defined sets $\Col[{\q_{\pi^{-1}(i)}},{\p_i}]$, $\Row[{\q_{\pi^{-1}(i)}},\p_i]$, $\Col[{\widetilde{\q}_{\pi^{-1}(i)}},\widetilde\p_i],\Row[{\widetilde{\q}_{\pi^{-1}(i)}},\widetilde \p_i]$ and, since $(\phi,\psi)$ is a shift-isomorphism, 
\be
\phi(\Row[{\q_{\pi^{-1}(i)}},\p_i])=\Row[{\widetilde{\q}_{\pi^{-1}(i)}},\widetilde \p_i],\quad \psi(\Col[{\q_{\pi^{-1}(i)}},\p_i])=\Col[{\widetilde{\q}_{\pi^{-1}(i)}},\widetilde\p_i].
\ee
We can extend the equalities above as follows. Remember that $\preceq$ denotes the (strong) Bruhat order on $S_k$.
\begin{lem}
\label{combOfCuts}
For any $\rho\preceq\pi$ one of the following holds:

(i) For some $i$ we have $\q_{\rho^{-1}(i)}\searrow\p_{i}$ and  $\widetilde\q_{\rho^{-1}(i)}\searrow\widetilde\p_{i}$.

(ii) For some $i$ we have $\q_{\rho^{-1}(i)}\nwarrow\p_{i}$ and  $\widetilde\q_{\rho^{-1}(i)}\nwarrow\widetilde\p_{i}$.

(iii) For all $i$ we have $\q_{\rho^{-1}(i)}\nearrow\p_{i}$, $\widetilde\q_{\rho^{-1}(i)}\nearrow\widetilde\p_{i}$ and
\be
\phi(\Row[{\q_{\rho^{-1}(i)}},\p_i])=\Row[{\widetilde{\q}_{\rho^{-1}(i)}},\widetilde\p_i],\quad \psi(\Col[{\q_{\rho^{-1}(i)}},\p_i])=\Col[{\widetilde{\q}_{\rho^{-1}(i)}},\widetilde\p_i].
\ee
\end{lem}
\begin{proof}
We use downward induction on $\rho$, starting from $\rho=\pi$. As noted before, (iii) holds for $\rho=\pi$.

Recall that the Bruhat order can be defined by the covering relations $\rho\preceq\rho\cdot\tau_{a,b}$ with $l(\rho)<l(\rho\cdot\tau_{a,b})$, where $a<b$ and $\tau_{a,b}$ is the transposition exchanging $a$ and $b$. To prove the step of induction, it is enough to prove for $\rho\preceq\rho\cdot\tau_{a,b}$ that one of the conditions (i)-(iii) holds for $\rho$ if one of the conditions (i)-(iii) holds for $\rho\cdot\tau_{a,b}$.

First assume that the condition (i) holds for $\rho\cdot\tau_{a,b}$, that is, there exists $i$ such that $\q_{\tau_{a,b}\cdot\rho^{-1}(i)}\searrow\p_{i}$ and $\widetilde\q_{\tau_{a,b}\cdot\rho^{-1}(i)}\searrow\widetilde\p_{i}$. If $\rho^{-1}(i)\neq a,b$ then (i) holds for $\rho$ with the same $i$. If $\rho^{-1}(i)=a$ we have
\be
\q_b=\q_{\tau_{a,b}\cdot\rho^{-1}(i)}\searrow\p_{i}=\p_{\rho(a)}\searrow\p_{\rho(b)}
\ee
where we have used $l(\rho)<l(\rho\cdot\tau_{a,b})$ to get $\rho(a)<\rho(b)$ and hence $\p_{\rho(a)}\searrow\p_{\rho(b)}$. In a similar way one can show that $\widetilde\q_b\searrow\widetilde\p_{\rho(b)}$, so (i) holds for $\rho$ with $i$ replaced by $\rho(b)$. If $\rho^{-1}(i)=b$ then we have
\be
\q_b\searrow\q_a=\q_{\tau_{a,b}\cdot\rho^{-1}(i)}\searrow\p_{i}=\p_{\rho(b)}
\ee
and similarly $\widetilde\q_b\searrow\widetilde\p_{\rho(b)}$, so (i) holds for $\rho$ and the same $i$. Thus, if (i) holds for $\rho\cdot\tau_{a,b}$, then (i) also holds for $\rho$.

Similarly if (ii) holds for $\rho\cdot\tau_{a,b}$ then (ii) holds for $\rho$. The proof is analogous to the proof for the condition (i), so we omit it.

Finally, assume that (iii) holds for $\rho\cdot\tau_{a,b}$. Then clearly for all $i\neq \rho(a),\rho(b)$ we have 
\be 
\q_{\rho^{-1}(i)}=\q_{\tau_{a,b}\cdot\rho^{-1}(i)}\nearrow\p_{i},\qquad\widetilde\q_{\rho^{-1}(i)}=\widetilde\q_{\tau_{a,b}\cdot\rho^{-1}(i)}\nearrow\widetilde\p_{i},
\ee
\be
\phi(\Row[{\q_{\rho^{-1}(i)}},\p_i])=\Row[{\widetilde{\q}_{\rho^{-1}(i)}},\widetilde\p_i],\quad \psi(\Col[{\q_{\rho^{-1}(i)}},\p_i])=\Col[{\widetilde{\q}_{\rho^{-1}(i)}},\widetilde\p_i].
\ee
So we only need to consider the pairs $(\q_{a},\p_{\rho(a)}),(\q_{b},\p_{\rho(b)}),(\widetilde\q_{a},\widetilde\p_{\rho(a)}),(\widetilde\q_{b},\widetilde\p_{\rho(b)})$. Here we have three subcases:

\emph{(iii).1: $\Row[\q_{a},\p_{\rho(b)}]\cap \Row[\q_{b},\p_{\rho(a)}]=\varnothing$}: This means that all the rows intersecting the cut $(\q_{a},\p_{\rho(b)})$ are strictly below the rows intersecting $(\q_{b},\p_{\rho(a)})$, forcing $\q_b\searrow\p_{\rho(b)}$. On the other hand, since (iii) holds for $\rho\cdot\tau_{a,b}$, we also have $\Row[\widetilde\q_{a},\widetilde\p_{\rho(b)}]\cap \Row[\widetilde\q_{b},\widetilde\p_{\rho(a)}]=\varnothing$, hence $\widetilde\q_b\searrow\widetilde\p_{\rho(b)}$. So, (i) holds for $\rho$.

\emph{(iii).2: $\Col[\q_{a},\p_{\rho(b)}]\cap \Col[\q_{b},\p_{\rho(a)}]=\varnothing$}: This case in analogous to the previous one: the columns intersecting the cut $(\q_{a},\p_{\rho(b)})$ are strictly to the right of the rows intersecting $(\q_{b},\p_{\rho(a)})$ forcing $\q_a\nwarrow\p_{\rho(a)}$. By (iii) for $\rho\cdot\tau_{a,b}$ we also have $\widetilde\q_a\nwarrow\widetilde\p_{\rho(a)}$, hence (ii) holds for $\rho$.

\emph{(iii).3: The intersections in both subcases above are nonempty.} Then rows intersecting the cut $(\q_{a},\p_{\rho(b)})$ overlap with the rows of $(\q_{b},\p_{\rho(a)})$. Note that the rows $\Row[\q_{a},\p_{\rho(b)}]\cap \Row[\q_{b},\p_{\rho(a)}]$ are exactly the rows between $\q_{b}$ and $\p_{\rho(b)}$, with $\q_b$ being lower than $\p_{\rho(b)}$. On the other hand, one can immediately see that the rows between $\q_{a}$ and $\p_{\rho(a)}$ are exactly $\Row[\q_{a},\p_{\rho(b)}]\cup \Row[\q_{b},\p_{\rho(a)}]$, and $\q_a$ is lower that $\p_{\rho(a)}$. Repeating the same argument for columns, we get $\q_a\nearrow\p_{\rho(a)}$ and $\q_b\nearrow\p_{\rho(b)}$ with
\be
\Row[\q_a,\p_{\rho(a)}]=\Row[\q_{a},\p_{\rho(b)}]\cup \Row[\q_{b},\p_{\rho(a)}],\qquad \Col[\q_a,\p_{\rho(a)}]=\Col[\q_{a},\p_{\rho(b)}]\cap \Col[\q_{b},\p_{\rho(a)}],
\ee
\be
\Row[\q_b,\p_{\rho(b)}]=\Row[\q_{a},\p_{\rho(b)}]\cap \Row[\q_{b},\p_{\rho(a)}],\qquad \Col[\q_b,\p_{\rho(b)}]=\Col[\q_{a},\p_{\rho(b)}]\cup \Col[\q_{b},\p_{\rho(a)}].
\ee
Since (iii) holds for $\rho\cdot\tau_{a,b}$ we can repeat the argument for $\widetilde\q$ and $\widetilde\p$ to get $\widetilde\q_a\nearrow\widetilde\p_{\rho(a)}$, $\widetilde\q_b\nearrow\widetilde\p_{\rho(b)}$ and identical identities about unions and intersection. Hence
\be
\phi(\Row[\q_a,\p_{\rho(a)}])=\phi(\Row[\q_{a},\p_{\rho(b)}]\cup \Row[\q_{b},\p_{\rho(a)}])=\Row[\widetilde\q_{a},\widetilde\p_{\rho(b)}]\cup \Row[\widetilde\q_{b},\widetilde\p_{\rho(a)}]=\Row[\widetilde\q_a,\widetilde\p_{\rho(a)}],
\ee
and, in a similar fashion,
\be
\phi(\Row[\q_b,\p_{\rho(b)}])=\Row[\widetilde\q_b,\widetilde\p_{\rho(b)}],\qquad \psi(\Col[\q_a,\p_{\rho(a)}])=\Col[\widetilde\q_a,\widetilde\p_{\rho(a)}], \qquad \psi(\Col[\q_b,\p_{\rho(b)}])=\Col[\widetilde\q_b,\widetilde\p_{\rho(b)}].
\ee
Thus, (iii) holds for $\rho$.
\end{proof}

Now we are ready to prove Lemma \ref{shiftInvarianceMatching}. Similarly to the notation $\Delta^{(\A,\B)}(\bw)$ used in the proof of Theorem \ref{mainTheorem}, for points $\p=(\alpha,\beta)$ and $\q=(\gamma,\delta)$ we set
\be
\Delta^{\p}(w):=\prod_{i=1}^{i<\beta}\frac{1-qx_{i}w}{1-x_{i}w}\prod_{j>\alpha}^{M}\frac{1-qy_jw}{1-y_jw},
\ee
\be
\Delta^{-\p}(w):=\(\Delta^{\p}(w)\)^{-1}=\prod_{i=1}^{i<\beta}\frac{1-x_{i}w}{1-qx_{i}w}\prod_{j>\alpha}^{M}\frac{1-y_jw}{1-qy_jw},
\ee
\be
\Delta^{\p-\q}(w):=\Delta^{\p}(w)\Delta^{-\q}(w).
\ee
In the same way we define $\widetilde\Delta^{\widetilde\p-\widetilde\q}(w)$ by replacing the rapidities $x_i$ and $y_j$ with $\widetilde x_i$ and $\widetilde y_j$. Note that for $\q\searrow\p$ the function $\Delta^{\p-\q}(w)$ is singular only at $w=q^{-1}x^{-1}_i$ and $w=q^{-1}y^{-1}_i$, while for for $\q\nwarrow\p$ it is singular only at $w=x_i^{-1}$ and $w=y_i^{-1}$.

Remember that for a cut $C_i=(\q_{i},\p_{\pi(i)})$ we have $h[C]=h^{\p_{\pi(i)}}_{>c(\q)}$, so applying Theorem \ref{mainTheorem} to the points $(\alpha_i,\beta_i)=\p_i$ and the colors $c_i=c(\q_i)$ we get
\begin{multline*}
\E\[q^{\sum_i h[C_i]}\]= q^{\frac{k(k-1)}{2}-l(\pi)} \oint_{\Gamma[1|\bz^{-1}]}\cdots\oint_{\Gamma[k|\bz^{-1}]}\prod_{a<b}\frac{w_b-w_a}{w_b-qw_a}T_\pi\left( \prod_{a=1}^k\Delta^{-\q_a}(w_a)\right)\prod_{a=1}^k\Delta^{\p_a}(w_a)\frac{dw_a}{2\pi \i w_a}\\
=\sum_\rho q^{\frac{k(k-1)}{2}-l(\pi)} \oint_{\Gamma[1|\bz^{-1}]}\cdots\oint_{\Gamma[k|\bz^{-1}]}\prod_{a<b}\frac{w_b-w_a}{w_b-qw_a}\kappa_\pi^\rho(\bw) \prod_{a=1}^k\Delta^{\p_a-\q_{\rho^{-1}(a)}}(w_a)\frac{dw_a}{2\pi \i w_a},
\end{multline*}
where for the last equality we use Proposition \ref{kappaDef}. Similarly
\begin{equation*}
\E\[q^{\sum_i h[\widetilde C_i]}\]=\sum_\rho q^{\frac{k(k-1)}{2}-l(\pi)} \oint_{\Gamma[1|\bz^{-1}]}\cdots\oint_{\Gamma[k|\bz^{-1}]}\prod_{a<b}\frac{w_b-w_a}{w_b-qw_a}\kappa_\pi^\rho(\bw) \prod_{a=1}^k\widetilde\Delta^{\widetilde\p_a-\widetilde\q_{\rho^{-1}(a)}}(w_a)\frac{dw_a}{2\pi \i w_a}.
\end{equation*}
Taking the right-hand sides of the two equalities above, we reduce the statement of the lemma to proving the following identity for every $\rho\in S_k$:
\begin{multline}
\label{termMatching}
\oint_{\Gamma[1|\bz^{-1}]}\cdots\oint_{\Gamma[k|\bz^{-1}]}\prod_{a<b}\frac{w_b-w_a}{w_b-qw_a}\kappa_\pi^\rho(\bw) \prod_{a=1}^k\Delta^{\p_a-\q_{\rho^{-1}(a)}}(w_a)\frac{dw_a}{2\pi \i w_a}\\
=\oint_{\Gamma[1|\bz^{-1}]}\cdots\oint_{\Gamma[k|\bz^{-1}]}\prod_{a<b}\frac{w_b-w_a}{w_b-qw_a}\kappa_\pi^\rho(\bw) \prod_{a=1}^k\widetilde\Delta^{\widetilde\p_a-\widetilde\q_{\rho^{-1}(a)}}(w_a)\frac{dw_a}{2\pi \i w_a}
\end{multline} 
Note that by the first part of Proposition \ref{Zprop} both sides are well-defined and finite, and Proposition \ref{kappaDef} implies that it is enough to consider $\rho\preceq\pi$. By Lemma \ref{combOfCuts}, we have three possible cases:

1) \emph{For some $i$ we have $\q_{\rho^{-1}(i)}\searrow\p_{i}$ and  $\widetilde\q_{\rho^{-1}(i)}\searrow\widetilde\p_{i}$.} We claim that both sides of \eqref{termMatching} vanish in this case. We will consider only one side, the argument for the other side is identical. 

Take the maximal $i$ satisfying $\q_{\rho^{-1}(i)}\searrow \p_{i}$. To prove that
\be
\oint_{\Gamma[1|\bz^{-1}]}\cdots\oint_{\Gamma[k|\bz^{-1}]}\prod_{a<b}\frac{w_b-w_a}{w_b-qw_a}\kappa_\pi^\rho(\bw) \prod_{a=1}^k\Delta^{\p_a-\q_{\rho^{-1}(a)}}(w_a)\frac{dw_a}{2\pi \i w_a}=0
\ee
we are taking the integral with respect to $w_i$. Since $\q_{\rho^{-1}(i)}\searrow \p_{i}$, the function $\Delta^{p_i-q_{\rho^{-1}(i)}}(w_i)$ has no singularities inside $\Gamma[\bz^{-1}]$, hence the whole integral with respect to $w_i$ has no nonzero residues inside $\Gamma[\bz^{-1}]$. For any $i<j$ we have $\rho^{-1}(i)>\rho^{-1}(j)$, because otherwise $\q_{\rho^{-1}(j)}\searrow\q_{\rho^{-1}(i)}\searrow\p_i\searrow\p_j$, contradicting the maximality of $i$. Hence,  by the second part of Proposition \ref{Zprop}, the integrand has no singularity at $w_i=q^{-1}w_j$ for $i<j$, so we can shrink $\Gamma[i|\bz^{-1}]$ to a small contour around $0$ without crossing any poles. Finally, note that since $\q_a\nearrow\p_{\pi(a)}$ for all $a$, we have $\q_{\rho^{-1}(i)}\nearrow\p_{\pi(\rho^{-1}(i))}$. Combining with $\q_{\rho^{-1}(i)}\searrow \p_i$ we deduce that $\p_{\pi(\rho^{-1}(i))}$ is strictly above $\p_i$, hence $\pi(\rho^{-1}(i))<i$ and by the third part of Proposition \ref{Zprop} the residue at $0$ vanishes, so the whole integral vanishes and we are done.

2) \emph{For some $i$ we have $\q_{\rho^{-1}(i)}\nwarrow\p_{i}$ and  $\widetilde\q_{\rho^{-1}(i)}\nwarrow\widetilde\p_{i}$.} In this case both sides also vanish, but the argument is slightly more complicated. Again, it is enough to consider only one side, the argument for the other side is identical. Take the minimal $a$ satisfying $\q_{\rho^{-1}(a)}\nwarrow \p_{a}$. to prove that
\be
\oint_{\Gamma[1|\bz^{-1}]}\cdots\oint_{\Gamma[k|\bz^{-1}]}\prod_{a<b}\frac{w_b-w_a}{w_b-qw_a}\kappa_\pi^\rho(\bw) \prod_{a=1}^k\Delta^{\p_a-\q_{\rho^{-1}(a)}}(w_a)\frac{dw_a}{2\pi \i w_a}=0
\ee
we take the integral with respect to $w_a$, computing its residues outside of $\Gamma[a|\bz^{-1}]$. Outside of $\Gamma[a|\bz^{-1}]$ the integrand can have nonzero residues only at $w_a=q^{-1}\z_i^{-1}$, $w_a=\infty$, $w_a=qw_b$ for $b<a$ and $w_a=q^{-1}w_b$ for $b>a$. We will deal with these residues one by one. 

For the first three types of residues the argument is similar to the previous case. Note that since $\q_{\rho^{-1}(a)}\nwarrow \p_{a}$ the function $\Delta^{p_a-q_{\rho^{-1}(a)}}(w_a)$ is regular at $w_a=q^{-1}\z^{-1}_i$, so the integrand has no singularity at $w_a=q^{-1}\z^{-1}_i$ and the corresponding residue is $0$. Since $\q_{\rho^{-1}(a)}\nearrow\p_{\pi(\rho^{-1}(a))}$ and $\q_{\rho^{-1}(a)}\nwarrow \p_a$ we deduce that $\p_{\pi(\rho^{-1}(a))}$ is strictly below $\p_a$, hence $\pi(\rho^{-1}(a))>a$ and by Proposition \ref{Zprop} we have $\restr{\kappa^\rho_\pi(\bw)}{w_a\to\infty}=0$, so the residue $\res_{w_a=\infty}$ also vanishes. By minimality, for any $b<a$ we cannot have $\q_{\rho^{-1}(b)}\nwarrow\p_b$, which implies that $\q_{\rho^{-1}(b)}$ is to the left of $\q_{\rho^{-1}(a)}$, forcing $\rho^{-1}(a)<\rho^{-1}(b)$. By Proposition \ref{Zprop} we get $\restr{\kappa^\rho_\pi(\bw)}{\substack{w_a=z\\ w_b=q^{-1}z}}=0$ and the integrand has no singularity at $w_a=qw_b$.

We are left with residues at $w_a=q^{-1}w_b$ for $a<b$. Here we have two possibilities: if $\rho^{-1}(a)>\rho^{-1}(b)$ then by Proposition \ref{Zprop} we again have $\restr{\kappa^\rho_\pi(\bw)}{\substack{w_a=z\\ w_b=qz}}=0$, hence there is no singularity at $w_a=q^{-1}w_b$ and we are done.  If $\rho^{-1}(a)<\rho^{-1}(b)$, recall that the point $q^{-1}w_b$ is outside of $\Gamma[a|\bz^{-1}]$ only for $w_b\in \Gamma[\bz^{-1}]$, due to the geometry of contours around $0$ and around $\bz^{-1}$. So, taking the residue $\res_{w_a=q^{-1}w_b}$, we get the integral
\begin{equation}
\label{case2ofShiftInv}
\oint\cdots\oint\res_{w_a=q^{-1}w_b}\(\prod_{c<d}\frac{w_d-w_c}{w_d-qw_c}\kappa_\pi^\rho(\bw) \prod_{c=1}^k\Delta^{\p_c-\q_{\rho^{-1}(c)}}(w_c)\frac{1}{w_c}\)\prod_{c\neq a}\frac{dw_c}{2\pi\i},
\end{equation}
where $w_b$ is integrated over $\Gamma[\bz^{-1}]$. All singularities of the integral with respect to $w_b$ inside $\Gamma[\bz^{-1}]$ must come from the term $\Delta^{\p_b-\q_{\rho^{-1}(b)}}(w_b)\Delta^{\p_a-\q_{\rho^{-1}(a)}}(q^{-1}w_b)$ of the integrand. We claim that this term has no singularities inside $\Gamma[\bz^{-1}]$. Indeed, let $\p_a=(\alpha, \beta)$ and $\p_b=(\alpha',\beta')$. Since $a<b$, we have $\alpha<\alpha'$, $\beta>\beta'$. Then
\begin{multline*}
\Delta^{\p_a}(q^{-1}w_b)\Delta^{\p_b}(w_b)=\prod_{i<\beta}\frac{1-x_iw_b}{1-x_iq^{-1}w_b}\prod_{j>\alpha}\frac{1-y_jw_b}{1-y_jq^{-1}w_b}\prod_{i<\beta'}\frac{1-qx_iw_b}{1-x_iw_b}\prod_{j>\alpha'}\frac{1-qy_jw_b}{1-y_jw_b}\\
=\prod_{i>\beta'}^{i<\beta}(1-x_iw_b)\prod_{i<\beta'}(1-qx_iw_b)\prod_{i<\beta}\frac{1}{1-x_iq^{-1}w_b}\prod_{j>\alpha}^{j<\alpha'}(1-y_jw_b)\prod_{j>\alpha'}(1-qy_jw_b)\prod_{j>\alpha}\frac{1}{1-y_jq^{-1}w_b},
\end{multline*}
so $\Delta^{\p_a}(q^{-1}w_b)\Delta^{\p_b}(w_b)$ has no poles inside $\Gamma[\bz^{-1}]$. A similar computation for $\q_{\rho^{-1}(a)}$ and $\q_{\rho^{-1}(b)}$ yields that $\Delta^{-\q_{\rho^{-1}(a)}}(q^{-1}w_b)\Delta^{-\q_{\rho^{-1}(b)}}(w_b)$ also has no poles inside $\Gamma[z^{-1}]$. Hence the integral \eqref{case2ofShiftInv} taken with respect to $w_b$ has no singularities inside the contour of integration, so it vanishes, finishing the proof in this case.

3) \emph{For all $i$ we have $\q_{\tau^{-1}(i)}\nearrow\p_{i}$, $\widetilde\q_{\tau^{-1}(i)}\nearrow\widetilde\p_{i}$ and}
\be
\phi(\Row[{\q_{\tau^{-1}(i)}},\p_i])=\Row[{\widetilde{\q}_{\tau^{-1}(i)}},\widetilde\p_i],\quad \psi(\Col[{\q_{\tau^{-1}(i)}},\p_i])=\Col[{\widetilde{\q}_{\tau^{-1}(i)}},\widetilde\p_i].
\ee
Then integrands in both sides are equal. Indeed, note that for a cut $(\q,\p)$ we have
\be
\Delta^{\p-\q}(w)=\prod_{i\in \Row[\q,\p]}\frac{1-qx_iw}{1-x_iw}\prod_{j\in \Col[\q,\p]}\frac{1-y_jw}{1-qy_jw}.
\ee
Hence, if for cuts $(\q,\p)$ and $(\widetilde\q,\widetilde\p)$ we have
\be
\phi(\Row[\q,\p])=\Row[\widetilde\q,\widetilde\p],\quad \psi(\Col[\q,\p])=\Col[\widetilde\q,\widetilde\p]
\ee
then
\begin{multline*}
\widetilde\Delta^{\widetilde\p-\widetilde\q}(w)=\prod_{i\in \Row[\widetilde\q,\widetilde\p]}\frac{1-q\widetilde x_iw}{1-\widetilde x_iw}\prod_{j\in \Col[\widetilde\q,\widetilde\p]}\frac{1-\widetilde y_jw}{1-q\widetilde y_jw}\\
=\prod_{i\in \phi(\Row[\q,\p])}\frac{1-qx_{\phi^{-1}(i)}w}{1-\widetilde x_{\phi^{-1}(i)}w}\prod_{j\in \psi(\Col[\q,\p])}\frac{1-y_{\psi^{-1}(j)}w}{1-q y_{\psi^{-1}(j)}w}=\Delta^{\p-\q}(w)
\end{multline*}
Applying the argument above to the cuts $(\q_{\rho^{-1}(i)},\p_i)$ and $(\widetilde\q_{\rho^{-1}(i)},\widetilde\p_i)$, we get
\be
\prod_a\Delta^{\p_a-\q_{\rho^{-1}(a)}}(w_a)=\prod_a\widetilde\Delta^{\widetilde\p_a-\widetilde\q_{\rho^{-1}(a)}}(w_a),
\ee
hence integrals on both sides of \eqref{termMatching} are equal and we are done.
 \end{proof}

\begin{rem}
Theorem \ref{shiftInv} covers the shift-invariance results from \cite{BGW19} and, more generally, the conjectured shift-invariance \cite[Conjecture 1.5]{BGW19} proved in \cite{Gal20}. At the same time, our proof does not cover the \emph{flip invariance} discovered in \cite{Gal20}, as well as the wider class of symmetries following from the flip invariance. Since equality between $q$-moments implies equality of distributions, the results of \cite{Gal20} are equivalent to certain identities between the integral expressions from Theorem \ref{mainTheorem}. But the transformations from \cite{Gal20} can change the permutation $\pi$ from Theorem \ref{mainTheorem}, unlike the shift-invariance for which $\pi$ is left unchanged. Thus, to match the integral expressions one needs certain symmetry statements for the coefficients $\kappa_\pi^\rho(\bw)$ with varying permutations $\pi$, but we do not know how to reach such statements directly.
\end{rem}

\section{Fusion}
\label{fusionSection} In this section we perform stochastic fusion in Theorem \ref{S6Vqmoments} to get integral formulas for observables of more general vertex models.
\subsection{Vertical fusion.}
We start with the stochastic higher spin colored vertex model. Consider a vertex model on $\Z_{\geq 1}\times\Z_{\geq 1}$, with rapidities of unfused rows $\bu=(u_1,u_2,\dots)$, rapidities of fused columns $\mathbf y=(y_1,y_2,\dots)$ and spin parameters of columns $\mathbf s=(s_1,s_2\dots)$. Assume that there are no paths entering the quadrant from the bottom, that is, all color compositions for the bottom incoming edges are $0$. The colors of the unfused incoming horizontal edges at rows $l_{c-1}+1,\dots, l_c$ are equal to $c$, where $\{l_i\}_i$ is an arbitrary sequence of integers satisfying
\be
0=l_0\leq l_1\leq l_2\leq \dots.
\ee
This data defines a probability measure on the set of configurations on any finite rectangle with $N$ rows and $M$ columns, where the probability of a configuration is given by the weights $L_{u}^{(s)}$ from \eqref{Lweights}. To fix the values of the height functions, we set $h^{(\alpha,\frac{1}{2})}_{>c}=0$ for all colors $c=0,1,\dots,n$.

\begin{theo} 
\label{fusedMainTheorem}
Assume that the row rapidities $u_i$ satisfy $u_i\neq qu_j$ for all $i,j$. Then for any $k$-tuples $(\alpha_1, \dots, \alpha_k), (\beta_1, \dots, \beta_k), (c_1, \dots, c_k)$ satisfying
\be
0<\alpha_1\leq \alpha_2\leq \dots\leq \alpha_k,\quad \beta_1\geq \beta_2\geq \dots\geq \beta_k>0,\qquad \alpha_i,\beta_j\in\dZ,
\ee
\be
0\leq c_1\leq c_2\leq \dots\leq c_k,\qquad c_i\in\Z,
\ee
and any permutation $\pi\in S_k$ we have
\begin{multline}
\label{qmomentseq}
\mathbb E\left(\exp_q\({h_{>c_1}^{\(\alpha_{\pi(1)},\beta_{\pi(1)}\)}+\dots+h_{>c_k}^{\(\alpha_{\pi(k)},\beta_{\pi(k)}\)}}\)\right) = q^{\frac{k(k-1)}{2}-l(\pi)} \int_{\Gamma[1|\bu^{-1}]}\cdots\int_{\Gamma[k|\bu^{-1}]}\prod_{a<b}\frac{w_b-w_a}{w_b-qw_a}\\
T_\pi\left( \prod_{a=1}^k\prod_{i=1}^{l_{c_a}}\frac{1-u_iw_a}{1-qu_iw_a}\right)\prod_{a=1}^k\left(\prod_{i=1}^{i<\beta_a}\frac{1-qu_iw_a}{1-u_iw_a}\prod_{j=1}^{j<\alpha_a}\frac{s_j(w_a s_j-y^{-1}_j)}{w_a-s_jy^{-1}_j}\frac{dw_a}{2\pi i w_a}\right).
\end{multline}
where $\exp_q(x)=q^x$, the integral with respect to $w_a$ is taken over the contour $\Gamma[a|\bu^{-1}]$ and the contours $\Gamma[a|\bu^{-1}]$ are $q$-nested around $0$, encircle $\{u_i^{-1}\}_i$ and encircle no other singularities of the integrand, that is, $\{s_jy_j^{-1}\}_j$ and $\{q^{-1}u_i^{-1}\}_i$.
\end{theo}

In particular, a possible choice of contours can be constructed as follows. Let $\Gamma[\bu^{-1}]$ be a small contour encircling all points $\{u_i^{-1}\}_i$, with all points $0, \{q^{-1}u^{-1}_i\}_i, \{s_jy_j^{-1}\}_j$ being outside of $\Gamma[\bu^{-1}]$. Let $c_0$ be a small circle around $0$ such that all points  $\{u_i^{-1}\}_i, \{q^{-1}u^{-1}_i\}_i, \{s_jy_j^{-1}\}_j$ are outside of $c_0$. Define $\Gamma[k|\bu^{-1}]$ as a union of $\Gamma[\bu^{-1}]$ and $q^{2k}c_0$.

\begin{proof}
We start with the unfused case $s_j=q^{-1/2}$, which follows directly from Theorem \ref{S6Vqmoments}.

Assume that $\alpha_k<M$ and $\beta_1<N$, so we can restrict the model to an $N\times M$ rectangle. Set
\be
(\z_1,\dots, \z_{N+M})=(y_M,\dots, y_1,u_1,\dots, u_N)
\ee
and assume that $\z_i\neq q\z_j$ for all $i\neq j$. Define the path $Q$ by
\be
Q=\(\frac{1}{2};N+\frac{1}{2}\)\to\dots\to \(\frac{1}{2};\frac{1}{2}\)\to\dots\to\(M+\frac{1}{2};\frac{1}{2}\)
\ee  
and let $P$ be any up-left path passing through the points $(\alpha_i,\beta_i)$. Applying Theorem \ref{S6Vqmoments} we get
\begin{multline*}
\mathbb E\left(\exp_q\({\H^{(\A,\B)}_{\pi.\c}}\)\right) = q^{\frac{k(k-1)}{2}-l(\pi)} \oint_{\Gamma[1|\bz^{-1}]}\cdots\oint_{\Gamma[k|\bz^{-1}]}\prod_{a<b}\frac{w_b-w_a}{w_b-qw_a}\\
T_\pi\left( \prod_{a=1}^k\prod_{i=1}^{l_{c_a}}\frac{1-u_{i}w_a}{1-qu_{i}w_a}\prod_{j=1}^{M}\frac{1-y_jw_a}{1-qy_jw_a}\right)\prod_{a=1}^k\left(\prod_{i=1}^{i<\beta_a}\frac{1-qu_{i}w_a}{1-u_{i}w_a}\prod_{j>\alpha_a}^{M}\frac{1-qy_jw_a}{1-y_jw_a}\frac{dw_a}{2\pi \i w_a}\right),
\end{multline*}
where $\c=(c_1,\dots, c_k)$. Here we have used the description of the boundary condition to get $(\gamma(c),\delta(c))=(\frac{1}{2},l_c+\frac{1}{2})$. Note that we can move the product $\prod_a\prod_{j=1}^{M}\frac{1-y_jw_a}{1-qy_jw_a}$ outside of the action of $T_\pi$ to get 
\begin{multline}
\label{unfusedeq}
\mathbb E\left(\exp_q\({\H^{(\A,\B)}_{\pi.\c}}\)\right) = q^{\frac{k(k-1)}{2}-l(\pi)} \oint_{\Gamma[1|\bz^{-1}]}\cdots\oint_{\Gamma[k|\bz^{-1}]}\prod_{a<b}\frac{w_b-w_a}{w_b-qw_a}\\
T_\pi\left( \prod_{a=1}^k\prod_{i=1}^{l_{c_a}}\frac{1-u_iw_\alpha}{1-qu_iw_a}\right)\prod_{a=1}^k\left(\prod_{i=1}^{i<\beta_a}\frac{1-qu_{i}w_a}{1-u_{i}w_a}\prod_{j=1}^{j<\alpha_a}\frac{1-y_jw_a}{1-qy_jw_a}\frac{dw_a}{2\pi \i w_a}\right).
\end{multline}
which is the unfused case of \eqref{qmomentseq}.

Now, we would like to perform fusion. Following the procedure of the stochastic fusion, we replace a column with rapidity $y_j$ by $J$ columns with rapidities $q^{J-i}\tilde y$ for $i=1,\dots, J$, and then analytically continue $q^J\to s_j^{-2}, \tilde y\to s_jy_j$. After this transformation the left-hand side (with a suitable change of coordinates $\alpha_i$) concides the $q$-moments of the vertex model with a fused column having parameters $y_j,s_j$. On the other hand, in the integrand the terms $\frac{1-y_jw_a}{1-qy_jw_a}$ change to
\be
\restr{\prod_{i=1}^{J}\frac{1-q^{J-i}\tilde yw_a}{1-q^{J-i+1}\tilde yw_a}}{\substack{q^J=s_j^{-2}\\\tilde y=s_jy_j}}=\restr{\frac{1-\tilde y w_a}{1-q^{J}\tilde yw_a}}{\substack{q^J=s_j^{-2}\\\tilde y=s_jy_j}}=\frac{1-s_jy_jw_a}{1-s^{-1}_jy_jw_a}=\frac{s_j(s_jw_a-y_j^{-1})}{w_a-s_jy_j^{-1}}
\ee
So, after fusion the integrand coincides with the integrand in the right-hand side of \eqref{qmomentseq}. The only problem is the restriction $y_i\neq qy_j$, which was necessary for existence of the contours $\Gamma[i|\bz^{-1}]$ with $\{y_j^{-1}\}_j$ inside and $\{q^{-1}y_j^{-1}\}_j$ outside of the contour. 

To lift the restriction $y_i\neq qy_j$, note that the integrand in the right-hand side of \eqref{unfusedeq} has no singularity at $y^{-1}_j$. Recall that the integration contours in the right-hand side of  \eqref{unfusedeq} are $\Gamma[i|\bz^{-1}]$, which consist of small $q$-nested contours around $0$, and small coinciding contours around each $\z_i^{-1}$. But, since the integrand has no singularities at $y_j^{-1}$, we can remove the small contours around $\z_{N-j+1}^{-1}=y_j^{-1}$ for $j=1,\dots, N$ without changing the integral. What remains is a collection of small contours around $\{u_i^{-1}\}_i$, so the remaining contours are exactly contours $\Gamma[1|\bu^{-1}],\dots, \Gamma[k|\bu^{-1}]$.

Now we can write both sides of \eqref{unfusedeq} as a rational functions in $y_j$. The left-hand side is rational in $y_j$ by definition, while in the right-hand side we can compute the integral by taking residues inside the contours. Since these residues depend only on $u_i$, the result is a rational function in $y_j$. These functions are well-defined if $qy_j\neq u_i$ for all $i,j$, so we can extend \eqref{unfusedeq} to all $y_j\notin\{0\}\cup\{q^{-1}u_i\}_i$. In particular, we can have $y_i=qy_j$. So, we can perform finite fusion, and prove \eqref{qmomentseq} for $s_j=q^{-J_j}$.

After finite fusion, using the same argument as before, we can write both sides of \eqref{qmomentseq} as rational functions in $s_j,y_j$, which will be well-defined if $s^{-1}_jy_j\neq u_i$ for all $i,j$. Then, we can analytically continue the spin parameters to obtain \eqref{qmomentseq} in full generality. 
\end{proof}

Using the notation from Section \ref{coefficientsSection} and Proposition \ref{kappaDef} we can write an alternative form of Theorem \ref{fusedMainTheorem}.

\begin{cor} 
\label{fusedMainTheoremAlt} Using the same notation as Theorem \ref{fusedMainTheorem}, we have
\begin{multline*}
\mathbb E\left(\exp_q\({h_{>c_1}^{\(\alpha_{\pi(1)},\beta_{\pi(1)}\)}+\dots+h_{>c_k}^{\(\alpha_{\pi(k)},\beta_{\pi(k)}\)}}\)\right) = q^{\frac{k(k-1)}{2}-l(\pi)} \sum_{\rho\in S_k}\int_{\Gamma[1|\bu^{-1}]}\cdots\int_{\Gamma[k|\bu^{-1}]}\prod_{a<b}\frac{w_b-w_a}{w_b-qw_a}\\
\kappa_\pi^\rho(\bw)\left( \prod_{a=1}^k\prod_{i=1}^{l_{c_a}}\frac{1-u_iw_{\rho(a)}}{1-qu_iw_{\rho(a)}}\right)\prod_{a=1}^k\left(\prod_{i=1}^{i<\beta_a}\frac{1-qu_iw_a}{1-u_iw_a}\prod_{j=1}^{j<\alpha_a}\frac{s_j(w_a s_j-y^{-1}_j)}{w_a-s_jy^{-1}_j}\frac{dw_a}{2\pi i w_a}\right).
\end{multline*}
where $\kappa_\pi^\rho(\bw)$ are the coefficients defined by Proposition \ref{kappaDef}.
\end{cor}

When the total number of colors $n=1$ and $l_1=\infty$, we obtain an analogue of Theorem \ref{fusedMainTheorem} for the colorless stochastic higher spin vertex model:

\begin{cor}\label{oneColorMainTheorem} For any $k$-tuples $(\alpha_1, \dots, \alpha_k), (\beta_1, \dots, \beta_k)$ satisfying
\be
0<\alpha_1\leq \alpha_2\leq \dots\leq \alpha_k,\quad \beta_1\geq \beta_2\geq \dots\geq \beta_k>0,\qquad \alpha_i,\beta_j\in\dZ,
\ee
we have
\begin{multline*}
\mathbb E\left(\exp_q\({h_{>0}^{\(\alpha_{1},\beta_{1}\)}+\dots+h_{>0}^{\(\alpha_{k},\beta_{k}\)}}\)\right) = \\
q^{\frac{k(k-1)}{2}} \int_{\Gamma[1|\bu^{-1}]}\cdots\int_{\Gamma[k|\bu^{-1}]}\prod_{a<b}\frac{w_b-w_a}{w_b-qw_a}\prod_{a=1}^k\left(\prod_{i=1}^{i<\beta_a}\frac{1-qu_iw_a}{1-u_iw_a}\prod_{j=1}^{j<\alpha_a}\frac{s_j(w_a s_j-y^{-1}_j)}{w_a-s_jy^{-1}_j}\frac{dw_a}{2\pi i w_a}\right).
\end{multline*}
where $\exp_q(x)=q^x$ and the contours $\Gamma[i|\bu^{-1}]$ are $q$-nested around $0$, encircle $\{u_i^{-1}\}_i$ and they encircle no other singularity of the integrand.
\end{cor}
When all points $(\alpha_i, \beta_i)$ are on the same row, that is, $\beta_1=\dots=\beta_k$, Corollary \ref{oneColorMainTheorem} coincides with \cite[Theorem 9.8]{BP16}.

\subsection{Shifted $q$-moments} 
\label{sec:shiftedQmoments}
Theorem \ref{fusedMainTheorem} can be used to obtain an integral representation for the expectation of certain observables introduced in \cite{BW20}. We work in the same setting as before: consider a higher spin vertex model in the positive quadrant with spectral parameters $u_i, y_j$, spin parameters of columns $s_j$ and the incoming boundary condition encoded by $l_c$. Let $\mathbf p=\(\p_i\)_{i=1}^k$ be a collection of points $\p_i=(\alpha_i,\beta_i)\in\ddZ$ such that
\be
0<\alpha_1\leq \alpha_2\leq\dots\leq\alpha_k, \qquad \beta_1\geq \beta_2\geq\dots\geq\beta_k>0,
\ee
and let $\bc=(c_1,\dots,c_k)$ be a tuple of colors, not necessarily monotonically ordered. Define an observable
\be
\O_\bc^{\bp}(\Sigma):=\prod_{i=1}^k\(\exp_q\(h_{>c_i}^{\p_i}(\Sigma)-r_{>c_i}^{>i}[\bc]\)-\exp_q\(h_{\geq c_i}^{\p_i}(\Sigma)-r_{\geq c_i}^{>i}[\bc]\)\),
\ee
where we set
\be
r_{>c}^{>i}[(c_1,\dots, c_k)]:=\#\{j>i | c_j>c\},\qquad r_{\geq c}^{>i}[(c_1,\dots, c_k)]:=\#\{j>i | c_j\geq c\}.
\ee

To formulate our result below, it is more convenient to consider monotonically ordered $\bc$, that is, we assume that $\bc=[1]^{\m_1}\dots[n]^{\m_n}$, where we again use $[x]^a$ to denote $x$ iterated $a$ times. Set $\m[a;b]=\m_a+\dots+\m_b.$ Let $S_{\bc}:=S_{\m_1}\times\dots\times S_{\m_n}\subset S_k$ denote the stabilizer of $\bc$. For a permutation $\pi\in S_k$ let $[\pi]_{\bc}:=\pi\cdot S_{\bc}$ denote the equivalence class of $\pi$ inside $S_k/S_{\bc}$. Note that for any $\pi_1,\pi_2$ such that $[\pi_1]_{\bc}=[\pi_2]_{\bc}$ we have
\be
\O_{\pi_1.\bc}^\bp\equiv\O_{\pi_2.\bc}^\bp.
\ee
For a coset $[\pi]_{\bc}$ define an operator
\be
T_{[\pi]_{\bc}}:=\sum_{\tau\in[\pi]_{\bc}}T_\tau.
\ee

\begin{cor}
\label{qShiftedCor} With the above notation, we have
\begin{multline}
\label{shiftedqmomentseq}
\mathbb E\, \O_{\pi.\bc}^{\bp} =(1-q)^k \int_{\Gamma[1|\bu^{-1}]}\cdots\int_{\Gamma[k|\bu^{-1}]}\prod_{a<b}\frac{w_b-w_a}{w_b-qw_a}\\
T_{[\pi]_\bc}\left(\prod_{c=1}^n\(\sum_{j=0}^{\m_c}\frac{(-1)^jq^{\binom{m_c-j}{2}}}{(q;q)_j(q;q)_{\m_c-j}}  	\ \	\prod_{a>\m{[1;c-1]}}^{\m{[1;c-1]}+j}\ \prod_{i=1}^{l_{c-1}}\frac{1-u_iw_a}{1-qu_iw_a} 	\ \	\prod_{b>\m{[1;c-1]+j}}^{\m{[1;c]}}\ \prod_{i=1}^{l_{c}}\frac{1-u_iw_b}{1-qu_iw_b}\right)\)\\
\prod_{a=1}^k\left(\prod_{i=1}^{i<\beta_a}\frac{1-qu_iw_a}{1-u_iw_a}\prod_{j=1}^{j<\alpha_a}\frac{s_j(w_a s_j-y^{-1}_j)}{w_a-s_jy^{-1}_j}\frac{dw_\alpha}{2\pi i w_\alpha}\right),
\end{multline}
where the integration contours are the same as in Theorem \ref{fusedMainTheorem}.
\end{cor}
\begin{proof}
The main idea is to write $\O_{\pi.\bc}^{\bp}$ as a sum of terms $\exp_q(\H^\bp_{\tilde\pi.\tilde\bc})$, and then apply Theorem \ref{fusedMainTheorem}.

We start with the following observation. Let $\pi^+$ denote the minimal element inside $[\pi]_\bc$, that is,
\begin{equation}
\label{colorOrdered}
\pi^+(\m[1,c-1]+1)<\pi^+(\m[1,c-1]+2)<\dots<\pi^+(\m[1,c])
\end{equation}
for any $c$. Then we have
\begin{multline*}
\sum_i r_{>c_{\pi^{-1}(i)}}^{>i}[\pi.\bc]=\#\{i,j|i<j,\ c_{\pi^{-1}(i)}<c_{\pi^{-1}(j)}\}=\#\{i,j|i<j,\ {\pi^{-1}(i)}<{\pi^{-1}(j)}\} \\
- \#\{i,j|i<j,\ {\pi^{-1}(i)}<{\pi^{-1}(j)},\ c_{\pi^{-1}(i)}=c_{\pi^{-1}(j)}\}=\frac{k(k-1)}{2}-l(\pi^+)-\sum_{c=1}^n\frac{\m_c(\m_c-1)}{2},
\end{multline*}
where for the last equality we replace $\pi$ by $\pi^+$ since the answer depends only on the class $[\pi]_{\bc}$. Factoring out $\exp_q\(-r_{>c_{\pi^{-1}(i)}}^{>i}[\pi.\bc]\)$ from each factor in the definition of $\O_{\pi.\bc}^{\bp}$, we obtain
\be
\O_{\pi.\bc}^{\bp}(\Sigma)=\O_{\pi^+.\bc}^{\bp}(\Sigma)=q^{l(\pi^+)-\frac{k(k-1)}{2}+\sum_{c=1}^n\binom{\m_c}{2}}\prod_{i=1}^k\(\exp_q\(h_{>c_i}^{\p_{\pi^+(i)}}(\Sigma)\)-\exp_q\(h_{\geq c_i}^{\p_{\pi^+(i)}}(\Sigma)-r_{= c_i}^{>\pi^+(i)}[\pi^+\bc]\)\),
\ee
where $r_{=c}^{>i}[\pi\bc]:=\#\{j>i | c_{\pi^{-1}(j)}=c\}$. Using \eqref{colorOrdered} we get
\be
\O_{\pi.\bc}^{\bp}(\Sigma)=q^{l(\pi^+)-\frac{k(k-1)}{2}+\sum_{c=1}^n\binom{\m_c}{2}}\prod_{c=1}^n\prod_{j=1}^{\m_c}\(\exp_q\(h_{>c}^{\p_{\pi^+(\m[1;c-1]+j)}}(\Sigma)\)-\exp_q\(h_{\geq c}^{\p_{\pi^+(\m[1;c-1]+j)}}(\Sigma)-\m_c+j\)\).
\ee
Now we want to use the following identity: let $X_1,\dots X_\m$ and $Y_1,\dots, Y_\m$ be two families of variables. Then
\begin{equation}
\label{qidentity}
\prod_{j=1}^\m(Y_j-q^{j-\m}X_j)=(1-q)^\m\sum_{\tau\in S_\m}q^{l(\tau)-\binom{\m}{2}}\sum_{j=0}^\m\frac{(-1)^j q^{\binom{\m-j}{2}}}{(q;q)_j(q;q)_{\m-j}}\prod_{i=1}^{j}X_{\tau(i)}\prod_{i=j+1}^{\m}Y_{\tau(i)}.
\end{equation}
For now we will show how to finish the proof of Corollary \ref{qShiftedCor}, postponing the proof of \eqref{qidentity}. For any collection of points $\tilde\bp=(\tilde\p_1,\dots,\tilde\p_\m)$  we get
\begin{multline*}
\prod_{j=1}^\m\(\exp_q\(h_{>c}^{\tilde\p_j}(\Sigma)\)-\exp_q\(h_{\geq c}^{\tilde\p_j}(\Sigma)-\m+j\)\)\\
=(1-q)^\m\sum_{\tau\in S_\m}q^{l(\tau)-\binom{\m}{2}}\sum_{j=0}^\m\frac{(-1)^j q^{\binom{\m-j}{2}}}{(q;q)_j(q;q)_{\m-j}}\exp_q\(\H^{\tilde\bp}_{\tau.\([c-1]^j[c]^{\m-j}\)}(\Sigma)\).
\end{multline*}
Plugging $\tilde \p_j=\p_{\pi^+(\m[1;c-1]+j)}$ and multiplying over all $c$ gives
\be
O_{\pi.\bc}^{\bp}(\Sigma)=(1-q)^kq^{l(\pi^+)-\frac{k(k-1)}{2}}\sum_{\tau\in S_\bc}q^{l(\tau)}\sum_{\substack{j_1,\dots, j_n:\\j_c\in[0,\m_c]}}\ \(\prod_{c=1}^n\frac{(-1)^{j_c} q^{\binom{\m_c-j_c}{2}}}{(q;q)_{j_c}(q;q)_{\m_c-j_c}}\cdot\exp_q\(\H^{(\pi^+)^{-1}.\bp}_{\tau.\bc[j_1,\dots,j_n]}(\Sigma)\)\),
\ee
where $\bc[j_1,\dots,j_n]=[0]^{j_1}[1]^{\m_1-j_1+j_2}[2]^{\m_2-j_2+j_3}\dots[n]^{\m_n-j_n}$. Remembering that $\pi^+$ is minimal inside $[\pi]_\bc=\pi S_\bc$, we get the final expression
\be
\E[O_{\pi.\bc}^{\bp}(\Sigma)]=(1-q)^k\sum_{\tau\in [\pi]_\bc}q^{l(\tau)-\frac{k(k-1)}{2}}\sum_{\substack{j_1,\dots, j_n:\\j_c\in[0,\m_c]}}\ \(\prod_{c=1}^n\frac{(-1)^{j_c} q^{\binom{\m_c-j_c}{2}}}{(q;q)_{j_c}(q;q)_{\m_c-j_c}}\cdot\E\[\exp_q\(\H^{\bp}_{\tau.\bc[j_1,\dots,j_n]}(\Sigma)\)\]\).
\ee
The proof is finished by applying Theorem \ref{fusedMainTheorem} to the right-hand side of the expression above.

It only remains to prove \eqref{qidentity}. Recall that 
\be
\sum_{\tau\in S_k}q^{l(\tau)}=\frac{(q;q)_k}{(1-q)^k}. 
\ee
Thus the right-hand side of \eqref{qidentity} is equal to
\begin{multline*}
(1-q)^\m\sum_{j=0}^\m \sum_{\substack{\tau^+\in S_\m:\\ \tau^+(1)<\dots<\tau^+(j)\\\tau^+(j+1)<\dots<\tau^+(\m)}}\sum_{\tau\in S_{j}\times S_{\m-j}}q^{l(\tau^+)+l(\tau)-\binom{\m}{2}}\frac{(-1)^j q^{\binom{\m-j}{2}}}{(q;q)_j(q;q)_{\m-j}}\prod_{i=1}^{j}X_{\tau^+(i)}\prod_{i=j+1}^{\m}Y_{\tau^+(i)}\\
=\sum_{j=0}^\m \sum_{\substack{\tau^+\in S_\m:\\ \tau^+(1)<\dots<\tau^+(j)\\\tau^+(j+1)<\dots<\tau^+(\m)}}(-1)^jq^{l(\tau^+)-\binom{\m}{2}+\binom{\m-j}{2}}\prod_{i=1}^{j}X_{\tau^+(i)}\prod_{i=j+1}^{\m}Y_{\tau^+(i)}.
\end{multline*}
Note that we can explicitly compute $l(\tau^+)$ in the sum above:
\be
l(\tau^+)=\tau^+(1)+\dots+\tau^+(j)-\binom{j+1}{2},
\ee
hence the right-hand side of \eqref{qidentity} is equal to
\be
\sum_{j=0}^\m \sum_{\substack{I\subset\{1,\dots, \m\},\\|I|=j}}(-1)^jq^{- j\m+\sum_{i\in I}i}\prod_{i\in I}X_{i}\prod_{i\notin I}Y_{i}=\sum_{\substack{I\subset\{1,\dots, \m\}}} \prod_{i\in I}\(-q^{i-m}X_{i}\)\prod_{i\notin I}Y_{i}=\prod_{j=1}^\m(Y_j-q^{j-\m}X_j).
\ee
\end{proof}
\begin{rem}
When all the points $\p_i$ are on the same row, i.e. all $\beta_i$ are equal, the observables $\O_{\pi.\c}^\bp$ coincide, up to a scaling, with observables $\O_\mu$ from \cite{BW20}, where $\mu$ corresponds to the composition encoding particles of colors $c_{\pi^{-1}(i)}$ at points $\p_i$. In this case Corollary \ref{qShiftedCor} is an inhomogeneous generalization of \cite[Theorem 6.1]{BW20} (``inhomogeneous" in the sense that Corollary \ref{qShiftedCor} allows distinct parameters $y_j$ and $s_j$, in contrast to \cite[Theorem 6.1]{BW20}).\footnote{The functions $f$ used in \cite{BW20} can be expressed using the operators $T_\pi$, see \cite[Remark 6.8]{BW20}.}
\end{rem}

\subsection{Horizontal fusion}
\label{subsec:hor-fusion}

Our next goal is to make a horizontal fusion. We will consider one specific model obtained by such a fusion, which is of particular importance due to its relation to last passage percolation and polymer models. Our discussion closely follows \cite[Section 6.5]{BW20}, see also \cite{BGW19}.

In more detail, we will be interested in the following vertex model in the quadrant. Both vertical and horizontal edges are allowed to contain arbitrary amount of arrows. As boundary conditions, all vertical arrows entering the quadrant from below are empty. At the left boundary, we have a random number of paths of one color entering at every row. These random numbers are independent and identically distributed for every row, and, as before, the arrows of color $c$ enter at rows $l_{c-1}+1,\dots, l_c$. The distribution of the number of paths that enter in a row is given by
\begin{equation}
\label{eq:qhahnBoundaryExplicit}
\mathrm{Prob} (k) = \frac{(s^2/z^2;q)_{\infty}}{(s^2;q)_{\infty}} \frac{(z^2;q)_k}{(q;q)_k} \left( \frac{s^2}{z^2} \right)^k, \qquad k=0,1,2, \dots,
\end{equation}
where $s$ and $z$ are fixed parameters of the model. The vertices inside the quadrant have the weights $W^{qH}_{s,z}$ described by \eqref{eq:qHahExplicit}.

The observables of this vertex model are given by the following theorem.

\begin{theo}
\label{th:qHahnFormulae}
Fix $k \in \mathbb{N}$. For any $k$ tuples $(\alpha_1, \dots, \alpha_k), (\beta_1, \dots, \beta_k), (c_1, \dots, c_k)$ satisfying
\be
0<\alpha_1\leq \alpha_2\leq \dots\leq \alpha_k, \quad \beta_1\geq \beta_2\geq \dots\geq \beta_k>0,\qquad   \alpha_i, \beta_j\in\Z+\frac{1}{2},
\ee
\be
0\leq c_1\leq c_2\leq \dots\leq c_k, \qquad c_i\in\Z,
\ee
together with an additional assumption $\beta_k > l_{c_k}$, and any permutation $\pi\in S_k$ we have
\begin{multline}
\label{fullyfusedqmomentseq}
\mathbb E\left(\exp_q\({h_{>c_1}^{(\alpha_{\pi(1)},\beta_{\pi(1)})}+h_{>c_2}^{(\alpha_{\pi(2)},\beta_{\pi(2)})}+\dots+
h_{>c_k}^{(\alpha_{\pi(k)},\beta_{\pi(k)})}}\)\right) \\ = q^{\frac{k(k-1)}{2}-l(\pi)} \int_{\Gamma[1|s^{-1}]}\cdots\int_{\Gamma[k|s^{-1}]}  \prod_{a<b}\frac{w_b-w_a}{w_b-qw_a}
T_\pi\left( \prod_{a=1}^k \left( \frac{1-sw_a}{1-z^{-2} s w_a} \right)^{l_{c_a}} \right) \\ \times \prod_{a=1}^k\left( \left( \frac{1-z^{-2} s w_a}{1-s w_a} \right)^{\beta_a-\frac12} \left( \frac{s(w_a s-1)}{w_a-s} \right)^{\alpha_a-\frac12} \frac{s dw_a}{2\pi \mathbf{i} w_a (s-w_a)}\right).
\end{multline}
where the integral with respect to $w_a$ is taken along $\Gamma[a|s^{-1}]$ and contours $\Gamma[a|s^{-1}]$ are $q$-nested around $0$ and $\{s^{-1}\}$, and do not encircle $s$.
\end{theo}

\begin{proof}

The claim can be obtained from \eqref{qmomentseq} by applying the following steps. It was shown in \cite{BW18}, \cite[Section 6.5]{BW20} that these steps transform horizontally unfused vertex model to a vertex model described above. Let us look how these steps affect the integral formula.

First, let us notice that our new assumption $\beta_k > l_{c_k}$ guarantees that there are no poles of the form $q^{-1} u_i^{-1}$ in the integrand in \eqref{qmomentseq}: Indeed, all such potential poles coming from $T_{\pi} \left( \cdot \right)$ cancel out with terms from numerators. Therefore, by continuity \eqref{qmomentseq} holds even if $u_i= q u_j$ for some pairs of $i$ and $j$. Now, let us split rows into groups of neighboring rows of size $L>0$, and set the row rapidities inside each group to be equal to $u$, $q u$, ..., $q^{L-1} u$ in the order from bottom to top. The right-hand side of \eqref{qmomentseq} will transform in a simple way
$$
\frac{1-uw_a}{1-q u w_a} \mapsto \frac{1-u w_a}{1 - q^L u w_a}, \qquad  \frac{1-q uw_a}{1- u w_a} \mapsto \frac{1-q^L u w_a}{1 - u w_a}.
$$
as a result of this step.

Second, we make a limit transition in the first column of our vertex model by setting $y_1:=s_1/s$ and considering the limit $s_1 \to 0$. Then
$$
\frac{s_1(w_a s_1 - s_1^{-1} s )}{w_a-s_1 s_1^{-1} s} \xrightarrow[s_1 \to 0]{} \frac{s}{s-w_a}.
$$
This limit is needed to get the boundary conditions of the form \eqref{eq:qhahnBoundaryExplicit}, which makes the subsequent analytic continuation possible, see \cite[Section 6.5]{BW20}.

Third, let us do an analytic continuation in $q^L =: z^{-2}$. In order to apply it, we need to notice that the integral expression depends on $L$ only through $q^L$ for any $L=1,2,3, \dots$, and also to make contours of integration $q$-nested (otherwise the term $\prod_{a<b} \frac{w_b-w_a}{w_b-qw_a}$ will contribute extra singularities).

Fourth, we set $s_j=s$ (for $j\ge 2$), $u=s$, and $y_j=1$ (for $j\ge 2$). This degeneration leads to the vertex weights $W^{qH}_{s,z}$.

Combining all these steps, we arrive at the statement of the theorem.
\end{proof}

\begin{rem}
In fact, the assumption $\beta_k > l_{c_k}$ in Theorem \ref{th:qHahnFormulae} can be weakened; however, the proof of corresponding formula becomes more involved due to the fact that points $q^{-1} u_i^{-1}$ might become poles of the integrand, and we omit it.
\end{rem}

\section{Beta polymer}

This section closely follows \cite[Sections 7.1-7.2]{BW20}, see also \cite{BGW19}.

Let $\sigma > \rho >0$ be reals. Let $\{ \eta_{t,m} \}_{t,m \in \Z_{\ge 0}}$ be a sequence of independent identically distributed beta random variables with parameters $(\sigma - \rho,\rho)$. Recall that this means that the density of a variable is given by
$$
B(\sigma-\rho,\rho)^{-1} x^{\sigma-\rho-1} (1-x)^{\rho-1}, \qquad x \in [0;1].
$$

Define a \textit{partition function} $Z^{(m,t)}$, $m,t \in \Z_{\ge 1}$, $m \le t$, of the Beta-polymer via recursion
$$
Z^{(m,t)} = \eta_{m,t} Z^{(m,t-1)} + (1- \eta_{m,t}) Z^{(m-1,t-1)},
$$
and boundary conditions
$$
Z^{(t,t)}=1, \qquad Z^{(1,t)} = \prod_{i=2}^t \eta_{1,i}, \qquad t \in \Z_{\ge 1}.
$$
The partition function has an interpretation as the sum over all directed lattice paths with (0,1) and (1,1) steps from $(1,1)$ to $(m,t)$ of products of edge weights that all have form $\eta_{*}$ or $1-\eta_{*}$, see \cite{BC17}, \cite{BGW19}.

Analogously, let us define \textit{delayed} partition functions $Z_{(k)}^{(m,t)}$, $t \ge m+k$, via the same recursion
$$
Z_{(k)}^{(m,t)} = \eta_{m,t} Z_{(k)}^{(m,t-1)} + (1- \eta_{m,t}) Z_{(k)}^{(m-1,t-1)},
$$
and boundary conditions
$$
Z_{(k)}^{(t-k,t)}=1, \qquad Z_{(k)}^{(1,t)} = \prod_{i=k+2}^t \eta_{1,i}, \qquad t \in \Z_{\ge k+1}.
$$
Obviously, one has $Z_{(0)}^{(m,t)} = Z^{(m,t)}$. The quantity $Z_{(k)}^{(m,t)}$ has a similar interpretation as the sum over directed paths which join $(1,k+1)$ and $(m,t)$.

The beta-polymer can be obtained from the six vertex model via the following proposition (\cite[Corollory 6.22]{BGW19}):

\begin{prop}
\label{prop:6v-polymer}
Consider the horizontally fused colored vertex model described in Section \ref{subsec:hor-fusion}, and set its parameters via
$$
q = \exp( - \epsilon), \qquad s^2 = q^{\sigma}, \qquad z^2 = q^{\rho}, \qquad l_i=i, \ \ i=1,2,\dots.
$$
Let $\{ h_{>c}^{(m,t)} (\epsilon) \}$ be its height functions.
Then, in the limit $\epsilon \to 0$, the random vector $\left\{ \exp \left( - \epsilon h_{>c}^{(m,t)} (\epsilon) \right) \right\}$ converges to $\left\{ Z_{(c)}^{(m,t)} \right\}$ for all $c \ge 0$ and $t \ge m \ge 1$.
\end{prop}

We can use this limit transition in order to obtain formulas for joint distribution of moments of (delayed) partition functions of Beta-polymer. In order to formulate the result, let us introduce another piece of notation. Let $\tilde t_{i}$ be an operator which acts on the space of rational functions in $\tilde w_1, \tilde w_2, \dots, \tilde w_k$ by swapping $\tilde w_i$ and $\tilde w_{i+1}$. Define
$$
\tilde T_i :=1+ \frac{\tilde w_{i+1} - \tilde w_i+1}{\tilde w_{i+1}-\tilde w_i} \left( \tilde t_i -1 \right).
$$
Analogously, for $\pi \in S_k$ with a shortest decomposition $\pi = \sigma_{i_1} \sigma_{i_2} \dots \sigma_{i_{l(\pi)}}$ into a product of nearest neighbor transpositions, we define
$$
\tilde T_{\pi} := \tilde T_{i_1} \tilde T_{i_2} \dots \tilde T_{i_{l(\pi)}}.
$$

\begin{theo}
\label{th:BetaPol}
Fix $k \in \mathbb{N}$. For any $k$ tuples $(\alpha_1, \dots, \alpha_k), (\beta_1, \dots, \beta_k), (c_1, \dots, c_k)$ satisfying
\be
0<\alpha_1\leq \alpha_2\leq \dots\leq \alpha_k, \quad  \beta_1\geq \beta_2\geq \dots\geq \beta_k>c_k,\qquad  \alpha_i,\beta_i\in\Z+\frac{1}{2}
\ee
\be
0\leq c_1\leq c_2\leq \dots\leq c_k< \beta_k,\qquad   c_i\in\Z
\ee
and any permutation $\pi\in S_k$
satisfying $\alpha_{\pi(i)} + c_i \le \beta_{\pi(i)}$, for $i=1, \dots, k$,
we have

\begin{multline}
\label{momentBetaPol}
\mathbb E \left( Z_{(c_1)}^{\left( \alpha_{\pi(1)},\beta_{\pi(1)} \right)} Z_{(c_2)}^{\left( \alpha_{\pi(2)},\beta_{\pi(2)} \right)} \cdot \dots \cdot Z_{(c_k)}^{\left( \alpha_{\pi(k)},\beta_{\pi(k)} \right)} \right) = \int_{\tilde \Gamma[1]}\cdots\int_{ \tilde \Gamma[k]}
\prod_{a<b} \frac{\tilde w_b-\tilde w_a}{\tilde w_b- \tilde w_a+1} \\
\times \tilde T_\pi\left( \prod_{a=1}^k \left( \frac{\tilde w_a - \sigma/2}{\tilde w_a - \sigma/2 + \rho} \right)^{c_a} \right) \prod_{a=1}^k \left( \left( \frac{\tilde w_a - \sigma/2 + \rho}{\tilde w_a - \sigma/2} \right)^{\beta_a-\frac12} \left(\frac{\tilde w_a - \sigma/2}{\tilde w_a+\sigma/2} \right)^{\alpha_a-\frac12}  \frac{d \tilde w_a}{2\pi \mathbf{i} (\tilde w_a+\sigma/2)}\right).
\end{multline}
where contours $\tilde \Gamma[i]$, $i=1, \dots, k$, encircle $-\sigma/2$ and (in the case $i>1$) the shift by -1 of $\tilde \Gamma[i-1]$, but do not encircle $\sigma/2$.
\end{theo}

\begin{proof}
We make a limit transition prescribed by Proposition \ref{prop:6v-polymer} in the integral formula from Theorem \ref{th:qHahnFormulae}. In the left-hand side of the formula as $\epsilon \to 0$ we have asymptotic equivalence
$$
q^{h^*_*} \approx \exp \left( - \epsilon \frac{ - \log Z_*^*}{\epsilon} \right) = Z_*^*.
$$
In the right-hand side, we first notice that the integrand does not have a residue at infinity. Thus, we can move contours through infinity and have them nested around $s$ only. Then we make a change of variables $w_i = 1+ \epsilon \tilde w_i$, $i=1,2, \dots, k$, and do
the limit transition in parameters as in Proposition \ref{prop:6v-polymer} (using the same $\epsilon$). Note that
if $f(w_1, w_2, \dots, w_k)$ converges to $\tilde f(\tilde w_1, \tilde w_2, \dots, \tilde w_k)$ in the limit $\epsilon \to 0$, then we have the convergence
$$
T_{\pi} f(w_1, w_2, \dots, w_k)  \to \tilde T_{\pi} \tilde f(\tilde w_1, \tilde w_2, \dots, \tilde w_k), \qquad \mbox{as $\epsilon \to 0$.}
$$
The rest is straightforward computations of limits of factors entering the integrand.
\end{proof}


\begin{thebibliography}{9999999}
\bibitem[AAV08]{AAV08} G.~Amir, O.~Angel, B.~Valko, \textit{The TASEP speed process}, Ann. Probab.
 39 (2011), 1205--1242, {\tt arXiv:0811.3706}.
 \bibitem[BC15]{BC17} G. Barraquand, I. Corwin, \emph{Random walk in Beta-distributed random environment}, Probability Theory and Related Fields, Volume 167, Issue 3, pp 1057-1116, {\tt arXiv:1503.04117}.
 \bibitem[BB19]{BB19} A.~Borodin, A.~Bufetov, \textit{Color-position symmetry in interacting particle systems}, preprint, {\tt arXiv:1905.04692}.
\bibitem[BCG14]{BCG16} A.~Borodin, I.~Corwin, V.~Gorin, \textit{Stochastic six-vertex model}, Duke Mathematical Journal, 165 (2016), 563--624, {\tt arXiv:1407.6729}.
\bibitem[BCPS13]{BCPS15b} A. Borodin, I. Corwin, L. Petrov, and T. Sasamoto. \emph{Spectral theory for the q-Boson particle system}, Compositio Mathematica, 151(1):1--67, 2015, {\tt arXiv:1308.3475}.
\bibitem[BCPS14]{BCPS15a} A. Borodin, I. Corwin, L. Petrov, and T. Sasamoto. \emph{Spectral theory for interacting particle systems solvable by coordinate Bethe ansatz}, Communications in Mathematical Physics, 339(3):1167--1245, 2015, {\tt arXiv:1407.8534}.
\bibitem[BG18]{BG18} A.~Borodin, V.~Gorin, \emph{A stochastic telegraph equation from the six-vertex model}, to appear in Ann. Probab., {\tt arXiv:1803.09137}.
\bibitem[BGW19]{BGW19} A. Borodin, V. Gorin, M. Wheeler. \emph{Shift-invariance for vertex models and polymers}, {\tt arXiv:1912.02957}.
\bibitem[BP16]{BP16} A. Borodin, L. Petrov. \emph{Higher spin six vertex model and symmetric rational functions}, Selecta Mathematica vol. 24, p. 751–874, 2018, {\tt arXiv:1601.05770}.
\bibitem[BW18]{BW18} A. Borodin, M. Wheeler. \emph{Coloured stochastic vertex models and their spectral theory}, {\tt arXiv:1808.01866}.
\bibitem[BW20]{BW20} A. Borodin, M. Wheeler. \emph{Observables of coloured stochastic vertex models and their polymer limits}, {\tt arXiv:2001.04913}.
\bibitem[BM16]{BM16}  G. Bosnjak, V. Mangazeev. \emph{Construction of $R$-matrices for symmetric tensor representations related to $U_q(\widehat{sl_n})$}, J. Phys. A: Math. Theor. 49, 2016, {\tt arXiv:1607.07968}.
\bibitem[Buf20]{Buf20} A. Bufetov, \emph{Interacting particle systems and random walks on Hecke algebras}, preprint, {\tt arXiv:2003.02730}.
\bibitem[BM17]{BM18} A. Bufetov, K. Matveev, \textit{Hall-Littlewood RSK field}, Selecta Mathematica 24 (2018), 4839--4884, {\tt arXiv:1705.07169}.
\bibitem[CP15]{CP16} I. Corwin and L. Petrov. \emph{Stochastic higher spin vertex models on the line}, Communications in Mathematical Physics, 343(2):651--700, 2016, {\tt arXiv:1502.07374}.
\bibitem[D20]{D20} D. Dauvergne, \emph{Hidden invariance of last passage percolation and directed polymers}, preprint, {\tt arXiv:2002.09459}.
\bibitem[Gal20]{Gal20} P. Galashin. \emph{Symmetries of stochastic colored vertex models}, preprint, {\tt arXiv:2003.06330}.
\bibitem[GS92]{GS92} L.-H. Gwa, H. Spohn, \textit{Six-vertex model, roughened surfaces, and an asymmetric spin Hamiltonian}, Physical review letters 68 (1992), 725--728.
\bibitem[Jim86]{Jim86} M. Jimbo. \emph{Quantum R matrix for the generalized Toda system}, Communications in Mathematical Physics, 102(4):537–547, 1986.
\bibitem[K20]{K20} J. Kuan, \emph{Coxeter Group Actions on Interacting Particle Systems}, preprint, {\tt arXiv:2003.03342}.
\bibitem[KMMO16]{KMMO16}A. Kuniba, V. Mangazeev, S. Maruyama, and M. Okado. \emph{Stochastic $R$ matrix for $U_q(\widehat{sl_n})$}, NuclearPhysics B, 913:248–277, 2016, {\tt arXiv:1604.08304}.
\bibitem[KR83]{KR83} P. Kulish and N. Reshetikhin. \emph{Quantum linear problem for the sine-Gordon equation and higher representations}. {N.Y. J Math Sci}, 23, 2435–2441, 1983.
\bibitem[KRS81]{KRS81} P. Kulish, N. Reshetikhin, and E. Sklyanin. \emph{Yang-Baxter equation and representation theory: I.} {Letters in Mathematical Physics 5.5}, 393–403, 1981.
\bibitem[TW07]{TW08a} C. Tracy and H. Widom. \emph{Integral formulas for the asymmetric simple exclusion process}, Communications in Mathematical Physics, 279(3):815--844, 2008, {\tt arXiv:0704.2633}.
\bibitem[TW08a]{TW08b} C. Tracy and H. Widom. \emph{A Fredholm determinant representation in ASEP}, Journal of Statistical Physics, 132(2):291--300, 2008, {\tt arXiv:0804.1379}.
\bibitem[TW08b]{TW09} C. Tracy and H. Widom. \emph{Asymptotics in ASEP with step initial condition}, Communications in Mathematical Physics, 290(1):129--154, 2009, {\tt arXiv:0807.1713}. 
\end{thebibliography}
\end{document}